\documentclass{amsart}
\usepackage[T1]{fontenc}
\usepackage[utf8]{inputenc}
\usepackage[french,english]{babel}

\newcommand{\newabstract}[1]{%
  \par\bigskip
  \csname otherlanguage*\endcsname{#1}%
  \csname captions#1\endcsname
  \item[\hskip\labelsep\scshape\abstractname.]
}
 \usepackage[all]{xy}

\usepackage[numbers]{natbib}
\usepackage[colorlinks]{hyperref}

\usepackage{amsmath}
\usepackage{amssymb}
\usepackage{graphicx}

\usepackage{amssymb}

\usepackage{amsfonts}

\usepackage{soul}

\usepackage{mathpazo}
\usepackage{color}

\usepackage{paralist}

\usepackage{stmaryrd}

\usepackage{amsxtra}

\def\id{{\rm Id}}

\newcommand{\cR}{\mathcal R}

\def\B{          \mathcal B}

\def\W{          \mathcal W}

\def \R{{\mathbb R}}
\def \la{{\lambda}}
\def \Z{{\mathbb Z}}
\def \N{{\mathbb N}}

\newcommand{\prf}{{\begin{proof}}}
\newcommand{\epf}{{\end{proof}}}

\newtheorem{theo}{\sc Theorem}
\newtheorem{prop}[theo]{\sc Proposition}
\newtheorem{lemma}[theo]{\sc Lemma}
\newtheorem{cor}[theo]{\sc Corollary}

\theoremstyle{definition}

\def\bee{\begin{equation}}
\def\eee{\end{equation}}

\newtheorem{defi}[theo]{\sc Definition}

\theoremstyle{rema}

\newtheorem{prob}[theo]{\sc Problem}

\newcommand{\RR}{{\mathbb R}}
\newcommand{\TT}{{\mathbb T}}
\newcommand{\ZZ}{{\mathbb Z}}

\numberwithin{equation}{section}
\def \E{{\mathcal{E}}}

\author[C.~Butler]{Clark Butler}\address{\noindent Department of Mathematics, University of Chicago,  5734 S University Ave, Chicago, IL 60637
  \newline e-mail: \rm
  \texttt{cbutler@math.uchicago.edu}}
  
  \author[D.~Xu]{Disheng Xu}\address{\noindent Univ Paris Diderot, Sorbonne Paris Cit\'e, Institut de Math\'ematiques de Jussieu-Paris Rive Gauche, UMR 7586, CNRS, Sorbonne Universit\'es, UPMC Univ Paris 06, F-75013, Paris, France
  \newline e-mail: \rm
  \texttt{disheng.xu@imj-prg.fr}}

\begin{document}
\title{Uniformly quasiconformal partially hyperbolic systems}
\begin{abstract}
We study smooth volume-preserving perturbations of the time-1 map of the geodesic flow $\psi_{t}$ of a closed Riemannian manifold of dimension at least three with constant negative curvature. We show that such a perturbation has equal extremal Lyapunov exponents with respect to volume within both the stable and unstable bundles if and only if it embeds as the time-1 map of a smooth volume-preserving flow that is smoothly orbit equivalent to $\psi_{t}$. Our techniques apply more generally to give an essentially complete classification of smooth, volume-preserving partially hyperbolic diffeomorphisms which satisfy a uniform quasiconformality condition on their stable and unstable bundles and have either compact center foliation with trivial holonomy or are obtained as perturbations of the time-1 map of an Anosov flow. 
\end{abstract}
\maketitle
\date{\today}

\section{Introduction}\label{section: introduction}
A surprising number of rigidity problems originally posed in negatively curved geometry turn out to have solutions that are dynamical in nature. We review one such story here: Sullivan proposed, following work of Gromov\cite{Gr} and Tukia\cite{Tu}, that closed Riemannian manifolds of constant negative curvature and dimension at least 3 should be characterized up to isometry by the property that the geodesic flow acts \emph{uniformly quasiconformally} on the unstable foliation\cite{Su}.  Informally, the uniform quasiconformality property states that the flow does not distort the shape of metric balls inside of a given horosphere over a long period of time. Sullivan's conjecture was partially confirmed by the work of Kanai \cite{Kan1} who showed that among contact Anosov flows the geodesic flows of constant negative curvature manifolds are characterized up to $C^{1}$ orbit equivalence by a uniform quasiconformailty. Later the minimal entropy rigidity theorem of Besson, Courtois, and Gallot \cite{BCG1} completed the proof of Sullivan's conjecture among many other outstanding conjectures in negatively curved geometry.

From a geometric perspective this completes the story, but from a dynamical perspective this raises many new questions. Already in the work of Kanai we see that the dynamical version of this rigidity result holds for a larger class of Anosov flows than just geodesic flows. Sadovskaya initiated a program to extend these results further to smooth volume-preserving Anosov flows and diffeomorphisms \cite{S05}, which was completed in a series of works by Fang (\cite{F04}, \cite{F07}, \cite{F14}) who obtained the following remarkable result:  all smooth volume-preserving Anosov flows which are uniformly quasiconformal on the stable and unstable foliation are smoothly orbit equivalent either to the suspension of a hyperbolic toral automorphism or the geodesic flow on the unit tangent bundle of a constant negative curvature closed Riemannian manifold. Thus we see that not even the contact structure of the flow is necessary to obtain dynamical rigidity for uniformly quasiconformal Anosov flows. 

In a different direction one can ask whether the uniform quasiconformality condition can be relaxed to a condition that is more natural from the perspective of ergodic theory. This direction was pursued by the first author, who showed that for geodesic flows of $\frac{1}{4}$-pinched negatively curved manifolds, uniform quasiconformality can be derived from the significantly weaker dynamical condition of equality of all Lyapunov exponents with respect to volume on the unstable bundle \cite{Bu15}. 

Our principal goal is to show that for all of the rigidity phenomena derived from uniform quasiconformality above, \emph{not even the structure of an Anosov flow is necessary}. Let us be more precise: consider a closed Riemannian manifold $X$ of constant negative curvature with $\dim X \geq 3$. Let $T^{1}X$ be the unit tangent bundle of $X$ and let $\psi_{t}: T^{1}X \rightarrow T^{1}X$ denote the time-$t$ map of the geodesic flow. This flow preserves a smooth volume $m$ on $T^{1}X$ known as the Liouville measure. Consider \emph{any} smooth diffeomorphism $f$ which is $C^{1}$-close to the time-$1$ map $\psi_{1}$ and which preserves the volume $m$. By the work of Hirsch, Pugh, and Shub \cite{HPS}, $f$ is \emph{partially hyperbolic}, meaning that there is a $Df$-invariant splitting $T(T^{1}X) = E^{u} \oplus E^{c} \oplus E^{s}$ where $E^{u}$ is exponentially expanded by $Df$, $E^{s}$ is exponentially contracted by $Df$, and the behavior of $Df$ on the 1-dimensional center direction $E^{c}$ (which is close to the flow direction for $\psi_{t}$) is dominated by the expansion and contraction on $E^{u}$ and $E^{s}$ respectively. We give a more precise definition in Section \ref{sec:statements}. We then choose a continuous norm $\| \cdot \|$ on $E^{u}$ and define the \emph{extremal Lyapunov exponents} of $f$ on $E^{u}$ by 
\[
\lambda_{+}^{u}(f) = \inf_{n \geq 1} \frac{1}{n}\int_{M} \log \|Df^{n}|E^{u}\| \,dm,
\]
\[
\lambda_{-}^{u}(f) = \sup_{n \geq 1}\frac{1}{n} \int_{M} \log \|(Df^{n}|E^{u})^{-1}\|^{-1} \,dm.
\]
\color{black}
We define $\lambda_{+}^{s}(f)$ and $\lambda_{-}^{s}(f)$ similarly with $E^{s}$ replacing $E^{u}$.

\begin{theo}\label{theorem: hyperbolic perturbation}
There is a $C^{2}$-open neighborhood $\mathcal{U}$ of $\psi_{1}$ in the space of $C^{\infty}$ volume-preserving diffeomorphisms of $T^{1}X$ such that if $f \in \mathcal{U}$ and both of the equalities $\la_{+}^{u}(f) = \la_{-}^{u}(f)$ and $\la_{+}^{s}(f) = \la_{-}^{s}(f)$  hold then there is a $C^{\infty}$ volume-preserving flow $\varphi_{t}$ with $\varphi_{1} = f$. Furthermore $\varphi_{t}$ is smoothly orbit equivalent to $\psi_{t}$. 
\end{theo}

This theorem improves on the techniques used in the previous rigidity theorems in several fundamental ways. We are able to deduce uniform quasiconformality of the action of $Df$ on $E^{u}$ and $E^{s}$ from equality of the extremal Lyapunov exponents entirely outside of the geometric context considered in \cite{Bu15} by using new methods. We then use this uniform quasiconformality to completely reconstruct the smooth flow $\varphi_{t}$ in which $f$ embeds as the time-$1$ map. We emphasize that for a typical perturbation $f$ of $\psi_{1}$ the foliation $\W^{c}$ tangent to $E^{c}$ (which is our candidate for the flowlines of $\varphi_{t}$) is only a continuous foliation of $T^{1}X$ with no transverse smoothness properties. This is one of the many reasons that strong rigidity results in the realm of partially hyperbolic diffeomorphisms are quite rare. Our inspiration was an impressive rigidity theorem of Avila, Viana and Wilkinson which overcame this obstacle to show that if we take $X$ to be a negatively curved surface instead and $f$ a $C^{1}$-small enough $C^{\infty}$ volume-preserving perturbation of the time-1 map $\psi_{1}$ such that the center foliation of $f$ is absolutely continuous, then $f$ is also the time-1 map of a smooth volume-preserving flow \cite{AVW}. Our result can be viewed in an appropriate sense as the higher dimensional analogue of this theorem. 

We now explain the organization of the paper. The techniques used in the proof of Theorem \ref{theorem: hyperbolic perturbation} have much more general applications which can also be applied to the study of $C^{\infty}$ volume-preserving partially hyperbolic diffeomorphisms  which satisfy a uniform quasiconformality condition on their stable and unstable bundles and either have uniformly compact center foliation with trivial holonomy or are obtained as a perturbation of the time-1 map of an Anosov flow. These results are stated in Theorems \ref{theorem: compact center} and \ref{theorem: flow suspension} and Corollary \ref{theorem: isometric center} of Section \ref{sec:statements} after we introduce some necessary terminology. In Section \ref{section: center holonomy} we show that under a Lyapunov stability type result on the action of a partially hyperbolic diffeomorphism $f$ on its center foliation, uniform quasiconformality implies that the holonomy maps of the center stable and center unstable foliations of $f$ are quasiconformal.  We use this to show that the center foliation of $f$ is absolutely continuous. In Section \ref{section: higher} we prove the linearity of center holonomy for $f$ between local unstable leaves in a suitable chart under some stronger assumptions on $f$; moreover, for such $f$   the center, center (un)stable foliations are all smooth, see section \ref{section: higherend}. In Section \ref{subsec:proofs} we finish the proofs of Theorems \ref{theorem: compact center} and \ref{theorem: flow suspension} and Corollary \ref{theorem: isometric center}. In Section \ref{section:theorem4} we finish the proof of Theorem \ref{theorem: hyperbolic perturbation} by deducing uniform quasiconformality from the condition of equality of extremal Lyapunov exponents. The arguments in Section \ref{section:theorem4} do not rely on the results of Sections \ref{section: center holonomy}, \ref{section: higher}, \ref{section: higherend} and \ref{subsec:proofs} and may be read independently of the rest of the paper. \color{black}

\textbf{Acknowledgments}: We thank Amie Wilkinson for numerous useful discussions regarding the content of this paper. These discussions resulted in significant simplification of the proof of Theorem \ref{theorem: hyperbolic perturbation}  and some mistakes in the first version of this paper.  We thank Danijela Damjanovic for her carefully reading the paper and pointing out that an assumption in the first version of the paper is not necessary. We also thank Sylvain Crovisier for bringing an error in the concluding arguments of Theorem \ref{theorem: flow suspension} to our attention. We would like to thank the anonymous reviewers for carefully reading the first version of our paper and many helpful comments and suggestions. The second author would like to thank his directors of thesis Professor Artur Avila and Julie D\'eserti for their supervision and encouragement. \color{black} This work was partially completed while both authors were visiting the Instituto Nacional de Matem\'atica Pura e Aplicada, the second author being supported by  \emph{r\'eseau franco-br\'esilien en math\'ematiques}. The first author was supported by the National Science Foundation Graduate Research Fellowship under Grant \# DGE-1144082.

\section{Statement of Results}\label{sec:statements}
A $C^{1}$ diffeomorphism $f: M \rightarrow M$ of a closed Riemannian manifold $M$ is \emph{partially hyperbolic} if there is a $Df$-invariant splitting $TM = E^{s} \oplus E^{c} \oplus E^{u}$ of the tangent bundle of $M$ such that for some $k \geq 1$, any $x \in M$, and any choice of unit vectors $v^{s} \in E^{s}_{x}$, $v^{c} \in E^{c}_{x}$, $v^{u} \in E^{u}_{x}$,
\[
\|Df^{k}(v^{s})\| < 1< \|Df^{k}(v^{u})\|,
\]
\[
\|Df^{k}(v^{s})\| < \|Df^{k}(v^{c})\| < \|Df^{k}(v^{u})\|.
\]
By modifying the Riemannian metric on $M$ if necessary we can always assume $k = 1$ in the above definition. We will always require that the bundles $E^{s}$ and $E^{u}$ are nontrivial. We will also always require that $M$ is connected. We define for $x \in M$, $n \in \mathbb{Z}$, 
\[
K^{u}(x,n) = \frac{\sup \{\|Df^{n}(v^{u})\|: v^{u} \in E^{u}(x), \; \|v^{u}\| = 1\}}{\inf \{\|Df^{n}(v^{u})\|: v^{s} \in E^{u}(x), \; \|v^{u}\| = 1\}},
\]
and define $K^{s}(x,n)$ similarly with $E^{u}$ replaced by $E^{s}$. The quantities $K^{u}$ and $K^{s}$ measure the failure of the iterates of $Df$ to be conformal on the bundles $E^{u}$ and $E^{s}$ respectively. We say that $f$ is \emph{uniformly $u$-quasiconformal} if $\dim E^{u} \geq 2$ and $K^{u}$ is uniformly bounded in $x$ and $n$. Similarly we say that $f$ is \emph{uniformly $s$-quasiconformal} if $\dim E^{s} \geq 2$ and $K^{s}$ is uniformly bounded in $x$ and $n$. If $f$ is both uniformly $u$-quasiconformal and $s$-quasiconformal then we simply say that $f$ is \emph{uniformly quasiconformal}. 

Our definition of uniform quasiconformality for partially hyperbolic systems extends previous definitions of uniform quasiconformality which were considered for Anosov diffeomorphisms and Anosov flows. If the center bundle $E^{c}$ is trivial or if $f$ embeds as the time-1 map of an Anosov flow (so that $E^{c}$ is tangent to the flow direction) then these definitions reduce to the standard notions of uniform quasiconformality for Anosov systems defined by Sadovskaya \cite{S05}. If $\dim E^{u} = 1$ then $K^{u} \equiv 1$ for any choice partially hyperbolic $f$, so the boundedness of $K^{u}$ does not give new information about $f$. This is the reason we require $\dim E^{u} \geq 2$ in the definition of uniform $u$-quasiconformality; the uniform quasiconformality conditions are only interesting when the bundles in question have dimension at least 2. 

We define a $C^{\infty}$ diffeomorphism $f$ to be \emph{volume-preserving} if there is an $f$-invariant probability measure $m$ on $M$ which is smoothly equivalent to the Riemannian volume. It is not hard to show using Kingman's subadditive ergodic theorem \cite{K68} that when $f$ is ergodic with respect to $m$ we have 
\[
\lim_{n \rightarrow \infty} \frac{\log K^{u}(x,n)}{n}  = \la_{+}^{u}(f) - \la_{-}^{u}(f) \; \; \text{for $m$-a.e. $x \in M$}.
\]
We refer to \cite{KS13} for more details on this equality. Thus asymptotic subexponential growth of $K^{u}$ is equivalent to the equality $\la_{+}^{u}(f) = \la_{-}^{u}(f)$. Theorem \ref{theorem: hyperbolic perturbation} asks in part for the deduction of a uniform bound $K^{u}(x,n) \leq C$ from this asymptotic subexponential growth condition. 

Fang proved that all volume-preserving $C^{\infty}$ uniformly quasiconformal diffeomorphisms are $C^{\infty}$ conjugate to a hyperbolic toral automorphism \cite{F04}. This generalized the classification result of Sadovskaya which held under the additional assumption that $f$ was symplectic \cite{S05}. Theorem \ref{theorem: compact center} and Corollary \ref{theorem: isometric center} below extend this classification to cover a certain $C^{1}$-open set of $C^{\infty}$ volume-preserving partially hyperbolic diffeomorphisms. Before stating these theorems we need to introduce a few more basic notions from partially hyperbolic dynamics. 

%We refer the reader to \cite{BW} for a deeper discussion of partial hyperbolicity and the properties that follow. %
We will assume for the rest of the paper that $f$ is $C^{\infty}$. Then the bundles $E^{s}$ and $E^{u}$ are tangent to foliations $\mathcal{W}^{s}$ and $\mathcal{W}^{u}$ known respectively as the stable and unstable foliations. These foliations have $C^{\infty}$ leaves but the distributions $E^{s}$ and $E^{u}$ which they are tangent to are themselves typically only H\"older continuous.

We say that $E^c(f)$ is integrable (or that $f$ has a center foliation) if there exists an $f-$invariant center foliation $\mathcal{W}^c = \{\mathcal{W}^c(x)\}_{x\in M}$ with $C^1-$ leaves everywhere tangent to the center bundle $E^c$. We say that $f$ is \emph{dynamically coherent} if there are also $f$-invariant foliations $\mathcal{W}^{cs}$ and $\mathcal{W}^{cu}$ with $C^{1}$ leaves which are tangent to $E^{s} \oplus E^{c}$ and $E^{c} \oplus E^{u}$ respectively. By \cite{BW08}, if $f$ is dynamically coherent then $f$ has a center foliation. Furthermorethat the foliations $\mathcal{W}^c$ and $\mathcal{W}^u$ subfoliate $\mathcal{W}^{cu}$ and the foliations $\mathcal{W}^c$ and $\mathcal{W}^s$ subfoliate $\mathcal{W}^{cs}$. The converse is not true; there exist examples of partially hyperbolic diffeomorphisms with an integrable center bundle that are not dynamically coherent \cite{HRHU10}. 

Suppose $E^c(f)$ is integrable and every leaf of the center foliation is compact.  The center foliation may not be a fibration; there may be leaves with non-trivial \textit{holonomy group}, the existence of which implies that $M/\mathcal{W}^c(f)$ is not a topological manifold (see \cite{BW05}). The holonomy group of a foliation is defined in Section \ref{section: center holonomy} (see also \cite{BB}, \cite{Eps76}). We call a foliation \textit{uniformly compact} if all its leaves are compact and have finite holonomy groups\footnote{ It is conjectured that every compact center foliation is uniformly compact, cf. \cite{BoThesis} \cite{Car11}, \cite{Gog11}.}. We say that a foliation has \emph{trivial holonomy} if the holonomy group of each leaf is trivial. Note that a foliation with trivial holonomy and compact leaves is uniformly compact. The existence of leaves with finite but nontrivial holonomy groups greatly complicates many of the constructions in our proofs; hence we will often assume that the center foliation has trivial holonomy which already covers many cases of interest. We also use that all partially hyperbolic diffeomorphisms with uniformly compact center foliation are dynamically coherent\cite{BB}. We note that Sullivan \cite{Sul76} (see also \cite{EV78}) has constructed an example of a circle foliation on a compact manifold with a leaf that has an infinite holonomy group. By \cite{Eps76} this implies the quotient space is even not a Hausdorff space. Our assumption of uniform compactness allows us to rule out these pathologies for the center foliation.

For $r \geq 1$ we write that a map is $C^{r+\alpha}$ if it is $C^{r}$ and the $r$th-order derivatives are uniformly H\"older continuous of exponent $\alpha > 0$. For a foliation $\mathcal{W}$ of an $n$-dimensional smooth manifold $M$ by $k$-dimensional submanifolds we define $\mathcal{W}$ to be a $C^{r+\alpha}$ foliation if for each $x \in M$ there is an open neighborhood $V_{x}$ of $x$ and a $C^{r+\alpha}$ diffeomorphism $\Psi_{x}: V_{x} \rightarrow D^{k} \times D^{n-k} \subset \RR^{n}$ (where $D^{j}$ denotes the ball of radius 1 centered at 0 in $\RR^{j}$) such that $\Psi_{x}$ maps $\W$ to the standard smooth foliation of $D^{k} \times D^{n-k}$ by $k$-disks $D^{k} \times \{y\}$, $y \in D^{n-k}$. This is the notion of regularity of a foliation considered by Pugh, Shub, and Wilkinson in their analysis of regularity properties of invariant foliations for partially hyperbolic systems \cite{PSW}. 

We say that $f$ is \emph{$r$-bunched} if  we can choose continuous positive functions $\nu, \hat{\nu}, \gamma , \hat{\gamma}$ with $$\nu,\hat{\nu}<1, \nu<\gamma<\hat{\gamma}^{-1}<\hat{\nu}^{-1}$$such that, for any unit vector $v\in T_xM$,
\begin{eqnarray*}
&\|Tf (v) \| &<  \nu(x), \text{ if } v\in E^s_x,\\
\gamma(x)<&\|Tf(v)\|& < \hat{\gamma}^{-1}(x), \text{ if }v\in E^c_x,\\
\hat{\nu}(x)^{-1}<&\|Tf(v)\|& ,\text{ if } v\in E^u_x.
\end{eqnarray*}
And $$\nu<\gamma^r, \hat{\nu}<\hat{\gamma}^r, \nu<\gamma\hat{\gamma}^r,\hat{\nu}<\hat{\gamma}\gamma^r$$
\color{black}

The case $r = 1$ corresponds to the \emph{center-bunching} condition considered by Burns and Wilkinson in their proof of the ergodicity of  accessible, volume-preserving, center-bunched $C^{2}$ partially hyperbolic diffeomorphisms \cite{BW}. When $f$ is smooth and dynamically coherent, the $r$-bunching inequalities imply that the foliations $\mathcal{W}^{cs}$ and $\mathcal{W}^{cu}$ have uniformly $C^{r+\alpha}$ leaves for some $\alpha > 0$ \cite{PSW}. We say that $f$ is \emph{$\infty$-bunched} if it is $r$-bunched for every $r \geq 1$. If $f$ is $\infty$-bunched  and dynamically coherent then the leaves of $\mathcal{W}^{cs}$ and $\mathcal{W}^{cu}$ are $C^{\infty}$. A natural situation in which the $\infty$-bunching condition holds is when there is a continuous Riemannian metric on $E^{c}$ with respect to which $Df|_{E^{c}}$ is an isometry. More generally if $f$ is center bunched, accessible, and volume-preserving and all of the Lyapunov exponents of $f$ with respect to volume on $E^{c}$ are zero,  then by the results of Kalinin and Sadovskaya $f$ is $\infty$-bunched \cite{KS13}.

%Finally, when $f$ is dynamically coherent we say that the center foliation $\mathcal{W}^{c}$ is \emph{uniformly compact} if all of the leaves of $\mathcal{W}^{c}$ are compact with a uniform bound on their diameters in the induced Riemannian metric on $E^{c}$ from $TM$. 

\begin{theo}\label{theorem: compact center}
Let $f$ be a $C^{\infty}$ volume-preserving partially hyperbolic diffeomorphism. Suppose that $f$ is uniformly quasiconformal,  $r$-bunched for some $r \geq 1$. In addition we suppose $f$ has compact center foliation and each center leaf has trivial holonomy group. Then 
\begin{enumerate}
\item There is an $\alpha > 0$ such that $\mathcal{W}^{cs}$, $\mathcal{W}^{c}$ and $\mathcal{W}^{cu}$ are $C^{r+\alpha}$ foliations of $M$. 

\item There is a closed $C^{r}$ Riemannian manifold $N$, a $C^{r+\alpha}$ submersion $\pi: M \rightarrow N$ with fibers given by the $\mathcal{W}^{c}$ foliation, and a $C^{r+\alpha}$ volume-preserving uniformly quasiconformal Anosov diffeomorphism $g: N \rightarrow N$ such that $g \circ \pi = \pi \circ f$. 

\item If $f$ is $\infty$-bunched then the statements of (1) and (2) are true with $r = \infty$. Furthermore $g$ may be taken to be a hyperbolic automorphism of a torus $N$. 
\end{enumerate} 
\end{theo}

When $\dim E^c =1$ or $\dim E^u =\dim E^s=2$, \color{black}
we can derive sharper results as a corollary. We define a smooth diffeomorphism $f: M \rightarrow M$ to be an \emph{isometric extension} of another smooth diffeomorphism $g: N \rightarrow N$ if there is a smooth submersion $\pi: M \rightarrow N$ satisfying $g \circ \pi = \pi \circ f$ and such that this submersion has compact fibers and there is a smoothly varying family of Riemannian metrics $\{d_{x}\}_{x \in N}$ on the fibers $\{\pi^{-1}(x)\}_{x \in N}$ such that the induced maps $f_{x}: \pi^{-1}(x) \rightarrow \pi^{-1}(g(x))$ are isometries with respect to these metrics.

\begin{cor}\label{theorem: isometric center}
Let $f$ be a $C^{\infty}$ volume-preserving partially hyperbolic diffeomorphism. Suppose that $f$ is uniformly quasiconformal.
\begin{enumerate}
\item If $\dim E^{c} = 1$, $f$ has compact center foliation and every central leaf has trivial holonomy group, then $f$ is an isometric extension of a hyperbolic toral automorphism.

\item If $\dim E^u =\dim E^s =2$ and $f$ has uniformly compact center foliation then the statements of Theorem \ref{theorem: compact center} hold for a finite cover of $f$. 
\end{enumerate}  
\end{cor}
\color{black}
We make some comments on Theorem \ref{theorem: compact center} and Corollary \ref{theorem: isometric center} before proceeding. If $M = N \times S$ for pair of compact smooth manifolds $N$ and $S$, $g_{0}: N \rightarrow N$ is an Anosov diffeomorphism and $f_{0}: M \rightarrow M$ is a smooth extension of $g_{0}$ such that $f_{0}$ is an $r$-bunched volume-preserving partially hyperbolic diffeomorphism ($r \geq 1$) with center leaves of the form $\W^{c}(x,s) = \{x\} \times S$ for $(x,s) \in N \times S$, then the center leaves of $f_{0}$ are normally hyperbolic, compact and have trivial holonomy\color{black}. Thus there is a $C^{1}$ open neighborhood $\mathcal{U}$ of $f_{0}$ in the space of $C^{\infty}$ volume-preserving diffeomorphisms of $M$ such that if $f \in \mathcal{U}$ then $f$ has a compact center foliation with trivial holonomy\color{black}. This follows from the theory of normally hyperbolic invariant manifolds developed by Hirsch, Pugh, and Shub \cite{HPS}. Hence, with the exception of the uniform quasiconformality hypothesis, the hypotheses of Theorem \ref{theorem: compact center} and Corollary \ref{theorem: isometric center} are not particularly restrictive among partially hyperbolic diffeomorphisms. Finally we emphasize that in part (2) of Corollary \ref{theorem: isometric center} we do not have to assume that the center leaves have trivial holonomy groups. 

The limiting factor for the smoothness of the foliations $\W^{cs}$ and $\W^{cu}$ in Theorem \ref{theorem: compact center} turns out to be the regularity of the leaves of the foliations themselves. Corollary \ref{cor: analytic} below shows that the holonomy maps of $\W^{cs}$ and $\W^{cu}$ between local unstable/local stable leaves respectively are $C^{\infty}$. In fact they are analytic maps in an appropriate choice of coordinates. The $r$-bunching inequalities in the hypotheses of Theorem \ref{theorem: compact center} are only required to obtain that the leaves of the foliations  $\W^{cs}$ and $\W^{cu}$ are $C^{r+\alpha}$; they are never used directly in the proof. The regularity of the uniformly quasiconformal Anosov diffeomorphism $g$ obtained from Theorem \ref{theorem: compact center} is limited by the regularity of the center foliation, which in turn is limited by the $r$-bunching hypothesis. The most we can obtain with our methods is that $g$ is $C^{r+\alpha}$. This is the reason we can only derive the stronger results of part (3) of Theorem \ref{theorem: compact center} under the $\infty$-bunching hypothesis on $f$. 

Finally we observe that the conclusions of Theorem \ref{theorem: compact center} imply in particular that the center foliation of $f$ is absolutely continuous with respect to volume. We refer to Definition \ref{defn: ac} below for our definition of absolute continuity of a foliation. Pugh and Wilkinson showed that an isometric extension of a hyperbolic automorphism of the two-dimensional torus $\mathbb{T}^{2}$ can be perturbed to make the center Lyapunov exponent nonzero and thus cause the center foliation to fail to be absolutely continuous \cite{PW00}. Corollary \ref{theorem: isometric center} shows that it is not possible to make such a perturbation of an isometric extension of a uniformly quasiconformal hyperbolic automorphism of a higher dimensional torus which maintains uniform quasiconformality on both the stable and unstable bundles.  

For our next theorem we consider partially hyperbolic diffeomorphisms which are obtained as perturbations of the time-1 maps of Anosov flows. Let $\psi_{t}: M \rightarrow M$ be a $C^{\infty}$ volume-preserving Anosov flow with stable and unstable bundles of dimension at least 2. 
\begin{theo}\label{theorem: flow suspension}
Suppose that there is a finite cover $\hat{M}$ of $M$ such that the lift of $\psi_{t}$ to an Anosov flow $\hat{\psi}_{t}: \hat{M} \rightarrow \hat{M}$ has no periodic orbits of period $\leq 2$.

Then there is a $C^{1}$-open neighborhood $\mathcal{U}$ of $\psi_{1}$ in the space of volume-preserving $C^{\infty}$ diffeomorphisms of $M$ such that if $f \in \mathcal{U}$ and $f$ is uniformly quasiconformal then the invariant foliations $\mathcal{W}^{cs}$, $\mathcal{W}^{c}$, and $\mathcal{W}^{cu}$ of $f$ are $C^{\infty}$ and there is a $C^{\infty}$ volume-preserving uniformly quasiconformal Anosov flow $\varphi_{t}: M \rightarrow M$ with $\varphi_{1} = f$. 
\end{theo}
Before making further comments on this theorem we recall the notion of orbit equivalence of Anosov flows. Two $C^{\infty}$ Anosov flows $\varphi_{t},\psi_{t}: M \rightarrow M$ are \emph{$C^{r}$ orbit equivalent} ($r \in [0,\infty]$) if there is a $C^{r}$ map $h: M \rightarrow M$ such that for every $x \in M$ and $t \in \R$, $h(\varphi_{t}(x))$ lies on the $\psi_{t}$-orbit of $h(x)$.

From the classification of $C^{\infty}$ volume-preserving uniformly quasiconformal Anosov flows obtained by Fang \cite{F14} we conclude that the flow $\varphi_{t}$ obtained in the conclusion of Theorem \ref{theorem: flow suspension} is $C^{\infty}$ orbit equivalent either to the suspension flow of a hyperbolic toral automorphism or the geodesic flow on the unit tangent bundle of a constant negative curvature Riemannian manifold. 

The hypothesis in Theorem \ref{theorem: flow suspension} that there is a finite cover $\hat{M}$ for which the lift $\hat{\psi}_{t}$ has no periodic orbits of period $\leq 2$ is very mild. It always holds if $\psi_{t}$ is $C^{0}$ orbit equivalent to the suspension flow of an algebraic Anosov diffeomorphism or the geodesic flow of a closed negatively curved Riemannian manifold. We expect that Theorem \ref{theorem: flow suspension} holds without this hypothesis, however this hypothesis does simplify some constructions in the proofs, particularly in Section \ref{subsec: fiber}.

We recall now the definitions of $su$-paths and accessiblity for a partially hyperbolic diffeomorphism which will play a crucial role in the proof of Theorem \ref{theorem: hyperbolic perturbation} in Section \ref{section:theorem4}. For a partially hyperbolic diffeomorphism $f: M \rightarrow M$ an $su$-path in $M$ is a piecewise $C^{1}$ curve $\gamma$ in $M$ such that $\gamma$ decomposes into finitely many $C^{1}$ subcurves $\gamma_{x_{i}x_{i+1}}$ connecting $x_{i}$ to $x_{i+1}$ and such that each curve $\gamma_{x_{i}x_{i+1}}$ is contained in a single $\W^{s}$ or $\W^{u}$ leaf. We define $f$ to be \emph{accessible} if any two points in $M$ can be joined by an $su$-path. 

A notable aspect of Theorems \ref{theorem: compact center} and \ref{theorem: flow suspension} and Corollary \ref{theorem: isometric center} is that their hypotheses do not include any accessibility or ergodicity assumptions on $f$ with respect to the volume $m$.  This requires us to take some additional care at certain points in the proof. The accessibility hypotheses is used strongly in the rigidity theorem of Avila-Viana-Wilkinson and ergodicity with respect to volume is used in the classification results of Sadovskaya and Fang. 

The results of Corollary \ref{theorem: isometric center} and Theorem \ref{theorem: flow suspension} suggest that it may be possible to obtain a global smooth classification of $C^{\infty}$ volume-preserving, dynamically coherent, uniformly quasiconformal partially hyperbolic diffeomorphisms with one-dimensional center in terms of the classification of uniformly quasiconformal Anosov diffeomorphisms and Anosov flows. We give an example which illustrates some of the difficulties in obtaining a classification beyond these theorems.  

 Consider the $5 \times 5$ integer matrix 
 $$A:=\begin{pmatrix}
0&1&0&0&0\\
0&0&1&0&0\\
0&0&0&1&0\\
0&0&0&0&1\\
-1&1&0&-3&1
\end{pmatrix},$$
and let $f_{A}: \TT^5\rightarrow\TT^5$ be the induced linear map of $A$ on the $5$-torus $\TT^5 = \RR^{5}/\Z^{5}$. By numerical computation the five complex eigenvalues of $A$ satisfy
\[
|\lambda_1|=|\lambda_2|>|\lambda_3|>1>|\lambda_4|=|\lambda_5|,
\]
\[
\overline{\lambda_1}={\lambda_2}\notin \mathbb{R},
\]
\[
\overline{\lambda_4}={\lambda_5}\notin \mathbb{R}.
\]
Thus $f_A$ is a hyperbolic toral automorphism which may also be viewed as a partially hyperbolic diffeomorphism with splitting $T \TT^5 = E^{u} \oplus E^{c} \oplus E^{s}$, where $E^{u}$ is the real part of the complex eigenspaces corresponding to the pair of conjugate complex eigenvalues $\la_{1}$ and $\la_{2}$,  $E^{c}$ is the eigenspace corresponding to $\la_{3}$, and $E^{s}$ is the real part of the complex eigenspaces corresponding to $\la_{4}$ and $\la_{5}$. We conclude $f_{A}$ is a smooth, volume-preserving, dynamically coherent uniformly quasiconformal partially hyperbolic diffeomorphism with one-dimensional center. 

We then pose the following problem, 

\begin{prob}\label{problem} 
Is there a $C^{1}$-open neighborhood $\mathcal{U}$ of $f_{A}$ in the space of smooth volume-preserving diffeomorphisms of $\TT^{5}$ such that if $f \in \mathcal{U}$ is uniformly quasiconformal then the invariant foliations $\W^{cs}$, $\W^{c}$, and $\W^{cu}$ of $f$ are smooth?
\end{prob}

We expect the answer to Problem \ref{problem} to be ``no" but the difficulty of constructing nontrivial uniformly quasiconformal perturbations of $f_{A}$ is a significant obstruction to confirming our suspicions.  We note that each $f \in \mathcal{U}$ is an Anosov diffeomorphism if $\mathcal{U}$ is chosen small enough. 

\section{Quasiconformality of the Center Holonomy}\label{section: center holonomy}
%
%Question: Do we always need to assume that $f$ is 1-bunched, i.e., that $f$ is center bunched throughout, so that $\W^{c}$, $\W^{cu}$, and $\W^{cs}$ are actually foliations with $C^{1}$ leaves? There is a minimal amount of regularity of the leaves throughout the arguments that might be needed. If there is less bunching are the leaves no longer differentiable (does it make sense to talk about them as dynamical foliations)? 
%\color{black}

\subsection{Holonomy along paths and the holonomy group of a foliation} In this subsection we define the notions of the holonomy group of a leaf of a foliation and uniform compactness for a foliation which appear in the statements of Theorem \ref{theorem: compact center} and Corollary \ref{theorem: isometric center}. For more details see \cite{BoThesis, BB}.

Consider a $q-$codimensional foliation $\mathcal{F}$ in a compact manifold $M$. Suppose $x \in M$, $y\in \mathcal{F}_{loc}(x)$ and $D_x, D_y$ are two small $C^1-$discs transverse to $\mathcal{F}$. The \textit{local holonomy map} $h^{\mathcal{F}}_{x,y}:D_x\to D_y$ is defined as the following: for any $z\in D_x$ we let $h^{\mathcal{F}}_{x,y}(z)$ be the unique point at which the local leaf $\mathcal{F}_z$ intersects $D_y$.

Moreover, for any continuous path $\gamma:[0,1]\to M$ that lies entirely inside a leaf $\mathcal{F}(x)$ we define the \textit{holonomy along }$\gamma$ (denoted $h^\mathcal{F}_\gamma$) as follows: Suppose $0=t_0<\cdots <t_n=1$ is a subdivision such that $|t_i-t_{i-1}|$ small enough and $x_i:= \gamma(t_i)$. We pick a sequence of $C^1-$ small discs $D(x_i)\ni x_i $ which are transverse to $\mathcal{F}$. By the definition of the local holonomy map,  $h^{\mathcal{F}}_{x_{k-1},x_k}: D(x_{k-1})\to D(x_k)$ is well-defined on a neighborhood of $x_k\in D(x_k)$. Then the holonomy along $\gamma$ is given by the formuls
\begin{equation}\label{def holonomy along path}
h^\mathcal{F}_\gamma:=h^{\mathcal{F}}_{x_{n-1},x_n}\circ\cdots\circ h^{\mathcal{F}}_{x_0,x_1}: D(\gamma(0))\to D(\gamma(1)).
\end{equation} 
This formula remains well-defined on a neighborhood of $\gamma(0)\in D(\gamma(0))$. 

Consider all closed paths $\gamma$ that lie in the leaf $\mathcal{F}(x)$ with $\gamma(0)=\gamma(1)=x$ and consider a small $C^1-$disc $D(x)\ni x$ transverse to $\mathcal{F}(x)$. By identifying $D(x)$ with $\RR^q$  we get a group homomorphism $$\pi_1  (\mathcal{F}(x), x ) \to Homeo  (\mathbb{R}^q,  0)$$
 where $Homeo (\mathbb{R}^q,  0)$ is the set of germs of homeomorphisms $R^q\to R^q$ which fix the origin (since the germ of $h^\mathcal{F}_\gamma$ only depends on the homotopy class of $\gamma$). The image of the homomorphism is called the \textit{holonomy
group} of the leaf $\mathcal{F}(x)$  and denoted by $Hol(\mathcal{F}(x), x)$. A leaf $\mathcal{F}(x)$ of a foliation $\mathcal{F}$ \textit{has finite (or trivial) holonomy} if the holonomy group $Hol(\mathcal{F}(x), x )$ is a finite (or trivial)
group for any $x\in M$. A foliation is called \textit{uniformly compact} if every leaf is compact and has finite holonomy group. The following general lemma will be needed later.
\begin{lemma}\label{lemma path trivial holonomy in big foliation}Suppose $\mathcal{F}_{i}, i=1,2$ are two foliations of a manifold $M$ such that $\mathcal{F}_1$ subfoliates $\mathcal{F}_2$. Assume $\gamma$ is a  closed path which lies in a leaf of $\mathcal{F}_1$ and represents the identity in $Hol(\mathcal{F}_1(\gamma(0)), \gamma(0))$. Then $\gamma$ also represents the identity in $Hol(\mathcal{F}_2(\gamma(0)), \gamma(0))$. 
\end{lemma}
\begin{proof}Suppose the foliations $\mathcal{F}_i,i=1,2$ have codimension $q_i, i=1,2,  q_2<q_1$ respectively. For any $q_2-$dimensional $C^1-$disc $D_2$ such that $D_2\cap \mathcal{F}_2=\{\gamma(0)\}$, we can find a $q_1-$dimensional $C^1-$disc $D_1\supset D_2$ such that  $D_1 \cap \mathcal{F}_1=\{\gamma(0)\}$ since $\mathcal{F}_1$ subfoliates $\mathcal{F}_2$. Since by assumption $h_\gamma^{\mathcal{F}_1}:D_1\to D_1$ is the  identity map, this implies that $h_\gamma^{\mathcal{F}_2}|_{D_2}=(h_\gamma^{\mathcal{F}_1}|_{D_1})|_{D_2}$ is also the identity map, which implies the assertion of the lemma.
\end{proof}
\subsection{Quasiconformality of center holonomy}\label{subsection: quasiconfomal of hc}
\color{black}
We fix $M$ to be a closed Riemannian manifold with distance $d$ and let $f: M \rightarrow M$ be a $C^{\infty}$ dynamically coherent partially hyperbolic diffeomorphism. For $* \in \{s,c,u,cu,cs\}$ we let $d_{*}$ denote the induced Riemannian metric on the leaves of the foliation $\mathcal{W}^{*}$. We write $\W^{*}(x)$ for the leaf of $\W^{*}$ passing through $x \in M$. We write $\text{diam}_{*}$ for the diameter of a subset of $\mathcal{W}^{*}$ measured with respect to the $d_{*}$ metric. For $r > 0$ we write $\W^{*}_{r}(x)$ for the open ball of radius $r$ in $\W^{*}(x)$ centered at $x$ in the $d_{*}$ metric.

We can find  small constants $R \geq r > 0$ with the property that for any $x \in M$, $y \in \W^{cs}_{r}(x)$ and $z \in \W^{u}_{r}(x)$ the local leaves $\W^{cs}_{r}(z)$ and $\W_{R}^{u}(y)$ intersect in exactly one point which we denote by $h^{cs}_{xy}(z)$. This defines the \emph{local center-stable holonomy} map between local unstable leaves of $f$. Similarly we require that if $x \in M$, $y \in \W^{cu}_{r}(x)$ and $z \in \W^{s}_{r}(x)$ then the local leaves $\W^{cu}_{r}(z)$ and $\W_{R}^{s}(y)$ intersect in exactly one point which we denote by $h^{cu}_{xy}(z)$, and use this to define the \emph{local center-unstable holonomy}. For any continuous path $\gamma:[0,1]\to M$ that lies in a center stable leaf we define the \textit{center stable holonomy along $\gamma$}, $h^{cs}_{\gamma}:\W^u_{\epsilon}(\gamma(0))\to \W^u_{\epsilon}(\gamma(1)) $, as in \eqref{def holonomy along path} for $\epsilon$ small enough. The only difference in the definition here is that all of the transversal discs along $\gamma$ are required to be local unstable discs. 
\color{black}

We introduce some useful shorthand related to these holonomy maps. The center-stable holonomy maps and center-unstable holonomy maps will sometimes be referred to as $cs$-holonomy and $cu$-holonomy respectively. When the domain and range are understood we will omit the subscripts on $h^{cs}$ and $h^{cu}$. We will write $\W^{*}_{loc}(x)$ for any open ball of the form $\W^{*}_{t}(x)$ with $r \leq t \leq R$. Hence it makes sense in our shorthand to write $h^{cs}: \W^{u}_{loc}(x) \rightarrow \W^{u}_{loc}(y)$ for the $cs$-holonomy maps.

Our starting point is the following non-stationary smooth linearization lemma of Sadovskaya applied to the unstable foliation $\W^{u}$ which is uniformly contracted by $f^{-1}$, 
\begin{prop}\cite[Proposition 4.1]{S05}\label{proposition: charts}
Suppose that $f$ is a $C^{\infty}$ uniformly $u$-quasiconformal partially hyperbolic diffeomorphism. Then for each $x \in M$ there is a $C^{\infty}$ diffeomorphism $\Phi_{x}: E^{u}_{x} \rightarrow \mathcal{W}^{u}(x)$ satisfying
\begin{enumerate}
\item $\Phi_{f(x)} \circ Df_{x} = f \circ \Phi_{x}$,

\item $\Phi_{x}(0) = x$ and $D_{0}\Phi_{x}$ is the identity map,

\item The family of diffeomorphisms $\{\Phi_{x}\}_{x \in M}$ varies continuously with $x$ in the $C^{\infty}$ topology. 
\end{enumerate}
The family $\{\Phi_{x}\}_{x \in M}$ satisfying (1), (2), and (3) is unique. 
\end{prop}
%Proposition 4.1 is stated for $r = \infty$ in \cite{S05}, however the linearizations $\Phi_{x}$ are constructed via power series formulas which are defined as soon as $r \geq 1$ and yield the desired regularity of $\Phi_{x}$ for finite $r$. 

The bundle $E^{u}$ is a H\"older continuous subbundle of $TM$  with some H\"older exponent $\beta > 0$ \cite{PSW}. Therefore the restriction $Df|_{E^{u}}$ of the derivative of $f$ to the unstable bundle is a H\"older continuous linear cocycle over $f$ in the sense of Kalinin-Sadovskaya \cite{KS13}. For $x,y \in M$ two nearby points we let $I_{xy}: E^{u}_{x} \rightarrow E^{u}_{y}$ be a linear identification which is $\beta$-H\"older close to the identity. The diffeomorphism $f$ is uniformly $u$-quasiconformal if and only if, in the terminology of \cite{KS13}, the cocycle $Df|_{E^{u}}$ is \textit{uniformly quasiconformal} (therefore $Df|_{E^u}$ is \textit{fiber bunched}, cf. \cite{KS13}). The following proposition thus applies to $Df|_{E^{u}}$.

\begin{prop}\cite[Proposition 4.2]{KS13}\label{proposition: holonomies}
For $y \in \mathcal{W}^{u}_{loc}(x)$, the limit 
\[
\lim_{n \rightarrow \infty}Df^{n}_{f^{-n}y} \circ I_{f^{-n}xf^{-n}y} \circ Df^{-n}_{x} |_{E^{u}} :=H^{u}_{xy},
\]
exists uniformly in $x$ and $y$ and defines a linear map from $E^{u}_{x}$ to $E^{u}_{y}$ with the following properties for $x,y,z \in M$,
\begin{enumerate}
\item $H^{u}_{xx} = Id$ and $H^{u}_{yz} \circ H^{u}_{xy} = H^{u}_{xz}$;
\item $H^{u}_{xy} = Df^{n}_{f^{-n}y} \circ H^{u}_{f^{-n}xf^{-n}y} \circ Df^{-n}_{x}$  for any $n \geq 0$. 
\item $\| H^{u}_{xy} - I_{xy}\| \leq Cd(x,y)^{\beta}$, $\beta$ the exponent of H\"older continuity for $E^{u}$.  
\end{enumerate}
Furthermore $H^{u}$ is the unique collection of linear identifications with these properties. Similarly if $y \in \mathcal{W}^{s}_{loc}(x)$ then the limit $\lim_{n \rightarrow \infty}Df^{-n}_{y} \circ I_{f^{n}xf^{n}y} \circ Df^{n}_{x} |_{E^{u}} :=H^{s}_{xy}$ exists and gives a linear map from $E^{u}_{x}$ to $E^{u}_{y}$ with analogous properties. $H^{u}$ and $H^{s}$ are known as the \emph{unstable} and \emph{stable holonomies} of $Df|_{E^{u}}$ respectively. 
\end{prop}

Using property (2) of the unstable and stable holonomies of $Df|_{E^{u}}$ from Proposition \ref{proposition: holonomies} we may uniquely extend $H^{u}$ and $H^{s}$ to be defined for any $y \in \W^{u}(x)$ and any $y \in \W^{s}(x)$ respectively.

The transition maps between the charts given by Proposition \ref{proposition: charts} are affine with derivatives given by the unstable holonomy $H^{u}$, 

\begin{prop}\label{proposition: transitions}
Suppose that $f$ is uniformly $u$-quasiconformal and let $\{\Phi_{x}\}_{x \in M}$ be the charts of Proposition \ref{proposition: charts}. Then for each $x \in M$ and $y \in \mathcal{W}^{u}(x)$ the map $\Phi_{y}^{-1} \circ \Phi_{x}: E^{u}_{x} \rightarrow E^{u}_{y}$ is an affine map with derivative $H^{u}_{xy}$. 
%For each $x \in M$, $y,z \in \mathcal{W}^{u}(x)$ and $w = \Phi_{y}(\Phi_{x}^{-1}(z))$, 
%\[
%D_{z}(\Phi_{y} \circ \Phi_{x}^{-1}) = H^{u}_{z w}
%\] 
\end{prop}

\begin{proof}
%To prove (1), we observe that for any $n \geq 0$, 
%\begin{align*}
%D_{z}(\Phi_{y} \circ \Phi_{x}^{-1})  &=  D_{z}(f^{n} \circ \Phi_{f^{-n}y} \circ Df^{-n}_{y} \circ  Df^{n}_{f^{-n}x} \circ \Phi_{f^{-n}x}^{-1} \circ f^{-n}) \\
%&= Df^{n}_{f^{-n}w} \circ D_{f^{-n}z}(\Phi_{f^{-n}y} \circ Df^{-n}_{y} \circ  Df^{n}_{f^{-n}x} \circ \Phi_{f^{-n}x}^{-1}) \circ Df^{-n}_{z},
%\end{align*}
%We claim that we have a bound 
%\[
%\left\| D_{f^{-n}z}(\Phi_{f^{-n}y} \circ Df^{-n}_{y} \circ  Df^{n}_{f^{-n}x} \circ \Phi_{f^{-n}x}^{-1})   - I_{f^{-n}z f^{-n}w} \right\| \leq C(z) d(f^{-n}(z),f^{-n}(w))^{\beta},
%\]
%for some constant $C \geq 1$ depending only on the distance from $x$ to $y$ and $z$. This follows from the facts that the charts $\{\Phi_{x}\}_{x \in M}$ depend continuously on $x$ in the $C^{\infty}$ topology and that $\|Df^{-n}_{y} \circ  Df^{n}_{f^{-n}x} - I_{f^{-n}x f^{-n}y}\| \leq C d(f^{-n}x,f^{-n}y)^{\beta}$ by Proposition \ref{proposition: holonomies}. By the proof of \cite[Proposition 4.2]{KS13} we then see that 
%\begin{align*}
%&\lim_{n \rightarrow \infty} Df^{n}_{f^{-n}w} \circ D_{f^{-n}z}(\Phi_{f^{-n}y} \circ Df^{-n}_{y} \circ  Df^{n}_{f^{-n}x} \circ \Phi_{f^{-n}x}^{-1}) \circ Df^{-n}_{z} \\
%&= \lim_{n\rightarrow\infty} Df^{n}_{f^{-n}w} \circ I_{f^{-n}zf^{-n}w}\circ Df^{-n}_{z} = H^{u}_{zw}
%\end{align*}

For any $n \geq 0$ and any $v \in E^{u}_{x}$ we use the defining properties of the charts $\{\Phi_{x}\}_{x \in M}$ to write
\begin{align*}
D_{v}(\Phi_{y}^{-1} \circ \Phi_{x})  &=  D_{v}(Df^{n}_{f^{-n}y} \circ \Phi_{f^{-n}y}^{-1} \circ \Phi_{f^{-n}x} \circ Df^{-n}_{x}) \\
&= Df^{n}_{f^{-n}y} \circ D_{Df^{-n}(v)}(\Phi_{f^{-n}y}^{-1} \circ \Phi_{f^{-n}x}) \circ Df^{-n}_{x} ,
\end{align*}
We have a bound 
\[
\left\| D_{Df^{-n}(v)}(\Phi_{f^{-n}y}^{-1} \circ \Phi_{f^{-n}x})    - I_{f^{-n}x f^{-n}y} \right\| \leq C(v) d(f^{-n}(x),f^{-n}(y))^{\beta},
\]
with the constant $C(v)$ depending only on the distance of $v$ from the origin in $E^{u}_{x}$,  because the charts $\{\Phi_{x}\}_{x \in M}$ vary continuously in the $C^{\infty}$ topology. From the existence of this bound and the proof of \cite[Proposition 4.2]{KS13} we conclude that 
\[
\lim_{n \rightarrow \infty} Df^{n}_{f^{-n}y} \circ D_{Df^{-n}(v)}(\Phi_{f^{-n}y}^{-1} \circ \Phi_{f^{-n}x}) \circ Df^{-n}_{x} = H^{u}_{xy}
\]
This implies that $D_{v}(\Phi_{y}^{-1} \circ \Phi_{x}) = H^{u}_{xy}$ for every $v \in E^{u}_{x}$, from which it follows that $\Phi_{y}^{-1} \circ \Phi_{x}$ is an affine map from $E^{u}_{x}$ to $E^{u}_{y}$ with linear part $H^{u}_{xy}$. 

% Note: I think that only part (2) of this Lemma is needed in the paper. 
%\color{black}
\end{proof}

We now set $k := \dim E^{u}$ and recall that our assumption that $f$ is uniformly $u$-quasiconformal requires that $k \geq 2$. We recall the notion of a quasiconformal map between domains in $\mathbb{R}^{k}$ where we equip $\R^{k}$ with the Euclidean norm $\| \cdot \|$, 

\begin{defi}\label{definition: dilatation} 
Let $h: U \rightarrow V$ be a homeomorphism between two open subsets $U,V$ of $\RR^{k}$ .The linear dilatation of $h$ at $x \in U$ is defined to be 
\[
	L_{h}(x)=\limsup_{r\to 0}\frac{\max_{\|y-x\|=r}\|f(y)-f(x)\|}{\min_{\|y-x\|=r}\|f(y)-f(x)\|}
\]
For $K \geq 1$ we define $h$ to be \emph{$K$-quasiconformal} if $L_{h}(x) \leq K$ for every $x \in U$. 
\end{defi}

Each of the normed vector spaces $E^{u}_{x}$ (with norm induced from the Riemannian metric on $TM$) carries the linear structure of $\RR^{k}$ with a norm that is uniformly comparable to the Euclidean norm on $\RR^{k}$. Hence $K$-quasiconformality can also be defined for homeomorphisms between open subsets of $E^{u}_{x}$ and $E^{u}_{y}$ for $x,y \in M$. It is this sense of $K$-quasiconformality which is used in Lemma \ref{lemma: quasiconformal} below.   Recall that the inverse of a continuous path $\gamma:[0,1]\to M$ is defined by $\gamma^{-1}:[0,1]\to M, \gamma^{-1}(t):=\gamma(1-t)$. For two paths $\gamma_1, \gamma_2: [0,1]\to M$ such that $\gamma_1(1)=\gamma_2(0)$, denote the composition of $\gamma_1, \gamma_2$ by $$\gamma_1\cdot \gamma_2:[0,1]\to M, \gamma_1\cdot \gamma_2(t)=\begin{cases}\gamma_1(2t), \quad 0\leq t\leq \frac{1}{2}\\ \gamma_2(2t-1), \quad  \frac{1}{2}\leq t\leq 1\end{cases}$$ Finally the length of a piecewise $C^1$ path $\gamma$ is denoted by $l(\gamma)$. We make the  following crucial definition.
\begin{defi}\label{definition of center non expansive}Suppose $f$ is a dynamically coherent partially hyperbolic diffeomorphism on $M$. A path $\gamma: [0,1]\to M$ is called a \textit{ good (local)  $*-$path}, $* \in \{s,c,u,cu,cs\}$ if $\gamma$ is piecewise $C^1$ and lies entirely in one $\mathcal{W}^*$ (local) leaf.  $f$ is called \textbf{center non-expansive} if there exists $l > 0$ which satisfies the following property: for any $x\in M$, $y\in \mathcal{W}^{c}_{loc}(x), n\geq 0$ and any good local $c-$path $\gamma$ from $x$ to $y$, there exists a good $c-$path $\gamma_n$ from $f^n(x)$ to $f^n(y)$ with $l(\gamma_n)\leq l$ such that $f^n(\gamma)\cdot \gamma_n^{-1}$ represents the identity element in $Hol(\mathcal{W}^c(f^n(x)),f^n(x))$.
\end{defi}
We will see the utility of this definition later as both partially hyperbolic diffeomorphisms with uniformly compact center foliation and $C^1-$perturbations of the time one map of an Anosov flow with no periodic orbits of period $\leq 2$ are both center non-expansive. Now we can state the main result of this section.

\begin{lemma}\label{lemma: quasiconformal} 
Let $f$ be a $C^{\infty}$ dynamically coherent partially hyperbolic diffeomorphism. Suppose that $f$ is uniformly $u$-quasiconformal and center non-expansive.
Then there is a constant $K \geq 1$ such that for any two points $x \in M$, $y \in \mathcal{W}^{cs}_{loc}(x)$, the homeomorphism 
\[
\Phi_{y}^{-1} \circ h^{cs} \circ \Phi_{x}:  \Phi_{x}^{-1}(\W^{u}_{loc}(x)) \rightarrow \Phi_{y}^{-1}(\W^{u}_{loc}(y))
\]
is $K$-quasiconformal.
\end{lemma}

\begin{proof}
%The center-stable bundle is uniformly H\"older continuous and consequently the local center stable holonomies between nearby transversals to the center stable foliation are uniformly H\"older continuous \cite{PSW}.  

Our strategy to prove Lemma \ref{lemma: quasiconformal} is to use forward iteration of $f$ to exploit the center non-expansive property in order to control the behavior of $ h^{cs}=h^{cs}_{xy}$ on a small \textit{unstable annulus} centered at $x$ inside of $\W^{u}(x)$. The first step is the following lemma.

\begin{lemma}\label{lemma forwad iterate hcs}There exists $C_0>0$ large enough such that for any $x\in M, y \in \mathcal{W}^{cs}_{loc}(x), n\geq 0$, there is a good $cs-$path $\gamma_n$ with $l(\gamma_n)\leq C_0$ and 
\begin{equation}\label{equation forwad iterate hcs}
\Phi_{y}^{-1} \circ h^{cs} \circ \Phi_{x}=Df_{y}^{-n} \circ \Phi_{f^{n}(y)}^{-1}  \circ h^{cs}_{\gamma_n} \circ \Phi_{f^{n}(x)} \circ Df^{n}_{x}
\end{equation}

\end{lemma}

\begin{proof}Notice that for any good local $cs-$path $\gamma$ from $x$ to $y$, $h^{cs}=h^{cs}_{xy}=h^{cs}_{\gamma}$.  Using the equivariance properties of the charts from Proposition \ref{proposition: charts}, we have for every $n \geq 0$ and good local $cs-$path $\gamma$, 
\begin{align*}
\Phi_{y}^{-1} \circ h^{cs} \circ \Phi_{x} &=\Phi_{y}^{-1} \circ h_\gamma^{cs} \circ \Phi_{x}\\
&=  \Phi_{y}^{-1} \circ f^{-n} \circ h^{cs}_{f^n(\gamma)} \circ f^{n} \circ \Phi_{x} \\
&=  Df_{y}^{-n} \circ \Phi_{f^{n}(y)}^{-1}  \circ h^{cs}_{f^n(\gamma)} \circ \Phi_{f^{n}(x)} \circ Df^{n}_{x}
\end{align*}
Without loss of generality we assume $\gamma=\gamma^c\cdot \gamma^s$, where $\gamma^c$ and $\gamma^s$ are a good local $c-$path and a good local $s-$path respectively. 
Since $f$ is center non-expansive, there is a good $c-$path $\gamma_n'$ from $f^n(x)=f^n(\gamma^c(0))$ to $f^n(\gamma^c(1)) $ with $l(\gamma_n')\leq l$ such that $f^n(\gamma^c)\cdot {\gamma_n'}^{-1}$ represents the identity in $Hol(\mathcal{W}^c(f^n(x)), f^n(x))$, where $l=l(f)$ is the bound in the Definition \ref{definition of center non expansive} that only depends on $f$. By Lemma \ref{lemma path trivial holonomy in big foliation}, $f^n(\gamma^c)\cdot {\gamma_n'}^{-1}$ also represents the identity in $Hol(\mathcal{W}^{cs}(f^n(x)),f^n(x))$, therefore $$h^{cs}_{f^n(\gamma)}=h^{cs}_{f^n(\gamma^c)\cdot f^n(\gamma^s)}=h^{cs}_{f^n(\gamma^s)}\circ h^{cs}_{f^n(\gamma^c)}=h^{cs}_{f^n(\gamma^s)}\circ h^{cs}_{\gamma_n'}=h^{cs}_{\gamma_n'\cdot f^n(\gamma^s)}$$
Let $\gamma_n$ be $\gamma_n'\cdot f^n(\gamma^s)$, then \eqref{equation forwad iterate hcs} holds. Moreover $l(\gamma_n)$ is uniformly bounded by some $C_0$ independent of the choice $x,y$ and $n$ since $f$ uniformly contracts $\mathcal{W}^s$. 
\end{proof}

We now come back to the proof of Lemma \ref{lemma: quasiconformal}. Since $f$ is uniformly $u$-quasiconformal there is a constant $\kappa \geq 1$ such that for every $x \in M$, any $r>0$, $n\in\ZZ$ and every $v,w \in E^{u}_{x}$ with $\|v\| = \|w\| =  r$,
\begin{equation}\label{eqn uniformy bounded delattion}
\kappa^{-1}\leq \frac{\|Df^{n}(w)\|}{\|Df^{n}(v)\|} \leq  \kappa
\end{equation}
We set $A:= \sup_{x \in M} \max \{\|Df|_{E^{u}(x)}\|,\|Df^{-1}|_{E^{u}(x)}\|\}$.

Recall that $\W^{*}_{r_0}(x)$ is the open ball of radius $r_0$ in $\W^{*}(x)$ centered at $x$ in the $d_{*}$ metric. The $*-$ closed ball of radius $r_1$ in $\W^{*}(x)$ centered at $x$ is denoted by $\B^{*}_{r_1}(x)$. The $*-$ closed annulus of radius $r_2, r_3 (r_2\leq r_3)$ centered at $x$ is defined by $\cR^{*}(x, r_2,r_3):=\B^{*}_{r_3}(x)\backslash \W^{*}_{r_2}(x)$.

 By compactness of $M$, for any $C>0$  there exists $\epsilon(C)$ small enough such that for any good $cs-$path $\gamma$ with $l(\gamma)\leq C$, the holonomy map $h^{cs}_\gamma$ is well-defined on $\W^u_{\epsilon}(x)$. 
 Moreover for any $0<\zeta\ll \epsilon(C)$ small enough, there is a constant $L=L(\zeta, C)$ such that for any good $cs-$path $\gamma$ such that $l(\gamma)\leq C$ and any 
 $z \in \cR^u(\gamma(0), \kappa^{-2}\zeta, A\kappa^2\zeta)$, we have that $h_\gamma^{cs}(z)$ is well-defined and 
\begin{equation}\label{eqn:csuniform}h_\gamma^{cs}(z)\in 
\cR^u(\gamma(1), L(\zeta,C)^{-1}, L(\zeta,C))
\end{equation}

%By Lemma 3.5 in \cite{Bohnet-Bonatti} (boundedness of holonomy cover), for any $n\in \ZZ$, there is a path $\gamma_n$ with length uniformly bounded by a number $C_0$ independent of $x,y,n$, such that $h^{cs}_{\gamma_n}=h^{cs}_{f^n(\gamma)}$. 

Let $\epsilon:=\epsilon(C_0)$, where $C_0$ is the constant provided by Lemma \ref{lemma forwad iterate hcs}. We fix a $\zeta$ small enough such that $\kappa A\zeta\ll \epsilon$ and such that the charts $\{\Phi_{x}\}_{x \in M}$ are uniformly $C^{1}-$ close to $\id$ on the balls of radius $\zeta$ in $E^{u}_{x}$ as $x$ ranges over $M$. We define $L:=L(\zeta, C_0)$ as in \eqref{eqn:csuniform}, which only depends on the geometry of foliations.

Now we fix $x \in M$ and $y \in \mathcal{W}^{cs}_{loc}(x)$. By  Lemma \ref{lemma forwad iterate hcs} we have a family of good $cs-$paths $\{\gamma_n, n\geq 0\}$ such that 
\begin{equation}\label{eqn goodness of gamman}
\gamma_n(0)=f^n(x), \; \gamma_n(1)=f^n(y), \;l(\gamma_n)\leq C_0
\end{equation}
Note that $\Phi_{y}^{-1} \circ h^{cs} \circ \Phi_{x}(0) = 0$. Now consider $r$ such that $0<r\ll \zeta $ and let $v \in E^{u}_{x}$ be any vector with $\|v\| = r$.  By the choice of $A$, there is an integer $n(v) \geq 0$ such that  
\[
\zeta \leq \|Df^{n(v)}(v)\| \leq A\zeta
\] 
If $w \in E^{u}_{x}$ is any other vector with $\|w\| = r$, then by $\kappa$'s definition we get 
$$\kappa^{-1}\zeta\leq \|Df^{n(v)}(w)\|\leq \kappa A\zeta$$

In other words, we can choose $n(v) = n(\|v\|)$ to only depend on the norm of $v$. For definiteness we take $n(\|v\|)$ to be the maximal integer such that all $w$ with  $\|w\|=\|v\|$  satisfy the above inequality. Then by uniformity of the coordinate charts $\Phi_{x}$ we have, 
$$\Phi_{f^n(x)}\circ Df^{n(\|v\|)}(w)\in \cR^u(f^n(x), \kappa^{-2}\zeta, A\kappa^2\zeta)$$ 
Combined with \eqref{eqn:csuniform} and \eqref{eqn goodness of gamman}, recalling that $L = L(\zeta, C_{0})$ and noting that $\gamma_n$ is a good path,  we get that $$h^{cs}_{\gamma_{n}}\circ \Phi_{f^n(x)}\circ Df^{n(\|v\|)}(w) \in  \cR^u(f^n(y), L^{-1}, L)$$

Again using the uniformity of the charts $\Phi$ we conclude that there is a constant $K \geq 1$ independent of $x, y$ and $v$ (as long as $\|v\| = r$) such that 
\[
K^{-1} \leq \left\|\Phi_{f^{n(\|v\|)}(y)}^{-1} \circ h^{cs}_{\gamma_{n(\|v\|)}} \circ \Phi_{f^{n(\|v\|)}(x)} \circ Df^{n(\|v\|)}_{x}(v)\right\| \leq K
\]
By \eqref{eqn uniformy bounded delattion} the linear map $Df^{-n(\|v\|)}$ has dilatation uniformly bounded by $\kappa^2$. Using Lemma \ref{lemma forwad iterate hcs} we conclude for any pair of vectors $v,w \in E^{u}_{x}$ with $\|v\| = \|w\| = r$,  
\[
\frac{\|\Phi_{y}^{-1}\circ h^{cs}\circ\Phi_{x}(v)\|}{\|\Phi_{y}^{-1}\circ h^{cs}\circ\Phi_{x}(w)\|} =\frac{   \left\|Df^{-n(r)}\circ\Phi_{f^{n(r)}(y)}^{-1} \circ h^{cs}_{\gamma_{n(r)}} \circ \Phi_{f^{n(r)}(x)} \circ Df^{n(r)}_{x}(v)\right\|}{\left\|Df^{-n(r)}\circ\Phi_{f^{n(r)}(y)}^{-1} \circ h^{cs}_{\gamma_{n(r)}} \circ \Phi_{f^{n(r)}(x)} \circ Df^{n(r)}_{x}(w)\right\|}\leq K^{2}\kappa^2 
\]
This holds for every positive $r\ll \zeta$ small enough, no matter how small $r$ is. \color{black}
We thus conclude that the linear dilatation of $\Phi_{y}^{-1} \circ h^{cs} \circ \Phi_{x}$ at 0 is bounded above by $K^{2}\kappa^2$ \color{black} for any $x \in M$ and $y \in \mathcal{W}^{cs}_{loc}(x)$. 

To bound the dilatation of $\Phi_{y}^{-1} \circ h^{cs} \circ \Phi_{x}$ at points other than 0 in a ball of bounded radius centered at 0 in $E^{u}_{x}$, we write for $z \in \mathcal{W}^{u}_{loc}(x)$, 
\[
\Phi_{y}^{-1} \circ h^{cs} \circ \Phi_{x} = (\Phi_{y}^{-1} \circ \Phi_{h^{cs}(z)}) \circ (\Phi_{h^{cs}(z)}^{-1} \circ h^{cs} \circ \Phi_{z}) \circ ( \Phi_{z}^{-1} \circ \Phi_{x}) 
\]
The dilatation of $\Phi_{h^{cs}(z)}^{-1} \circ h^{cs} \circ \Phi_{z}$ at 0 is bounded above by $K^{2}\kappa^2$ \color{black}, by our above reasoning. By Proposition \ref{proposition: transitions} the maps $\Phi_{y}^{-1} \circ \Phi_{h^{cs}(z)}$ and  $ \Phi_{z}^{-1} \circ \Phi_{x} $ are both affine maps with linear parts $H^{u}_{y h^{cs}(z)}$ and $H^{u}_{zx}$ respectively. Since we are working on balls of bounded radius centered at $0$ in $E^{u}_{x}$ and $E^{u}_{y}$ respectively and the unstable holonomies are linear maps depending continuously on the base points which are thus a bounded distance from the identity, we conclude that after possibly increasing the constant $K$ the linear dilatation of $\Phi_{y}^{-1} \circ h^{cs} \circ \Phi_{x}$ at $\Phi_{x}^{-1}(z)$ is bounded above by $K^{3}$. This gives us the required quasiconformality assertion of the lemma. 
\end{proof}

\subsection{Absolute continuity of foliations}
\color{black}

We next recall some standard analytic properties of quasiconformal mappings. A homeomorphism $h: U \rightarrow V$ between open domains of $\RR^{k}$ is \emph{absolutely continuous} if it preserves the collection of zero sets of $k$-dimensional Lebesgue measure. There is a natural Lebesgue measure class on the space of affine lines in $\R^{k}$ given by the identification of this space with all translates of lines in $\R^{k}$, i.e., with $\mathbb{R}\mathbb{P}^{k-1} \times \RR^{k}$. Such a homeomorphism is \emph{absolutely continuous on lines} if for each of the coordinate directions $e_{1},\dots,e_{k}$ in $\RR^{k}$ we have that for almost every line $\ell \subset \R^{k}$ parallel to $e_{i}$ the restriction of $h$ to a homeomorphism $\ell \cap U \rightarrow h(\ell \cap U)$ takes subsets of $\ell \cap U$ of 1-dimensional Lebesgue measure zero to zero measure sets of $h(\ell \cap U)$, where $h(\ell \cap U)$ is equipped with the 1-dimensional Hausdorff measure in $\mathbb{R}^{k}$. Here the ``almost everywhere" quantifier on the space of lines parallel to $e_{i}$ (which we identify with $\RR^{k-1}$) is taken with respect to the Lebesgue measure on $\RR^{k-1}$. By Fubini's theorem if $h$ is ACL then $h$ is absolutely continuous. 

Let $\text{vol}_{k}$ denote the standard Lebesgue measure on $\RR^{k}$. For an absolutely continuous homeomorphism $h: U \rightarrow V$ we define the \emph{Jacobian} of $h$ to be the Radon-Nikodym derivative of $h_{*}(\text{vol}_{k})$ with respect to $\text{vol}_{k}$ and denote it by $Jac(h)$. 

We let $\| \cdot \|_{\infty}$ denote the $L^{\infty}$ norm on measurable functions $f: \RR^{k} \rightarrow \RR$,
\[
\|f\|_{\infty} = \inf_{V} \sup_{x \in V} |f(x)|
\]
where the infimum is taken over all measurable subsets $V$ of $\RR^{k}$ with $\text{vol}_{k}(\RR^{k} \backslash V) = 0$. A standard reference for the claims in Proposition \ref{proposition: quasiconformal properties} as well as a more precise discussion of the ACL property is V\"ais\"al\"a's book \cite{V71}. 

\begin{prop}\label{proposition: quasiconformal properties}
Suppose that $h: U \rightarrow V$ is a $K$-quasiconformal homeomorphism between open subsets of $\RR^{k}$, $k \geq 2$. Then $h$ is ACL, differentiable $\text{vol}_{k}$-a.e. in $U$, and we also have $\|Dh_{x}\|_{\infty}\cdot \|(Dh_{x})^{-1}\|_{\infty} \leq K $.
\end{prop}

We next discuss the notion of absolute continuity of a foliation. Let $m$ be a measure on $M$ which is equivalent to the Riemannian volume. Let $\W$ be a $k$-dimensional foliation of an $n$-dimensional Riemannian manifold $M$ which is tangent to a continuous subbundle $E$ of $M$. For each $y \in M$ we let $\W_{r}(y)$ denote the ball of radius $r$ in the induced Riemannian metric on the leaf $\W(y)$ through $y$ which is centered at $y$. Then there is a family of conditional measures $\{m_{x}^{\mathcal{W}}\}_{x \in M}$ of $m$ on the foliation $\mathcal{W}$ with the following properties: for each $x \in M$  we have $m_{x}^{\mathcal{W}}( M \backslash \mathcal{W}(x)) = 0$, the function $x \rightarrow m_{x}^{\W}$ is constant on the leaves of $\W$, and if $S_{x}$ denotes a small $(n-k)$-dimensional disk passing through $x$ and transverse to $\W$ and 
\[
V_{x} := \bigcup_{y \in S_{x}} \W_{r}(y),
\]
denotes an open neighborhood of $x$, then up to scaling $m_{y}^{\W}$ on each local leaf $\W_{r}(y)$ the family $\{m_{y}^{\mathcal{W}}|_{\W_{r}(y)}\}_{y \in S_{x}}$ coincides with the classically defined notion of disintegration of a measure with respect to a measurable partition given by Rokhlin \cite{Ro49}. The family $\{m_{x}^{\mathcal{W}}\}_{x \in M}$ is uniquely defined up to $m$-null sets of $M$ and up to scaling each of the measures on a given leaf of $\W$ by a positive constant. We refer to \cite[Section 3]{AVW} for the proof of the existence and uniqueness of the disintegration claimed in this paragraph. 

For a submanifold $S$ of $M$ we let $\nu^{S}$ be the induced Riemannian volume on $S$ from $M$. We define a $k$-dimensional foliation $\W$ to be \emph{strongly absolutely continuous} if for any pair of nearby smooth transversal $(n-k)$-dimensional submanifolds $S_{1}$ and $S_{2}$ for $\W$ the $\W$-holonomy map $h^{\W}: S_{1} \rightarrow S_{2}$ is absolutely continuous with respect to the measures $\nu^{S_{1}}$ and $\nu^{S_{2}}$, i.e., $h_{*}(\nu^{S_{1}})$ is absolutely continuous with respect to $\nu^{S_{2}}$. Every $C^{1}$ foliation is strongly absolutely continuous. The most important examples of strongly absolutely continuous foliation for purposes are the stable and unstable foliations $\W^{s}$ and $\W^{u}$ of a partially hyperbolic diffeomorphism; strong absolute continuity of these foliations is well-known and a proof may be found in \cite{AV16}.

What we call ``strong absolute continuity" is the notion of absolute continuity used in \cite{AVW}, but this notion of absolute continuity is too strong for our purposes. We define a weaker notion of absolute continuity below, 

\begin{defi}\label{defn: ac}
A foliation $\W$ is  \emph{absolutely continuous} if for each $x \in M$ there is an open neighborhood $V$ of $x$ and a strongly absolutely continuous foliation $\mathcal{F}$ of $V$ transverse to $\W$ such that for any pair of points $y,z \in V$ the $\W$-holonomy map $h^{\W}: \mathcal{F}(y) \rightarrow \mathcal{F}(z)$ is absolutely continuous with respect to the induced Riemannian volumes on  $\mathcal{F}(y)$ and $\mathcal{F}(z)$ respectively. 
\end{defi}

This definition is weaker because we only require the existence of a \emph{particular} foliation $\mathcal{F}$ transverse to $\W$ for which the $\W$-holonomy maps between any pair of leaves are absolutely continuous. We emphasize that the transverse foliation $\mathcal{F}$ need not be smooth in our definition.  

% for $\lambda_{x}^{S}$-a.e. $y \in S_{x}$ the conditional measure $m_{y}^{\W}$ is equivalent to the Riemannian volume $\lambda_{y}^{\W}$ on $\W_{r}(y)$. It is \emph{absolutely continuous with bounded Jacobians} if in addition there is a $C > 0$ such that 
%\[
%\max\left\{\left\|\frac{dm_{y}^{\W}}{d\lambda_{y}^{\W}}\right\|_{\infty}, \left\|\frac{d\lambda_{y}^{\W}}{dm_{y}^{\W}}\right\|_{\infty}\right\} \leq C
%\]
%for every $x \in M$, $y \in S_{x}$. It is not hard to show that the notion of absolute continuity does not change if we use different transversals $S_{x}$ or different neighborhoods inside of the leaves of $\W$. We state a criterion for absolute continuity with bounded Jacobians of a foliation in terms of its holonomy maps between certain transversals, 

Given a foliation $\W$ we say that $m$ has \emph{Lebesgue disintegration along $\W$} if for $m$-a.e. $x \in M$ the conditional measure $m_{x}^{\W}$ on the leaf $\W(x)$ is equivalent to the induced Riemannian volume on $\W(x)$ from $M$. Our definition of absolute continuity is designed such that the following proposition is true, 

%Let $\W$ and $\mathcal{F}$ be two transverse foliations of $M$. Suppose that $\mathcal{F}$ is transversely absolutely continuous with bounded Jacobians and  that for a small $r > 0$ we have that for each $x \in M$, $y \in \mathcal{W}_{r}(x)$ the holonomy map $h_{xy}^{\W}: \mathcal{F}_{r}(x) \rightarrow \mathcal{F}(y)$ is absolutely continuous with $\|Jac(h_{xy}^{\W})\|_{\infty}$ and $\|(Jac(h_{xy}^{\W}))^{-1}\|_{\infty}$  bounded uniformly, independent of $x$ and $y$. Then $\W$ is absolutely continuous with bounded Jacobians. 

\begin{prop}\label{prop: transverse ac}
Suppose $\W$ is an absolutely continuous foliation with respect to a transverse strongly absolutely continuous foliation $\mathcal{F}$. Then $m$ has Lebesgue disintegration along $\W$.
\end{prop}
\begin{proof}
Fix a point $x \in M$ and let $V$ be an open neighborhood of $x$ on which there is a strongly absolutely continuous foliation $\mathcal{F}$ transverse to $\W$ for which the $\W$-holonomy maps between any two $\mathcal{F}$-leaves are absolutely continuous. In the case that $\mathcal{F}$ is a $C^{1}$ foliation, the proof that the conclusion of the proposition holds is given by \cite[Proposition 6.2.2]{BS02}. However the only property of the transversal foliation $\mathcal{F}$ which is used in that proof is the strong absolute continuity, for completeness we give a detailed proof using only this strong absolute continuity property. 

Without loss of generality we assume that $V$ has a local product structure, i.e. for any $x',x''\in V$, the local leaves $\mathcal{F}_{loc}(x')\cap V$ and $\W_{loc}(x'')\cap V$ intersect at exactly one point in $V$. If we denote $\mathcal{F}_{loc}(x')\cap V, \mathcal{W}_{loc}(x'')\cap V$ by $\mathcal{F}_V(x')$ and $\mathcal{W}_V(x'')$ respectively for any $x',x''\in V$, then we have 
\begin{equation}
V=\cup_{y\in \W_V(x)}\mathcal{F}_V(y)=\cup_{s\in \mathcal{F}_V(x)}\W_V(s)
\end{equation}

Since $\mathcal{F}$ is a strongly absolutely continuous foliation, there exists a positive measurable conditional density $\delta_y(\cdot)$ for $\nu^{\W_{V}(x)}$-almost every $y\in \W_V(x)$ such that for any measurable subset $A\subset V$ we have 
\begin{equation}\label{m(A)}
m(A)=\int_{\W_V(x)}\int_{\mathcal{F}_V(y)}\mathbb{1}_A(y,z)\delta_y(z)d\nu^{\mathcal{F}_V(y)}(z)d\nu^{\W_{V}(x)}(y)
\end{equation} 
where we recall from above that $\nu^S$ denotes the induced Riemannian volume on the submanifold $S$.

Let $p_y(\cdot)$ denote the holonomy maps along the leaves of $\W$ from $\mathcal {F}_V(x)$ to $\mathcal {F}_V(y)$, and let $q_y(\cdot)$ denote the Jacobian of $p_y$. We have 
\begin{equation}
\int_{\mathcal{F}_V(y)}\mathbb{1}_A(y,z)\delta_y(z)d\nu^{F_V(y)}(z)=\int_{F_V(x)}\mathbb{1}_A(p_y(s))\delta_y(p_y(s))q_y(s)d\nu^{F_V(x)}(s),
\end{equation}
and by changing the order of integration in \eqref{m(A)} we get 
\begin{equation}\label{m(A) change order}
m(A)=\int_{\mathcal{F}_V(x)}\int_{\W_V(x)}\mathbb{1}_A(p_y(s))\delta_y(p_y(s)) q_y(s)d\nu^{\W_V(x)}(y) d\nu^{F_V(x)}(s).
\end{equation} 

Let $\bar{p}_s(\cdot)$ denote the holonomy map along the leaves of $\mathcal{F}_V$ from $\W_V(s)$ to $\W_V(x)$. Since $\mathcal{F}$ is a strongly absolutely continuous foliation the map $\bar{p}_s(\cdot)$ is absolutely continuous and thus admits a Jacobian $\bar{q}_s$ with respect to the induced volumes on $\W_V(s)$ and $\W_V(x)$ respectively. 

We transform the inner integral in \eqref{m(A) change order} into an integral over $\W_V(s)$ by making the change of variables $r=p_y(s)$. Note that $y = y(r)$ is uniquely determined by $r$ and is continuous as a function of $r$. Therefore we have
\begin{eqnarray}
&&\int_{\W_V(x)}\mathbb{1}_A(p_y(s))q_y(s)\delta_y(p_y(s)) d\nu^{\W_V(x)}(y) \nonumber\\
&=&\int_{\W_V(s)}\mathbb{1}_A(r)q_{y(r)}(s)\delta_{y(r)}(r) \bar{q}_s(r)d\nu^{\W_V(s)}(r).
\end{eqnarray} 
Combining this with  \eqref{m(A) change order} we get 
\begin{equation}
m(A)=\int_{\mathcal{F}_V(x)}\int_{\W_V(s)}\mathbb{1}_A(s,r)q_{y(r)}(s)\delta_{y(r)}(r)\bar{q}_s(r)d\nu^{\W_V(s)}(r)d\nu^{F_V(x)}(s),
\end{equation}
which implies the statement of the proposition.
\end{proof}

By appealing to the equations \eqref{m(A)} and \eqref{m(A) change order} derived in the proof of Proposition \ref{prop: transverse ac} we obtain the following corollary. For $x \in M$ we continue to let $V_{x}$ denote an open subset of $M$ containing $x$ on which the combination of the two transverse foliations $\mathcal{W}$ and $\mathcal{F}$ has local product structure. For two measures $\mu$ and $\nu$ we write $\mu \asymp \nu$ if these two measures are equivalent, i.e., they have the same null sets. We will need the following equivalences later, 

\begin{cor}\label{mutual continuity}
Suppose that $\mathcal{W}$ is an absolutely continuous foliation with respect to a transverse strongly absolutely continuous foliation $\mathcal{F}$. Then any $x \in M$ we have the equivalence of measures on $V_{x}$, 
\[
m \asymp \int_{\mathcal{W}_{V}(x)} \nu^{\mathcal{F}_{V}(y)} \, d\nu^{\mathcal{W}_{V}(x)}(y) \asymp \int_{\mathcal{F}_{V}(x)} \nu^{\mathcal{W}_{V}(z)} \, d\nu^{\mathcal{F}_{V}(x)}(z)  
\]
\end{cor}

Combining Proposition \ref{prop: transverse ac} with our work above leads to the following conclusions,

 \begin{cor}\label{corollary: cs absolute continuity}
Let $f$ be a $C^{\infty}$ dynamically coherent, center non-expansive partially hyperbolic diffeomorphism. 
\begin{enumerate}
\item If $f$ is uniformly $u$-quasiconformal, then $m$ has Lebesgue disintegration along $\W^{cs}$-leaves. 

\item
If $f$ is uniformly quasiconformal, then for $* \in \{cs,c,cu\}$, $m$ has Lebesgue disintegration along $\W^{*}$ leaves.
\end{enumerate}
\end{cor}
\color{black}
\begin{proof}
By Lemma \ref{lemma: quasiconformal} and Proposition \ref{proposition: quasiconformal properties} the hypotheses of (1) imply that the local $\W^{cs}$-holonomy maps between local leaves of the unstable foliation $\W^{u}$ are absolutely continuous. Since the stable and unstable foliations $\W^{s}$ and $\W^{u}$ are strongly absolutely continuous this implies by Proposition \ref{prop: transverse ac} that $m$ has Lebesgue disintegration along $\W^{cs}$ leaves. 

Under the hypotheses of (2) we may also apply Lemma \ref{lemma: quasiconformal} to the local $cs$-holonomies of $f^{-1}$, i.e., to the $cu$-holonomies of $f$. This implies that the local $\W^{cu}$-holonomy maps between local leaves of the unstable foliation $\W^{s}$ are absolutely continuous with bounded Jacobians as well, and thus by Proposition \ref{prop: transverse ac} (using $\W^{s}$ as the transverse strongly absolutely continuous foliation) we conclude that $m$ also has Lebesgue disintegration along $\W^{cu}$ leaves. 

For each $x \in M$ the $cu$-leaf $\W^{cu}(x)$ is foliated by $\W^{u}$-leaves and this foliation is strongly absolutely continuous when we consider $\W^{cu}(x)$ to be the ambient manifold. The holonomy of the $\W^{c}$ foliation between local $\W^{u}$ leaves inside of $\W^{cu}(x)$ coincides with the $\W^{cs}$ holonomy in $M$. Thus by Lemma \ref{lemma: quasiconformal} and Proposition \ref{proposition: quasiconformal properties} these local $\W^{c}$-holonomy maps are absolutely continuous. 

Let $m^{cu}_{x}$ denote the conditional volume of $m$ on $\W^{cu}(x)$. Since the holonomy maps of $\W^{c}$ between local $\W^{u}$ leaves inside of $\W^{cu}(x)$ are absolutely continuous we can apply Proposition \ref{prop: transverse ac} again to obtain that $m^{cu}_{x}$ has Lebesgue disintegration along $\W^{c}$ leaves inside of $\W^{cu}(x)$. Since this holds for every $x \in M$ and $m$ has Lebesgue disintegration along $\W^{cu}$ leaves it follows that $m$ has Lebesgue disintegration along $\W^{c}$ leaves. 
\end{proof}

%Notice that locally we can choose the conditional densities of the volume $m$ on $\W^{s(u)}$ to be continuous and positive. Moreover, the Jacobian of local (un)stable holonomy between $\W^{cu(cs)}$ are continuous depends on the points and the choice of $\W^{cu(cs)}$. (we can easily checked by hand using the traditional proof). 

%Suppose the Jacobian of $h^{cu(cs)}$ between (un)stable manifolds are uniformly bounded (or continuous), then repeat classical argument of "transverse a.c ness implies a.c.ness (see the book "Introduction to dynamical systems" for example)", we get the proof of the lemma for $\W^{cu(cs)}$.

%Since the conditional densities of volume $\lambda_{\W^{cs(cu)}}$ on $\W^{s(u)}$ leaves and Jacobian of (un)stable holonomy between center leaves have the same continuity as the conditional density and holonomy of $\W^s(u)$ on the whole manifold $M$ (and depends on the choice of $\W^{cs(cu)}$ continuously). We can repeat the proof above to conclude the lemma for $\W^c$.

%\begin{rema}\label{rem: lower reg}
%All of the results in this section (and some of the results from Section \ref{section: higher} as well) apply with much lower regularity hypotheses on $f$. It suffices to assume for these results that $f$ is only $C^{2}$. One must check only that the linearization proposition \cite[Proposition 4.1]{S05} applies in $C^{2}$ regularity as well as $C^{\infty}$ regularity, which is not done here for sake of space. 
%\end{rema}

\section{Linearity of the center holonomy}\label{section: higher}In this section we will prove the center holonomy between local unstable leaves is linear in the charts $\{\Phi_{x}\}_{x \in M}$ under stronger assumptions on $f$, i.e. Lemma \ref{lemma: linear center}. To obtain it,  we will show that the a.e. defined map $Dh^c$ is equivariant with respect to $H^u$ holonomy almost everywhere (Lemma \ref{lemma: L holo eqva}).  Subsection \ref{subsec: fiber} is for the construction of a fiber bundle which is critical to us for studying the center holonomy through measure-theoretic arguments. This is where we use the hypothesis in Theorem \ref{theorem: compact center} and Corollary \ref{theorem: isometric center} that the center foliation has trivial holonomy. In subsection \ref{subsec: confstruct} we construct an invariant bounded measurable conformal structure which is invariant under $H^{s,u}$ holonomies. It plays an important role in Section \ref{section: higherend}. In subsection \ref{subsec: equi} we prove the main results Lemmas \ref{lemma: L holo eqva} ND  \ref{lemma: linear center}. 

\color{black} Unless stated otherwise, in all of the claims of this section we assume that $f$ is a $C^\infty$ dynamically coherent partially
hyperbolic diffeomorphism which is uniformly quasiconformal. We further assume that $f$ preserves an invariant measure $m$ which is smoothly equivalent to the volume on $M$.

\subsection{A fiber bundle construction}\label{subsec: fiber}

We first formulate an additional condition on $f$ which is related to the proof of Theorem \ref{theorem: flow suspension}.

Suppose that $\dim E^{c} = 1$, $E^{c}$ is orientable, $f(\mathcal{W}^{c}(x)) = \mathcal{W}^{c}(x)$ for every $x \in M$, and $f$ has no fixed points. Each center leaf $\mathcal{W}^{c}(x)$ has as its universal cover a copy $\widetilde{\mathcal{W}}^{c}(x)$ of $\mathbb{R}$ with orientation determined by the orientation of $E^{c}$. The restriction of $f$ to $\mathcal{W}^{c}(x)$ lifts to an orientation-preserving diffeomorphism $\widetilde{f}$ of $\widetilde{\mathcal{W}}^{c}(x)$ with no fixed points. Fix a lift $\widetilde{x}$ of $x$ and let $\widetilde{U}_{x}$ be the closed segment joining $\widetilde{f}^{-1}(\widetilde{x})$ to $\widetilde{f}(\widetilde{x})$ inside of $\widetilde{\mathcal{W}}^{c}(x)$. We then let $U_{x}$ be the projection of this segment to $\mathcal{W}^{c}(x)$. An easy exercise shows that the neighborhood $U_{x}$ is independent of the chosen lift of $x$. It is also clear that for every $x \in M$ we have $f(U_{x}) = U_{f(x)}$. 

We say that $f$ \emph{does not wrap} if for each $x \in M$, the neighborhood $U_{x}$ of $x$ in $\mathcal{W}^{c}(x)$ is simply connected, i.e., it is a line segment instead of a circle. 

For the remainder of Section \ref{section: higher} and Section \ref{section: higherend} we will assume that $f$ satisfies one of the following two assumptions,

\renewcommand{\theenumi}{\Alph{enumi}}
\begin{enumerate}
\item $\mathcal{W}^{c}$ is compact and each center leaf with trivial holonomy; or\color{black}

\item $\dim E^c = 1$, $f(\W^c(x)) = \W^c(x)$ for every $x\in M$, $E^c$ is orientable with
orientation preserved by $f$, $f$ has no fixed points, and $f$ does not wrap.  
\end{enumerate}
\renewcommand{\theenumi}{\arabic{enumi}}

In the case that $f$ satisfies assumption (B) we set $\{U_{x}\}_{x \in M}$ to be the family of neighborhoods of points of $M$ inside of the center foliation constructed above. When $f$ satisfies assumption (A) we instead set $U_{x} = \W^{c}(x)$. In both cases we have the properties that $f(U_{x}) = U_{f(x)}$. Without loss of generality (decreasing $r,R$ in subsection \ref{subsection: quasiconfomal of hc} if necessary), we assume for any $x\in M$, $\mathcal{W}^c_{loc}(x)\subset U_x$.\color{black} 

We note that if $\psi_{1}$ is the time-1 map of an Anosov flow $\psi_{t}$ which has no periodic orbits of period $\leq 2$ then assumption (B) always holds for $C^{1}$-small enough  perturbations of $\psi_{1}$. We refer to the proof of Theorem \ref{theorem: flow suspension} for the details behind this assertion. 

%Aside from the assumptions that $\dim E^{c} = 1$ and that $f(\mathcal{W}^{c}(x)) = \mathcal{W}^{c}(x)$ for every $x \in M$, the assumptions in (2) are likely unnecessary in what follows and are purely intended to eliminate some potential technical complications. As we will see in Section \ref{section: proofs} these assumptions suffice for proving Theorems \ref{theorem: flow suspension} and \ref{theorem: hyperbolic perturbation}. 

%Recall that for any $x\in M$ and $y\in \W^c_{loc}(x)$ we could define $h^c_{xy}$ and $H^c_{xy}=Dh^c_{xy}:E^u_x\to E^u_y$. Moreover under assumption (B), $h^c_{xy}$ and $H^c_{xy}$ can be extended to $\{(x,y)\in M^2: y\in U_x\}$ since $f$ does not wrap. In general for $y\in \W^c(x)$ without specific assumption on $x,y$, $h^c_{xy}$ and $H^c_{xy}$ are not well-defined unless the group $Hol(\W^c(x),x)$ is trivial.
\begin{prop}\label{prop: nonexpansive}
Suppose $f$ satisfies one of assumptions (A) and (B), then $f$ is center non-expansive.
\end{prop}

\begin{proof}
Under assumption (A), since each center leaf is compact and HAS trivial holonomy,  there is a uniform bound on the diameter of center leaves (cf. \cite{BB}). Then for any $x\in M, y\in \W^c_{loc}(x), n\in \ZZ$ and any good local $c-$path $\gamma$ from $x$ to $y$, we take a geodesic (within the center leaf) $\gamma_n$ from $f^n(x)$ to $f^n(y)$ contained inside $\W^c(f^n(x))$. Then $l(\gamma_n)$ is uniformly bounded and the path $f^n(\gamma)\cdot \gamma_n^{-1}$ represents the identity in $Hol(\W^c(f^n(x)), f^n(x))$. We conclude that $f$ satisfies the conditions of center non-expansiveness.

In the case of assumption (B), for any $x\in M, y\in \W^c_{loc}(x), n\in \ZZ$ and any good local $c-$path $\gamma$ from $x$ to $y$, we have by our assumption that $$f^n(\gamma)\subset f^n(\W^c_{loc}(x))\subset f^n(U_x)=U_{f^n(x)}$$
By uniformity and simple connectedness of $\{U_x, x\in M\}$, we can easily find a $C^1-$path $\gamma_n\subset U_{f^n(x)}$ from $f^n(x)$ to $f^n(y)$ such that $f^n(\gamma)\cdot \gamma_n^{-1}$ represents the identity in $Hol(\W^c(x),x)$ and $l(\gamma_n)\leq\sup_{x\in M}l(U_x)<\infty$. 
\end{proof}
\color{black}
As a consequence of Proposition \ref{prop: nonexpansive}, all of the work from Section \ref{section: center holonomy} applies to both $f$ and $f^{-1}$ for systems satisfying assumptions (A) or (B). In particular the center-stable holonomy maps $h^{cs}_{xy}: \mathcal{W}_{loc}^{u}(x) \rightarrow \mathcal{W}^{u}_{loc}(y)$ and center-unstable holonomy maps $h^{cu}_{xz}: \mathcal{W}_{loc}^{s}(x) \rightarrow \mathcal{W}^{s}_{loc}(z)$ for $x \in M$, $y \in \mathcal{W}^{cs}_{loc}(x)$, $z \in \W^{cu}_{loc}(x)$ are all $K$-quasiconformal for some constant $K \geq 1$. Consequently $m$ has Lebesgue disintegration along each  of the foliations $\W^{cs}$, $\W^{cu}$, and $\W^{c}$ by Corollary \ref{corollary: cs absolute continuity}. By combining this with the strong absolute continuity of the foliations $\mathcal{W}^{u}$ and $\mathcal{W}^{s}$ we conclude that $m$ has Lebesgue disintegration along all of the invariant foliations $\mathcal{W}^{*}$ for $f$. 

We now consider the space 
\[
\E=\{(x,y)\in M^2:  y\in U_{x} \},
\]
and we define $F: \E \rightarrow \E$ by $F(x,y) = (f(x),f(y))$. 

\begin{prop}\label{prop: fiber measure}
$\E$ is a continuous fiber bundle over $M$ with compact fibers.  $F$ preserves an invariant measure $\mu$ on $\E$ which locally decomposes as the product of the volume $m$ on $M$ and the conditional volume $m^{c}_{x}$ on $U_{x}$. 
\end{prop}

\begin{proof}
We first show that for $r > 0$ sufficently small and $y \in \W^{u}_{r}(x)$, under either assumption (A) or (B), the unstable holonomy map $h^{u}_{xy}$ from $\W^{c}_{loc}(x)$ to $\W^{c}_{loc}(y)$ can be extended to a homeomorphism $h^{u}: U_{x} \rightarrow U_{y}$. Under either assumption (A) or (B) the neighborhoods $U_{x} \subset \W^{c}(x)$ are uniformly compact (i.e., have $d_{c}$-diameter uniformly bounded in $x$) and depend continuously on $x$ in the Hausdorff topology on sets. Thus we can choose $r$ small enough that for any $x \in M$, $y \in \W^{u}_{r}(x)$ we have that $\W^{u}_{loc}(z) \cap U_{y}$ consists of at most one point for any $z \in U_{x}$. Under assumption (A) the last assertion requires the condition each center leaf has trivial holonomy, under assumption (B) this last assertion requires the condition that $f$ does not wrap.  

We claim that there is in fact exactly one point in $\W^{u}_{loc}(z) \cap U_{y}$. This is obvious under assumption (A) since $\W^{u}_{loc}(z)$ must intersect $\W^{c}(y)$ if $y$ is close enough to $x$, by the uniform compactness of the center foliation. Under assumption (B) let $\varepsilon > 0$ be a constant chosen small enough that for each $x \in M$ the $\varepsilon$-neighborhood $U_{x}^{\varepsilon}$ of $U_{x}$ inside of $\W^{c}(x)$ still satisfies the no wrapping condition, that is to say, $U_{x}^{\varepsilon}$ remains an interval instead of a circle.  Then for $\varepsilon > 0$ sufficiently small, $r$ sufficiently small depending of $\varepsilon$, and any $x \in M$, $y \in \W^{u}_{r}(x)$ and $z \in U_{x}$ the intersection $\W^{u}_{loc}(z) \cap U_{y}^{\varepsilon}$ consists of exactly one point. Thus $h^{u}: U_{x} \rightarrow  U_{y}^{\varepsilon}$ is an orientation-preserving homeomorphism onto its image inside of $U_{y}^{\varepsilon}$. But if $y \in \W^{u}_{r}(x)$ for $r$ sufficiently small then $f(y)$ and $f^{-1}(y)$ lie in $\W^{u}_{loc}(f(x))$ and $\W^{u}_{loc}(f^{-1}(x))$ respectively. Thus the endpoints of the interval which is the image of $U_{x}$ under $h^{u}$ are $f^{-1}(y)$ and $f(y)$ (Since $f(y)\in U^\varepsilon_y\cap \W^u_{loc}(f(x))$ and $\W^{u}_{loc}(z) \cap U_{y}^{\varepsilon}$ consists of exactly one point, then $h^u(f(x))\in U^\varepsilon_y\cap \W^u_{loc}(f(x))=f(y)$, similarly we got $h^u(f^{-1}(x))=\{f^{-1}(y)\}$). \color{black} This shows that $h^{u}$ actually gives a homeomorphism from $U_{x}$ to $U_{y}$. 

Hence $h^{u}: U_{x} \rightarrow U_{y}$ is a homeomorphism for $y \in \W^{u}_{r}(x)$. By similar reasoning (possibly taking $r$ smaller) for any $x \in M$, $y \in \W^{s}_{r}(x)$ the stable holonomy $h^{s}: U_{x} \rightarrow U_{y}$ is also a homeomorphism. Finally, it is easy to see that for any $y \in \W^{c}_{r}(x)$ there is a homeomorphism $h^{c}: U_{x} \rightarrow U_{y}$ depending uniformly continuously on the pair $(x,y)$: in the case of assumption (A) this is trivial since $U_{x} = U_{y}$. In the case of assumption (B) each of the subsets $U_{y}$ of $\W^{c}(x)$ is an interval determined canonically by its endpoints $f^{-1}(y)$ and $f(y)$ according to the construction at the beginning of this section. These endpoints depend continuously on $y \in \W^{c}(x)$ hence it follows that we can find a continuous family of orientation-preserving homeomorphisms $h^{c}: U_{y} \rightarrow U_{x}$ identifying these intervals for $y$ near $x$. Putting all of this together, for any $z$ close enough to $x$ we can find a composition of three homeomorphisms 
\[
U_{z} \rightarrow U_{h^{s}(z)} \rightarrow U_{h^{u}(h^{s}(z))} \rightarrow U_{h^{c}(h^{u}(h^{s}(z)))} = U_{x}
\]
which depends continuously on $z$, where $h^{s}(z) = \W^{s}_{loc}(z) \cap \W^{cu}_{loc}(x)$, $h^{u}(h^{s}(z)) = \W^{u}_{loc}(h^{s}(z)) \cap \W^{c}_{loc}(x)$. This proves that $\E$ is a continuous fiber bundle over $M$ with compact fibers. 

We now prove the second assertion. Under assumption (A) we consider the measurable partition of $M$ into compact center fibers $M = \bigcup_{x \in M} \W^{c}(x)$ and let $\{m^{c}_{x}\}_{x \in M}$ be the family of conditional measures of $m$ on the center fibers $\W^{c}(x)$ determined by this partition. Since $f$ preserves $m$ we have $f_{*}m_{x}^{c} = m_{f(x)}^{c}$. for $m$-a.e. $x \in M$.

In the case of assumption (B) we refer to \cite[Section 3]{AVW}. It is shown there that for the foliation $\W^{c}$ of $M$ there is a measurable family of conditional measures $\{\hat{m}^{c}_{x}\}_{x \in M}$ supported on the leaves $\W^{c}(x)$ of the center foliation such that for $y \in \W^{c}(x)$ the measures $\hat{m}^{c}_{x}$ and $\hat{m}^{c}_{y}$ coincide up to a constant factor. We normalize these measures such that $\hat{m}^{c}_{x}(U_{x}) = 1$. Furthermore, since $f$ fixes all of the leaves of $\W^{c}$, by \cite[Proposition 3.3]{AVW} we have $f_{*}\hat{m}^{c}_{x} = \hat{m}^{c}_{x}$ for $m$-a.e. $x \in M$. We define $m_{x}^{c}$ to be the restriction of $\hat{m}^{c}_{x}$ to $U_{x}$. Since $f(U_{x}) = U_{f(x)}$ we conclude that $f_{*}m_{x}^{c}$ is a constant multiple of $m_{f(x)}^{c}$, and by our choice of normalization of the measures $\hat{m}^{c}_{x}$ this implies $f_{*}m_{x}^{c} = m_{f(x)}^{c}$ since these measures both assign mass 1 to $U_{f(x)}$. 

The measures $m_{x}^{c}$ are equivalent to the Riemannian volume on $U_{x}$ since $m$ has Lebesgue disintegration along $\W^{c}$ leaves. We define the measure $\mu$ on the fiber bundle $\E$ by setting, for any measurable set $A\subset \E$, 
\[
\mu(A)=\int \mathbb1_A(x,y)dm_x^c(y) dm(x),
\]
where $\mathbb1_{A}$ denotes the characteristic function of $A$. Since $f$ preserves $m$ and $f_{*}m_{x}^{c} = m_{f(x)}^{c}$ for $m$-a.e. $x \in M$, we conclude that $\mu$ is $F-$invariant. 
\end{proof}

\subsection{Conformal structures}\label{subsec: confstruct}
We now introduce the bundle $\mathcal{C}E^{u}$ of conformal structures on $E^{u}$ over $M$.  For more details related to the discussion that follows we defer to \cite{KS10}. The fiber $\mathcal{C}E^{u}_{x}$ over $x$ is the space of all inner products on $E^{u}_{x}$ modulo scaling by a nonzero real number, which can be identified with the nonpositively curved Riemannian symmetric space $SL(k,\mathbb{R})/SO(k,\mathbb{R})$. Each fiber thus carries a canonical Riemannian metric $\rho_{x}$ given by an isometric identification of $\mathcal{C}E^{u}_{x}$ with $SL(k,\mathbb{R})/SO(k,\mathbb{R})$. We will always explicitly identify $\mathcal{C}E^{u}_{x}$ with the space of inner products on $E^{u}_{x}$ for which the determinant of a positively oriented orthonormal basis is 1 in the reference inner product on $E^{u}_{x}$ induced from the given Riemannian metric on $TM$. 

Any linear isomorphism $A: E^{u}_{x} \rightarrow E^{u}_{y}$ induces a map $A^{*}: \mathcal{C}E^{u}_{y} \rightarrow \mathcal{C}E^{u}_{x}$ by, for $\tau_{x} \in \mathcal{C}E^{u}_{x}$ and any $v,w \in E^{u}_{y}$, 
\[
A^{*}\tau_{x}(v,w) = \frac{\tau_{x}(A(v), A(w))}{\det(A)^{2/k}},
\]
where we recall that $k = \dim E^{u}$ and $\det(A)$ denotes the determinant of $A$ in the metric induced from $TM$. The induced map $A^{*}$ is an isometry from $(\mathcal{C}E^{u}_{y},\rho_{y})$ to  $(\mathcal{C}E^{u}_{x},\rho_{x})$. 

A (measurable) \emph{conformal structure} on $E^{u}$ is a measurable section $\tau: M \rightarrow \mathcal{C}E^{u}$ defined on the complement of an $m$-null set of $M$. We say that a measurable conformal structure is \emph{invariant} if $Df_{x}^{*}\tau_{f(x)} = \tau_{x}$ for $m$-a.e. $x \in M$. A measurable conformal structure is \emph{bounded} if there is a constant $C  > 0$ such that $\rho_{x}(\tau_{x},\tau^{0}_{x}) \leq C$ for $m$-a.e. $x \in M$, where $\tau^{0}$ denotes the conformal structure on $E^{u}$ induced from the Riemannian metric on $TM$. The condition that $f$ is uniformly $u$-quasiconformal is equivalent to the existence of a constant $C > 0$ such that for every $x \in M$, 
\[
\rho_{x}\left((Df^{n}_{x})^{*}\tau_{f^{n}(x)}^{0}, \tau_{x}^{0}\right) \leq C \; \; \forall n \in \Z.
\]

The following measure-theoretic lemma is vital for recovering holonomy invariance properties of measurable objects by guaranteeing simultaneous recurrence to a continuity set on a full measure set of points. Our first application will be to show that a measurable invariant conformal structure for $f$ must be invariant under the stable and unstable holonomies $H^{s}$ and $H^{u}$ on a full measure subset of $M$. 

\begin{lemma}\label{lemma: half}
Let $T$ be a measure-preserving transformation of a finite measure space $(X,\mu)$ and let $\{K_{n}\}_{n \geq 1}$ be a sequence of measurable subsets of $X$ with $\sum_{n=1}^{\infty} \mu(X \backslash K_{n})  < \infty$. Then there is a full measure subset $\Omega \subseteq X$ with the property that if $x,y \in \Omega$ then there is an $n \in \N$ and a sequence $n_{k} \rightarrow \infty$ with $T^{n_{k}}(x) \in K_{n}, T^{n_{k}}(y) \in K_{n}$ for each $n_{k}$. 
\end{lemma}

\begin{proof}
By the Birkhoff ergodic theorem, for each $n \in \N$ the Birkhoff averages $\frac{1}{k}\sum_{j=0}^{k-1}\mathbb1_{K_n}(T^{j}(x))$ converge pointwise (on a measurable set $E_{n}$ with $\mu(X \backslash E_{n}) = 0$) as $k \rightarrow \infty$ to a nonnegative $T-$ invariant measurable function $P_{n}$ with integral $\mu(K_n)$. Define $\Omega \subset X$ by 
\[
\Omega =\bigcap_{n=1}^{\infty} E_{n} \cap \left\{x \in X: \exists N \in \N \; \text{such that for $n \geq N$,} \; P_{n}(x) >\frac{3}{4}  \right\}.
\]
We claim that $\mu(X \backslash \Omega ) = 0$. Consider the sets
\[
B_{n} = \left\{x\in E_{n} : 1 - P_{n}(x) \geq \frac{1}{4}  \right\}.
\]
By the Markov inequality we have
\[
\mu(B_n)   \leq 4\int_{X} 1- P_{n}\, d\mu   = 4\mu(X \backslash K_{n}) 
\] 
Since $\sum_{n=1}^{\infty} \mu(X \backslash K_{n}) < \infty$ by hypothesis, we conclude by the Borel-Cantelli lemma that for $\mu$-a.e. $x \in X$ there are only finitely many $n$ such that $x \in B_{n}$. This implies that $\Omega$ is a full measure subset of $X$. 

Now we verify that $\Omega$ has the desired properties of the Lemma's conclusion. If $x,y \in \Omega$ are any two given points then from the definition of $\Omega$ there is a common $n$ such that $P_{n}(x) > \frac{3}{4}$ and $P_{n}(y) > \frac{3}{4}$. We explain how to construct $n_{k}$ inductively from $n_{k-1}$, with our construction also showing how to construct the initial $n_{1}$ from $n_{0} = 0$. Given $n_{k-1}$, we choose $N_{k} > n_{k-1}$ large enough that 
\[
\frac{1}{N_{k}}\sum_{j=n_{k-1}}^{N_{k}}\mathbb1_{K_n}(T^{j}(x)) > \frac{3}{4}
\]
and the same for $y$, and we also take $N_{k}$ large enough that $\frac{N_{k}-n_{k-1}}{N_{k}} > \frac{7}{8}$. The existence of this $N_{k}$ is guaranteed by the fact that $P_{n}(x) > \frac{3}{4}$. We conclude from these estimates that 
\[
\frac{\left|\left\{j \in [n_{k-1},N_{k}]: T^{j}(x)\in K_{n}\right\}\right|}{N_{k}-n_{k-1}} > \frac{4}{7}
\]
and the same for $y$, where $[n_{k-1},N_{k}]$ denotes the set of integers $j$ satisfying $n_{k-1} \leq j \leq N_{k}$.  It follows that there is a common $n_{k} \in [n_{k-1},N_{k}]$ such that $T^{n_{k}}(x) \in K_{n}$ and $T^{n_{k}}(y) \in K_{n}$. Inducting on $k$ completes the proof. 
\end{proof}

Proposition \ref{prop: confstruct} below summarizes the essential properties of invariant conformal structures for uniformly quasiconformal linear cocycles which we will need. It is a slight improvement of the proof of \cite[Proposition 4.4]{KS13} as it removes the assumption of ergodicity of $f$ with respect to the volume $m$ which was used in that proof.

\begin{prop}\label{prop: confstruct}
Suppose that $f$ is uniformly $u$-quasiconformal and volume-preserving. Then there is an invariant bounded measurable conformal structure $\tau: M \rightarrow \mathcal{C}E^{u}$. Furthermore there is a full measure subset $\Omega$ of $M$ such that if $x,y \in \Omega$ and $y \in \W^{u}_{loc}(x)$ then 
\[
(H^{u}_{xy})^{*}\tau_{y} = \tau_{x},
\]
and similarly if $y \in \W^{s}_{loc}(x)$ then 
\[
(H^{s}_{xy})^{*}\tau_{y} = \tau_{x}.
\]
\end{prop}

\begin{proof}
By \cite[Proposition 2.4]{KS10} the uniform quasiconformality of the linear cocycle $Df|_{E^{u}}$ implies that there is an invariant bounded measurable conformal structure $\tau: M \rightarrow \mathcal{C} E^{u}$. 

By Lusin's theorem we can find an increasing sequence $\{K_{n}\}_{n \in \N}$ of compact subsets of $M$ such that $\tau$ is uniformly continuous on $K_{n}$ and $m(M \backslash K_{n}) < 2^{-n}$. Let $\Omega \subset M$ be the full measure set of points satisfying the conclusion of Lemma \ref{lemma: half} for both $f$ and $f^{-1}$. We also require that $\tau$ is defined and $Df$-invariant on $\Omega$. 

Let $x,y \in \Omega$ be given with $y \in \W^{s}_{loc}(x)$. Then there is an $n > 0$ and a sequence $n_{k} \rightarrow \infty$ such that $f^{n_{k}}(x)$ and $f^{n_{k}}(y)$ both lie in $K_{n}$ for each $n_{k}$.  Then 
\begin{align*}
\rho_{x}(\tau_{x}, (H^{s}_{xy})^{*}\tau_{y}) &= \rho_{x}((Df^{n_{k}}_{x})^{*}\tau_{f^{n_{k}}(x)}, (H^{s}_{xy})^{*}(Df^{n_{k}}_{y})^{*}\tau_{f^{n_{k}}(y)})  \\
&= \rho_{x}((Df^{n_{k}}_{x})^{*}\tau_{f^{n_{k}}(x)}, (Df^{n_{k}}_{x})^{*}(H^{s}_{f^{n_{k}}x f^{n_{k}}y})^{*}\tau_{f^{n_{k}}(y)}) \\
&=  \rho_{f^{n_{k}}(x)}(\tau_{f^{n_{k}}(x)}, (H^{s}_{f^{n_{k}}x f^{n_{k}}y})^{*}\tau_{f^{n_{k}}(y)}) \\
&\leq \rho_{f^{n_{k}}(x)}(\tau_{f^{n_{k}}(x)}, (I_{f^{n_{k}}x f^{n_{k}}y})^{*}\tau_{f^{n_{k}}(y)})  \\
&+\rho_{f^{n_{k}}(x)}((I_{f^{n_{k}}x f^{n_{k}}y})^{*}\tau_{f^{n_{k}}(y)}, (H^{s}_{f^{n_{k}}x f^{n_{k}}y})^{*}\tau_{f^{n_{k}}(y)}),
\end{align*}
where we recall that $I_{xy}: E^{u}_{x} \rightarrow E^{u}_{y}$ is our chosen H\"older continuous family of identifications of nearby fibers of $E^{u}$. The uniform continuity of $\tau$ on $K_{n}$ implies that 
\[
\rho_{f^{n_{k}}(x)}(\tau_{f^{n_{k}}(x)}, (I_{f^{n_{k}}x f^{n_{k}}y})^{*}\tau_{f^{n_{k}}(y)}) \rightarrow 0,
\] 
as $n_{k} \rightarrow \infty$ since $d(f^{n_{k}}(x), f^{n_{k}}(y)) \rightarrow 0$ as $n_{k} \rightarrow \infty$.  Since $\| I_{f^{n_{k}}x f^{n_{k}}y} - H^{s}_{f^{n_{k}}x f^{n_{k}}y}\| \rightarrow 0$ uniformly as $n_{k} \rightarrow \infty$ we also conclude that 
\[
\rho_{f^{n_{k}}(x)}((I_{f^{n_{k}}x f^{n_{k}}y})^{*}\tau_{f^{n_{k}}(y)}, (H^{s}_{f^{n_{k}}x f^{n_{k}}y})^{*}\tau_{f^{n_{k}}(y)}) \rightarrow 0,
\]
as $n_{k} \rightarrow \infty$. Combining these two facts gives $\tau_{x} = (H^{s}_{xy})^{*}\tau_{y}$. 

The same proof replacing $\W^{s}$ by $\W^{u}$ and $n_{k}$ by $-n_{k}$ shows that if $x,y \in \Omega$ with $y \in \W^{u}_{loc}(x)$ then  $\tau_{x} = (H^{u}_{xy})^{*}\tau_{y}$.
\end{proof}

%We will need the next proposition in the proof of Lemma \ref{lemma: L holo eqva}, 

%\begin{prop}\label{prop: lower angle}
%Suppose that $f$ is uniformly $u$-quasiconformal. Then there is a constant $\delta > 0$ such that for every $x \in M$, every $n \in \Z$, and every pair of unit vectors $v,w \in E^{u}_{x}$, 
%\[
%\left\|\frac{Df^{n}_{x}(v)}{\|Df^{n}_{x}(v)\|} - \frac{Df^{n}_{x}(w)}{\|Df^{n}_{x}(w)\|} \right\| \geq \delta \|v-w\| 
%\]
%\end{prop}

%\begin{proof}
%kn
%\end{proof}

\subsection{Equivariance properties of the center holonomy}\label{subsec: equi}
From now to the end of section \ref{section: higherend}, for any $x,y\in M$ such that $y\in U_x$, for any $C^1$ center path $\gamma$ lies in $U_x$ from $x$ to $y$, \color{black}
we write $h^{c}_{xy}:=h^{cs}_\gamma:  \mathcal{W}^{u}_{loc}(x) \rightarrow \mathcal{W}^{u}_{loc}(y)$ for the center-stable holonomy between local unstable leaves inside the same center-unstable leaf, which coincides with the center holonomy between these leaves. It is easy to see that under either assumption (A) or (B), $h^c_{xy}$ is well-defined and does not depend on the choice of $\gamma$.\color{black}

\begin{prop}\label{prop: differentiability} Let
\begin{align*}
Q = &\{x \in M: \text{for $m^{c}_{x}$-a.e. $y \in U_{x}$}, \\
&\text{$h^{c}_{xy}:\W^{u}_{loc}(x) \rightarrow \W^{u}_{loc}(y)$ is differentiable at $x$}\}.
\end{align*}
Then $m(Q)=1$.
\end{prop}
\begin{proof} 
For $x \in M$ we let $m_{x}^{u}$ denote the conditional measure of $m$ on $\W^{u}_{loc}(x)$. Let $m^{cu}_{x}$ be the conditional measure of $m$ on the subset 
\[
S_{x} = \bigcup_{y \in \W^u_{loc}(x)} U_{y} \subset \W^{cu}_{loc}(x).
\]
Considering $\mathcal{W}^{c}(x)$ as an absolutely continuous foliation with respect to the transverse strongly absolutely continuous foliation $\mathcal{W}^{u}(x)$ inside of $\mathcal{W}^{cu}(x)$, we conclude from Corollary \ref{mutual continuity} that the measure $m^{cu}_{x}$ decomposes as conditional measures in two different ways,
\[
m^{cu}_{x} \asymp \int_{U_{x}} m^{u}_{y} \, dm^{c}_{x}(y) \asymp \int_{W^u_{loc}(x)} m^{c}_{y} \, dm^{u}_{x}(y) 
\] 
where we recall that we use the notation $\asymp$ to indicate that the two measures are equivalent on $S_{x}$. By Lemma \ref{lemma: quasiconformal} and Proposition \ref{proposition: quasiconformal properties}, for every $y \in U_{x}$ the center holonomy map $h^{c}_{yx}: \W^{u}_{loc}(y) \rightarrow \W^{u}_{loc}(x)$ is differentiable at $m^{u}_{y}$-a.e. $z \in \W^{u}_{loc}(y)$. Thus if we set 
\[
T_{x} = \{z \in S_{x}: \text{$h^{c}_{xy}$ is differentiable at $z$ for $y = h^{u}_{z x}(z)$}\},
\]
then by the first expression for $m^{cu}_{x}$ we have $m^{cu}_{x}(T_{x}) = m^{cu}_{x}(S_{x})$ for $m$-a.e. $x \in M$. Since $T_{x}$ has full $m^{cu}_{x}$ measure in $S_{x}$ we conclude by the second expression for $m^{cu}_{x}$ that $m^{c}_{y}(T_{x} \cap U_{y}) = m^{c}_{y}(U_{y})$ for $m^{u}_{x}$-a.e. $y \in \W^{u}_{loc}(x)$. This immediately implies that $m^{u}_{x}( Q \cap W^{u}_{loc}(x)) = m^{u}_{x}(\W^{u}_{loc}(x))$ from the definition of $Q$. Since the $\W^u$ foliation is absolutely continuous and this holds for $m$-a.e. $x \in M$ we conclude that $m(Q) = 1$, i.e., $Q$ has full volume in $M$. 
\end{proof}
We then define $\mathcal{Q} = \{(x,y) \in \E: x \in Q\}$. From the definition of $Q$ and Proposition \ref{prop: differentiability} we see that $\mathcal{Q}$ has full $\mu$-measure inside of $\E$. For $(x,y) \in \mathcal{Q}$ we can then define $H^{c}_{xy}: E^{u}_{x} \rightarrow E^{u}_{y}$ to be the derivative of $h^{c}_{xy}$ at $x$. The map $(x,y) \rightarrow H^{c}_{xy}$ is clearly measurable and defined $\mu$-a.e. by Proposition \ref{prop: differentiability}. Our next goal is to show that the maps $H^{c}$ are equivariant with respect to the stable and unstable holonomies $H^{s}$ and $H^{u}$ of $Df|_{E^{u}}$.  

\begin{lemma}\label{lemma: L holo eqva}There is a full $\mu$-measure subset $\Omega$ of $\mathcal{Q}$ such that if $(x,y)$, $(z,w) \in \Omega$ with $z \in \W^{u}_{loc}(x)$ and $w \in \W^{u}_{loc}(y)$ then the following equation holds,
\begin{equation}\label{cu holonomy}
H^{c}_{zw}\circ H^u_{xz}=H^u_{yw}\circ H^c_{xy},
\end{equation}
and similarly if $(z,w) \in \Omega$ with $z \in \W^{s}_{loc}(x)$ and $w \in \W^{s}_{loc}(y)$ then,
\begin{equation}\label{cs holonomy}
H^{c}_{zw}\circ H^s_{xz}=H^s_{yw}\circ H^{c}_{xy}.
\end{equation}
\end{lemma}

\begin{proof}
We let $\Lambda \subset M$ be the full $m$-measure set of points on which the invariant bounded measurable conformal structure $\tau: M \rightarrow \mathcal{C}E^{u}$ of Proposition \ref{prop: confstruct} is defined and invariant under both $Df$ and the stable and unstable holonomies $H^{s}$ and $H^{u}$. We let $\Omega_{0} \subset \E$ be the set of $(x,y) \in \E$ such that both $x$ and $y$ are in $\Lambda$. The absolute continuity of $\W^{c}$ together with the construction of the measure $\mu$ implies that $\mu(\E \backslash \Omega_{0}) = 0$.

By Lusin's theorem we can find an increasing sequence of compact subsets $K_n\subset \E$ such that $\mu(\E \backslash K_{n}) < 2^{-n}$ and such that $H^{c}$ restricts to a uniformly continuous function on each $K_{n}$. Since $\mu$ is $F$-invariant, by applying Lemma \ref{lemma: half} to both $F$ and $F^{-1}$ there is a measurable set $\Omega$ with $\mu(\E \backslash \Omega) = 0$ and such that for any pair of points $(x,y), (z,w) \in \Omega$ there is an $n \in \N$ and a pair of infinite sequences $n_{k} \rightarrow \infty$ and $n_{k}' \rightarrow \infty$ with $F^{-n_{k}}(x,y), F^{-n_{k}}(z,w) \in  K_{n}$ for each $n_{k}$ and $F^{n_{k}'}(x,y), F^{n_{k}'}(z,w) \in  K_{n}$. 

We now prove that equation \eqref{cu holonomy} holds. The proof for equation \eqref{cs holonomy} will be completely analogous. Let $(x,y) , (z,w) \in \Omega$ be such that $z \in \W^{u}_{loc}(x)$ and $w \in \W^{u}_{loc}(y)$. Since $H^{c}$ is uniformly continuous on $K_{n}$ and $d(f^{-n}x,f^{-n}z), d(f^{-n}y,f^{-n}w) \rightarrow 0$ as $n \rightarrow \infty$, we conclude that 
\begin{equation}\label{eqn: center convergence}
\left\| H^{u}_{f^{-n_{k}}y f^{-n_{k}}w} \circ H^{c}_{f^{-n_{k}}x f^{-n_{k}}y} - H^{c}_{f^{-n_{k}}z f^{-n_{k}}w}\circ H^{u}_{f^{-n_{k}}x f^{-n_{k}}z}\right\| \rightarrow 0,
\end{equation}
 as $k \rightarrow \infty$, where $\{n_{k}\}$ is the infinite sequence from the previous paragraph corresponding to the pair $(x,y),(z,w)$. 

For  $x\in \Lambda$ we let $SE^{u}_{x}$ denote the unit sphere in $E^{u}_{x}$ in the metric $\tau_{x}$. For any two points $x,y \in \Lambda$ and an invertible linear map $A: E^{u}_{x} \rightarrow E^{u}_{y}$ we then define
\[
SA(v) = \frac{A(v)}{\det(A)^{\frac{1}{k}}},
\]
where the determinant is taken with respect to the induced Riemannian metric $\tau^{0}$ on $E^{u}$ from $TM$. We remark that if $A^{*}\tau_{y} = \tau_{x}$ then $SA$ maps $SE^{u}_{x}$ to $SE^{u}_{y}$ and consequently $SA$ is an isometry from $E^{u}_{x}$ to $E^{u}_{y}$ when these are given the metrics $\tau_{x}$ and $\tau_{y}$ respectively. This is because $A$ then maps $SE^{u}_{x}$ to $\det_{\tau}(A)^{\frac{1}{m}} \cdot SE^{u}_{y}$, where $\det_{\tau}$ denotes the determinant of this linear map with respect to the family of inner products on $E^{u}$ given by $\tau$. Our convention for representing elements of $\mathcal{C}E^{u}_{x}$ is to take the inner product which has determinant $1$ with respect to the background metric $\tau^{0}_{x}$. Thus $\det_{\tau} = \det$ for linear maps between fibers of $E^{u}$. 

We also note that $A$ is clearly determined by $SA$ and $\det(A)$. 

For $(x,y) , (z,w) \in \Omega$ given as in the statement of the lemma we will show that 
\begin{equation}\label{det equation}
\det(H^{c}_{zw})\det(H^{u}_{xz}) = \det(H^{u}_{yw}) \det(H^{c}_{xy}),
\end{equation}
\begin{equation}\label{sphere equation}
SH^{c}_{zw} \circ SH^{u}_{xz} =  SH^{u}_{yw} \circ SH^{c}_{xy}.
\end{equation}
The desired statement of the lemma follows from these two equations. 

From differentiating the equation 
\[
f^{-k} \circ h^{c}_{xy} = h^{c}_{f^{-k}x f^{-k}y} \circ f^{-k},
\]
expressing the equivariance of center holonomy with respect to the dynamics $f$ we obtain the equation 
\begin{equation}\label{eqn: center equi}
Df^{-k}_{y} \circ H^{c}_{x y}= H^{c}_{f^{-k}x f^{-k}y}\circ Df^{-k}_{x},
\end{equation}
which is valid for any $(x,y) \in \mathcal{Q}$. Taking determinants and rearranging, we conclude that 
\[
\frac{\det(Df^{-k}_{y}|_{E^{u}})}{\det(Df^{-k}_{x}|_{E^{u}})} \det(H^{c}_{xy}) = \det(H^{c}_{f^{-k}x f^{-k}y})
\]
Applying the same equation to $(z,w)$ with $H^{c}_{zw}$ and then taking ratios at the iterates $n_{k}$ gives 
\[
\frac{\det(Df^{-n_{k}}_{y}|_{E^{u}})}{\det(Df^{-n_{k}}_{w}|_{E^{u}})} \cdot \frac{\det(Df^{-n_{k}}_{z}|_{E^{u}})}{\det(Df^{-n_{k}}_{x}|_{E^{u}})} \cdot \frac{\det(H^{c}_{xy})} {\det(H^{c}_{z w})}= \frac{\det(H^{c}_{f^{-n_{k}}x f^{-n_{k}}y})}{\det(H^{c}_{f^{-n_{k}}z f^{-n_{k}}w})}
\]
As $k \rightarrow \infty$  the right side converges to 1 by equation \ref{eqn: center convergence}, the first factor in the product on the left side converges to $\det(H^{u}_{yw})$, and the second factor converges to $\det(H^{u}_{xz})^{-1}$. Rearranging the resulting equation gives equation \eqref{det equation}.

For the second equation we first consider the following lemma which only uses the hypothesis that $f$ is uniformly $u$-quasiconformal. We let $\| \cdot \|_{\tau}$ denote the norm on $E^{u}_{x}$ induced by the inner product $\tau_{x}$. 
\begin{lemma}\label{lemma: liminf holo}
Suppose $x,y\in \Lambda$,  $y \in \W^{u}_{loc}(x)$ and  $v\in E^u_x, v' \in E^u_y$. If
\[
\liminf_{n\to \infty}\|SDf^{-n}(v) -  SH^{u}_{f^{-n}yf^{-n}x} (SDf^{-n}(v'))\|_{\tau} =0,
\]
then $SH_{xy}^u(v)=v'$. 
\end{lemma}

\begin{proof} Let $w = SH^{u}_{xy}(v)$. Then we have 
\[
SDf^{-n}(v) =  SH^{u}_{f^{-n}yf^{-n}x} (SDf^{-n}(w)),
\]
by the equivariance properties of the unstable holonomy $H^{u}$. Therefore we have 
\[
\liminf_{n\to \infty}\|SH^{u}_{f^{-n}yf^{-n}x} (SDf^{-n}(w)) - SH^{u}_{f^{-n}yf^{-n}x} (SDf^{-n}(v'))\|_{\tau}=0.
\]
But the invariance of $\tau$ under the unstable holonomy $H^{u}$ on $\Lambda$ and its invariance under $Df$ imply that $SDf^{-n}$ and $SH^{u}$ are both isometries with respect to the family of metrics given by $\tau$ on the fibers $E^{u}_{x}$ of the vector bundle $E^{u}$. This then implies that
\[
\liminf_{n \to \infty} \|w-v'\|_{\tau} = 0,
\]
which means that $w = v'$ as desired. 
 \end{proof}

Now let $(x,y), (z,w) \in \Omega$ be given as in the statement of the Lemma. From equation \ref{eqn: center convergence} and the equivariance properties of $H^{u}$ we conclude that 
\begin{align*}
\lim_{k\to \infty} &\|SH^{c}_{f^{-n_{k}}z f^{-n_{k}}w}( SDf^{-n_{k}}_{z}(SH^u_{xz} (v))) \\
&- SH^{u}_{f^{-n_{k}}y f^{-n_{k}}w}(SH^{c}_{f^{-n_{k}}x f^{-n_{k}}y}(SDf^{-n_{k}}_{x}(v)))\|=0.
\end{align*}
 Applying the equivariance relation \eqref{eqn: center equi}, this implies that 
\begin{align*}
\lim_{k\to \infty} &\|SDf^{-n_{k}}_{w}(SH^{c}_{zw}(SH^u_{xz} (v))) \\
&- SH^{u}_{f^{-n_{k}}y f^{-n_{k}}w}(SDf^{-n_{k}}_{y}(SH^{c}_{xy}(v)))\|=0
\end{align*}
Since the measurable conformal structure $\tau$ is bounded, the norms $\| \cdot \|_{\tau}$ are uniformly comparable to the norm $\| \cdot \| = \| \cdot\|_{\tau_{0}}$ and thus this equation also holds with $\| \cdot \|$ replaced by $\|\cdot\|_{\tau}$. We thus conclude by Lemma \ref{lemma: liminf holo} that 
\[
SH^{c}_{zw} (SH^{u}_{xz}(v)) = SH^{u}_{yw}(SH^{c}_{xy}(v)).
\]
Since this holds for every $v \in E^{u}_{x}$ we deduce equation \eqref{cu holonomy} as desired.  

To prove the second equation \eqref{cs holonomy} where instead $z \in \W^{s}_{loc}(x)$ and $w \in \W^{s}_{loc}(y)$, we follow the exact same proof, replacing $H^{u}$ by $H^{s}$ and $-n_{k}$ by $n_{k}$ everywhere. 
\end{proof}

We next recall the following elementary lemma from analysis,
\begin{lemma}\label{lemma: cont ext}
Suppose $f: \RR^k \to \RR^k$ is ACL and that there is a continuous map $G:\RR^k\to GL(k,\RR)$ such that $Df=G$ almost everywhere. Then $f$ is a $C^{1}$ map and $Df=G$ everywhere. 
\end{lemma}

\begin{proof}
%First we consider the case in $F:\RR\to \RR$ is absolutely continuous, then $F'(x)$ exists almost everywhere and by Newton-Leibniz formula, we have 
%\begin{equation}
%F(x)=F(0)+\int_{0}^{x}F'(x)dx=F(0)+\int_{0}^{x}G(x)dx
%\end{equation}
%As a result, $F'=G$ everywhere, and $F$ is $C^1$.

Let $f = (f_1,\dots ,f_k)$, $f_{i}: \RR^{k} \rightarrow \RR$.  Since $f$ is ACL each coordinate function $f_{i}$ is ACL.  Thus there exists a full Lebsgue measure set $\Lambda \subset \RR^k$ such that for every $x\in \Lambda$ and $1 \leq i,j\leq k$, $f_{i}|_{x+\RR\cdot e_j}$ is absolutely continuous and $Df=G$ for almost every point (with respect to arc length) in $\{x+\RR \cdot e_j\}$, where $e_{1}, \dots, e_{k}$ denote the standard basis of $\RR^k$.

Absolute continuity of $f_{i}|_{x+\RR\cdot e_j}$ implies that 
 \begin{align*}
 f_i(x+t\cdot e_j)&= f_i(x)+\int_{0}^{t}\frac{\partial f_i}{\partial x_j}(x+s\cdot e_i)ds\nonumber\\
 &= f_i(x)+\int_{0}^{t}G_{ij}(x+s\cdot e_i)ds,
  \end{align*}
where $G = (G_{ij})_{1\leq i,j\leq k}$ is the matrix representation of $G$ in the standard basis of $\RR^{k}$.  Since both $f$ and $G$ are continuous this last equation holds for all $x\in \RR^k$ and $t\in \RR$. This proves that for each $1 \leq i,j\leq k$ the partial derivative $\frac{\partial f_i}{\partial x_j}$ of $f$ exists and coincides with $G_{ij}$. In particular all partial derivatives of $f$ exist and are continuous at every point in $\RR^{k}$ which implies that $f$ is $C^1$ and $Df=G$. 
\end{proof}

\begin{lemma}\label{lemma: linear center}
For any $x \in M$ and $y \in U_{x}$ we have the equality 
\[
\Phi_{y}^{-1} \circ h^{c}_{xy} \circ \Phi_{x} = H^{c}_{xy}  
\]
as maps from $E^{u}_{x}$ to $E^{u}_{y}$. The measurable function $H^{c}$ on $\Omega$ therefore admits a continuous extension to $\E$ and the center holonomy is linear in the charts $\{\Phi_{x}\}_{x \in M}$. 
\end{lemma}

\begin{proof}
We first consider pairs $(x,y) \in \Omega$. Since $D_{0}\Phi_{x} = Id_{E^{u}_{x}}$ for every $x \in M$ and $h^{c}_{xy}(x) = y$, the equation 
\[
D_{0}(\Phi_{y}^{-1} \circ h^{c}_{xy} \circ \Phi_{x}) = H^{c}_{xy},
\]
holds for any $(x,y) \in \Omega$. To compute the derivative at other points of $E^{u}_{x}$, we let $v \in E^{u}_{x}$, $z = \Phi_{x}(v)$, and $w = h^{c}_{xy}(z) = h^{c}_{zw}(z)$. We suppose that $(z,w) \in \Omega$ and compute, 
\[
D_{v}(\Phi_{y}^{-1} \circ h^{c}_{xy} \circ \Phi_{x}) = D_{0}(\Phi_{y}^{-1} \circ \Phi_{w}) \circ D_{0}(\Phi_{w}^{-1} \circ h^{c}_{zw} \circ \Phi_{z}) \circ D_{v}(\Phi_{z}^{-1} \circ \Phi_{x}). 
\]
By Proposition \ref{proposition: transitions} we know that $D_{0}(\Phi_{y}^{-1} \circ \Phi_{w}) = H^{u}_{wy}$ and $D_{v}(\Phi_{z}^{-1} \circ \Phi_{x}) = H^{u}_{xz}$. Hence
\[
D_{v}(\Phi_{y}^{-1} \circ h^{c}_{xy} \circ \Phi_{x}) = H^{u}_{wy} \circ H^{c}_{zw} \circ H^{u}_{xz} = H^{c}_{xy},
\]
whenever $(x,y), (z,w) \in \Omega$, by Lemma \ref{lemma: L holo eqva}. Since $\Omega$ has full $\mu$-measure we conclude that for $m$-a.e. $x \in M$ and $m^{c}_{x}$-a.e. $y \in U_{x}$ the map $\Phi_{y}^{-1} \circ h^{c}_{xy} \circ \Phi_{x}: E^{u}_{x} \rightarrow E^{u}_{y}$ is differentiable almost everywhere on $E^{u}_{x}$ with derivative $H^{c}_{xy}$ almost everywhere. By Lemma \ref{lemma: quasiconformal} the map $\Phi_{y}^{-1} \circ h^{c}_{xy} \circ \Phi_{x}$ is quasiconformal and therefore ACL.  By Lemma \ref{lemma: cont ext} this implies that $\Phi_{y}^{-1} \circ h^{c}_{xy} \circ \Phi_{x}$ is a $C^{1}$ map with derivative $H^{c}_{xy}$ everywhere, i.e., $\Phi_{y}^{-1} \circ h^{c}_{xy} \circ \Phi_{x}$ coincides exactly with the linear map $H^{c}_{xy}$. 

Since $\Omega$ has full $\mu$-measure in $\E$ and $\mu$ is fully supported we conclude that $\Omega$ is dense in $\E$ and thus the equation 
\begin{equation}\label{eqn:lincntr}
\Phi_{y}^{-1} \circ h^{c}_{xy} \circ \Phi_{x} = H^{c}_{xy}
\end{equation}
holds on $E^{u}_{x}$ for a dense set of pairs $(x,y) \in \E$. But the left side of this equation depends uniformly continuously on the pair $(x,y)$ on as a map $\Phi_{x}^{-1}(\W^{u}_{loc}(x)) \rightarrow \Phi_{y}^{-1}(\W^{u}_{loc}(y))$ between neighborhoods of 0 in $E^{u}_{x}$ and $E^{u}_{y}$ of uniform size. Furthermore the linear map $H^{c}_{xy}$ is determined by its restriction to a map between these neighborhoods. Hence we conclude that $H^{c}_{xy}$ also depends uniformly continuously on the pairs $(x,y) \in \Omega$. 

When $y$ is close to $x$ the map $\Phi_{y}^{-1} \circ h^{c}_{xy} \circ \Phi_{x}$ is uniformly close to the linear identifications $I_{xy}: E^{u}_{x} \rightarrow E^{u}_{y}$ introduced in Section \ref{section: center holonomy}. Hence $H^{c}_{xy}$ is uniformly close to $I_{xy}$ for $(x,y) \in \Omega$. In particular for $(x,y) \in \Omega$ the maps $H^{c}_{xy}$ belong to a uniformly bounded subset of the space of invertible linear maps $E^{u}_{x} \rightarrow E^{u}_{y}$. This shows that $H^{c}$ admits a continuous extension to $\E$ such that equation \eqref{eqn:lincntr} still holds on a neighborhood of 0 in $E^{u}_{x}$ for any $(x,y) \in \E$. Finally, because $\Phi_{x}$, $h^{c}_{xy}$, and $H^{c}_{xy}$ all have the proper equivariance properties with respect to $f^{-1}$ and $Df^{-1}$ which uniformly contract $E^{u}$ it follows that equation \ref{eqn:lincntr} actually holds on all of $E^{u}_{x}$. 
\end{proof}

As a corollary of Lemma \ref{lemma: linear center} we deduce that equations \eqref{cu holonomy} and \eqref{cs holonomy} from Lemma \ref{lemma: L holo eqva} actually hold on all of $\E$ because the uniform continuity of $H^{c}$ implies each side of these equations is uniformly continuous in the quadruple of points $x,y,z,w$ and both of these equations hold on a dense subset of $\E$. 
\section{Higher regularity of foliations}\label{section: higherend}
In this section we will prove higher regularity properties for the foliations $\mathcal{W}^{c}$, $\mathcal{W}^{cs}$ and $\mathcal{W}^{cu}$. We start by proving the smoothness of center stable holonomy. To obtain it we will construct a continuous conformal structure on the unstable bundles.

\subsection{Smoothness of center stable holonomy}
\color{black}
Given $x \in M$, $y \in \W^{cs}_{loc}(x)$, we let $z$ be the unique intersection point of $\W^{c}_{loc}(x)$ with $\W^{s}_{loc}(y)$ and define 
\[
H^{cs}_{xy} =  H^{s}_{zy} \circ H^{c}_{xz}.
\]
By the observation in the previous paragraph, if we let $w$ be the intersection of $\W^{s}_{loc}(x)$ with $\W^{c}_{loc}(y)$ then we also have the equality
\[
H^{cs}_{xy} = H^{c}_{wy} \circ H^{s}_{xw}.
\]
We note that $H^{cs}_{xy}$ depends in a uniformly continuous fashion on $x$ and $y$ from the uniform continuity of $H^{c}$ and $H^{s}$.

\begin{lemma}\label{lemma: cs diff} 
The center-stable holonomy $h^{cs}_{xy}: \W^{u}_{loc}(x) \rightarrow \W^{u}_{loc}(y)$ between two local unstable leaves is $C^{1}$ with derivative $H^{cs}$. 
\end{lemma}

\begin{proof}
We will first show that if $y\in \W^s_{loc}(x)$ and $h^{cs}_{xy}:\W^{u}_{loc}(x) \rightarrow \W^{u}_{loc}(y)$ is differentiable at $x$ then $Dh^{cs}_{xy}(x)=H^{s}_{xy}$. We will prove this by contradiction. 

If $D_{x}(h^{cs}_{xy}) \neq H^{s}_{xy}$ then $L_{xy} := \Phi_{y}^{-1} \circ h^{cs}_{xy} \circ \Phi_{x}$ is differentiable at $0 \in E^{u}_{x}$ but $D_{0}(L_{xy}) \neq H^{s}_{xy}$.  Thus there exists $v\in E^u_x$ with $\|v\|=1$ and some constants $\varepsilon_{0},\eta >0$ such that 
\[
\|L_{xy}(tv) - H^{s}_{xy}(tv)\| \geq \varepsilon_{0} t, \; \; \forall |t| \leq \eta.
\]
By the uniform $u$-quasiconformality of $f$ there is then a constant $C \geq 1$ independent of $n$ such that 
\begin{equation}\label{eqn:expanding}
\|Df^{n}_{y}(L_{xy}(tv)) - Df^{n}_{y}(H^{s}_{xy}(tv))\| \geq C^{-1}\det(Df^{n}_{y}|_{E^{u}})\varepsilon_{0} t, \; \; \forall |t| \leq \eta,
\end{equation}
and also with the properties that for for every $x \in M$ and any unit vector $\xi \in E^{u}_{x}$, 
\[
C^{-1}  \det(Df^{n}_{x}|_{E^{u}}) \leq \| Df^{n}_{x}(\xi)\| \leq C \det(Df^{n}_{x}|_{E^{u}}),
\]
and lastly the distortion estimate $\det(Df^{n}_{y}|_{E^{u}}) \leq C \det(Df^{n}_{x}|_{E^{u}})$ holds for $y \in \W^{s}_{loc}(x)$ and $n \geq 0$. 

By the uniform continuity of the charts $\Phi_{x}$ in the $x$-variable, the $\W^{cs}$ foliation and $H^{s}$, given any $\varepsilon > 0$ there exists $\delta = \delta(\varepsilon)$ such that for any $z \in M$, $w \in \W^{s}_{loc}(z)$ with $d_{s}(z,w) \leq \delta$ and any $\xi \in E_{z}^u$ satisfying $\|\xi\| \leq 1$  we have 
\begin{equation}\label{eqn:hcs close to id}
\|L_{zw}(\xi)-H^s_{zw}(\xi) \|<\varepsilon
\end{equation}
We choose $\varepsilon < C^{-3}\varepsilon_{0}$ and then choose $n$ large enough that 
\[
d_{s}(f^{n}(x),f^{n}(y)) \leq \delta(\varepsilon),
\]
and such that  $C^{-1}\det(Df^{n}_{x}|_{E^{u}})^{-1} < \eta$. 

We put $z = f^{n}(x)$ and $w = f^{n}(y)$. Applying the equivariance of $Df$ with respect to the charts $\Phi_{x}$, the center stable holonomy $h^{cs}$, and the linear stable holonomy $H^{s}$ in equation \eqref{eqn:expanding} we obtain 
\[
\|L_{zw}( Df^{n}_{x}(tv)) - H^{s}_{zw}(Df^{n}_{x}(tv))\| \geq C^{-2}\det(Df^{n}_{x}|_{E^{u}})\varepsilon_{0} t, \; \; \forall |t| \leq \eta,
\]
Let $t$ be the maximal number such that $\|Df^{n}_{x}(tv)\| \leq 1$ and then put $\xi = Df^{n}_{x}(tv)$. We conclude that equation \eqref{eqn:hcs close to id} applies to the above and thus obtain 
\[
\varepsilon > C^{-2}\det(Df^{n}_{x}|_{E^{u}})\varepsilon_{0} t.
\]
But we have $\eta \geq t \geq C^{-1}\det(Df^{n}_{x}|_{E^{u}})^{-1}$ by the uniform $u$-quasiconformality of $f$. Hence we conclude that $\varepsilon > C^{-3} \varepsilon_{0}$, contradicting our choice of $\varepsilon$. 

Now suppose that $y \in \W^{cs}_{loc}(x)$ and $h^{cs}_{xy}$ is differentiable at $x$. Let $z = \W^{cu}_{loc}(x) \cap \W^{s}_{loc}(y)$. Then $h^{cs}_{xy} =h^{cs}_{zy} \circ h^{cs}_{xz}$. The map $h^{cs}_{xz}: \W^{u}_{loc}(x)\rightarrow \W^{u}_{loc}(z)$ coincides with the center holonomy from $\W^{u}_{loc}(x)$ to $\W^{u}_{loc}(z)$ and thus it follows from Lemma \ref{lemma: linear center} that $h^{cs}_{xz}$ is a $C^{1}$ diffeomorphism with derivative $H^{c}_{xz}$ at $x$. We conclude that $h^{cs}_{zy}$ is differentiable at $z$ and thus by our work above the derivative of $h^{cs}_{zy}$ at $z$ is given by $H^{s}_{zy}$. Thus $D_{x}(h^{cs}_{xy}) = H^{s}_{zy} \circ H^{c}_{xz} = H^{cs}_{xy}$. Since $h^{cs}_{xy}$ is ACL from the quasiconformality given by Lemma \ref{lemma: quasiconformal} and $H^{cs}_{xy}$ is uniformly continuous in $x$ and $y$ by the remarks preceding this lemma we conclude by Lemma \ref{lemma: cont ext} that $h^{cs}_{xy}$ is $C^{1}$ with derivative given by $H^{cs}$. 
\end{proof}

\begin{lemma}\label{lemma: center conformal}
There is a continuous invariant conformal structure $\tau: M \rightarrow E^{u}$ for $Df|_{E^{u}}$ which is invariant under $H^{c}_{loc}$, $H^{u}$, and $H^{s}$ holonomies.
\end{lemma}

\begin{proof}
By Proposition \ref{prop: confstruct} there is a bounded measurable invariant conformal structure $\hat{\tau}: M \rightarrow \mathcal{C}E^{u}$ for $Df|_{E^{u}}$ defined on a full measure subset $\Omega$ of $M$ such that $\hat{\tau}$ is $H^{u}$ and $H^{s}$ invariant on $\Omega$. Without loss of generality we assume $\Omega$ is $f-$invariant. Then for any $x\in M$, 
\begin{equation}\label{eqn: Omega cap U f inv}
f(\Omega\cap \W^*(x))=\Omega \cap \W^*(f(x)),*\in \{s,c,u\} \text{ and }f(\Omega \cap U_x)=\Omega\cap U_{f(x)}
 \end{equation} 
 
Under assumption (B), for $x \in M$ we define $\tau_{x}$ to be the \textit{center} \color{black} in $\mathcal{C}E^{u}_{x}$ of the set 
\[
\mathcal{O}_{x} = \left\{ (H^{c}_{xy})^{*}\hat{\tau}_{y}: y \in \Omega \cap U_x \right\},
\]
with respect to the nonpositively curved metric $\rho_{x}$. As in the proof of Proposition 3.1 of \cite{S05} the center of a non-empty bounded set $D$ in the space of conformal structure is defined to be the center of the uniquely determined ball of smallest radius containing $D$. \color{black}For the existence and uniqueness of centers of subsets with respect to nonpositively curved metrics see \cite{E96}. 

The definition of $\tau_{x}$ assumes the set $\mathcal{O}_{x}$ is nonempty; this will be true for $m$-a.e. $x \in M$ because of the absolute continuity of the center foliation. The definition of $\tau_{x}$ also assumes that $\mathcal{O}_{x}$ is a \emph{bounded} subset of $\mathcal{C}E^{u}_{x}$. This is clear under assumption  (B). 

We claim that $\tau$ is invariant under $H^{c}_{loc}$ holonomy. We only need to prove $(H^c_{xy})^*\mathcal{O}_y=\mathcal{O}_x$ for any $x\in M$ and any $y\in \W^c_{loc}(x)$ in the closed segment joining $x$ to $f(x)$ such that the segment joining from $f^{-1}(x)$ to $f(y)$ does not wrap. In this case $h^c_{wz}, H^c_{wz}$ are well-defined for $w, z\in U_x\cup U_y$. Moreover, for any $z\in U_x \backslash U_y$, we have $f^2(z)\in U_y\backslash U_x$. Therefore $h^c_{zf^2(z)}$ is a well-defined map on an unstable disc $\W^u_\epsilon(z)$ for sufficiently small $\epsilon$. Since $h^c$ is equivariant with respect to $f$, we have $h^c_{zf^2(z)}|_{\W^u_\epsilon(z)}=f^2|_{\W^u_\epsilon(z)}$. Then for any $z\in U_x \backslash U_y$,
\begin{equation}\label{eqn:hc zf2z=f2}
H^c_{zf^2(z)}=Df^2(z)|_{E^u(z)}
\end{equation}

As a result,
\begin{eqnarray*}\nonumber
(H^c_{xy})^*\mathcal{O}_y&=&(H^c_{xy})^*\{(H^c_{yz})^*\hat{\tau}_z: z\in \Omega\cap U_y\}\\\nonumber
&=&\{(H^c_{xz})^*\hat{\tau}_z: z\in \Omega\cap U_y\}\\\nonumber
&=&\{(H^c_{xz})^*\hat{\tau}_z: z\in \Omega\cap U_y\cap U_x\}\cup \{(H^c_{xz})^*\hat{\tau}_z: z\in \Omega\cap U_y\backslash U_x\}\\ \nonumber
&=&\{(H^c_{xz})^*\hat{\tau}_z: z\in \Omega\cap U_y\cap U_x\}\cup \{(H^c_{xf^2(z)})^*\hat{\tau}_{f^2(z)}: z\in \Omega\cap U_x\backslash U_y\}\\ \nonumber
&&(\text{ by }\eqref{eqn: Omega cap U f inv}, f^2(\Omega\cap U_x\backslash U_y)=\Omega\cap U_y\backslash U_x)\\ \nonumber
&=&\{(H^c_{xz})^*\hat{\tau}_z: z\in \Omega\cap U_y\cap U_x\} \cup \\ \nonumber
&& \{(H^c_{xz})^* (H^c_{zf^2(z)})^*\hat{\tau}_{f^2(z)}: z\in \Omega\cap U_x\backslash U_y\}\\ \nonumber
&=&\{(H^c_{xz})^*\hat{\tau}_z: z\in \Omega\cap U_y\cap U_x\}\cup \{(H^c_{xz})^*\hat{\tau}_{z}: z\in \Omega\cap U_x\backslash U_y\}\\ \nonumber
&&(\text{ by \eqref{eqn:hc zf2z=f2} and $Df-$invariance of }\hat{\tau})\\
&=&\mathcal{O}_x,
\end{eqnarray*}
where we use here that the linear map $H^{c}_{xy}$ induces an isometry $\mathcal{C}E^{u}_{y} \rightarrow \mathcal{C}E^{u}_{x}$. 

Thus $\tau$ is invariant under $H^c_{loc}$ holonomy. \color{black} By the equivariance of $H^{c}_{loc}$  with respect to $H^{u}$ and $H^{s}$ given by equations \eqref{cu holonomy} and \eqref{cs holonomy}, $\tau$ is also invariant under $H^{s}$ and $H^{u}$ holonomy on $\Omega$. In particular $\tau$ is invariant under both $H^{cs}$ and $H^{u}$ so $\tau$ is invariant under uniformly continuous holonomies along two transverse foliations of $M$. It follows that $\tau$ is uniformly continuous on $\Omega$ and thus has a unique continuous extension to $M$ which is invariant under $H^{c}_{loc}$, $H^{u}$, and $H^{s}$ holonomies.

In the case that $f$ satisfies assumption (A), 
%we firstly assume each center leaf has trivial holonomy group. 
we have that for any $x,y$ in the same central leaf the holonomy maps $h^c_{xy}$ and $H^c_{xy}$ are uniquely defined. Therefore as in the case of assumption (B), $\mathcal{O}_{x} = \left\{ (H^{c}_{xy})^{*}\hat{\tau}_{y}: y \in \Omega \cap \W^c_x \right\}$ is well-defined and we define $\tau_x$ similarly. Clearly $\tau$ is invariant under $H^c$ holonomy because of
the composition property $H^c_{yz}=H^c_{xz}\circ H^c_{yx}$ for $x,y,z$ in the same center leaf. Then the rest of the proof is the same as the previous paragraphs.
\end{proof}

By combining Lemmas \ref{lemma: cs diff} and \ref{lemma: center conformal} we derive the main result of this subection,  

\begin{cor}\label{cor: analytic}
The center-stable holonomy $h^{cs}$ between local unstable leaves is analytic in the charts $\{\Phi_{x}\}_{x \in M}$. Hence $h^{cs}_{xy}: \W^{u}_{loc}(x) \rightarrow \W^{u}_{loc}(y)$ is a $C^{\infty}$ diffeomorphism.
\end{cor}

\begin{proof}
Let $\tau$ be the continuous invariant conformal structure on $E^{u}$ from Lemma \ref{lemma: center conformal} which is invariant under $H^{c}$, $H^{u}$ and $H^{s}$. Let $y \in \W^{cs}_{loc}(x)$. We consider $\tau_{x}$ and $\tau_{y}$ as conformal structures $\omega_{x}$ and $\omega_{y}$ on the Euclidean spaces $E^{u}_{x}$ and $E^{u}_{y}$ respectively by using the canonical identification for each $v \in E^{u}_{x}$ of $T_{v}E^{u}_{x}$ with $E^{u}_{x}$ and assigning $\omega_{x}$ to be the image of $\tau_{x}$ in $T_{v}E^{u}_{x}$ under this identification. 

We claim that $(\Phi_{y}^{-1} \circ h^{cs}_{xy} \circ \Phi_{x})^{*}\omega_{y} = \omega_{x}$. To show this, let $v \in E^{u}_{x}$ be given and let $v' = \Phi_{y}^{-1}(h^{cs}_{xy}( \Phi_{x}(v))) \in E^{u}_{y}$ be its image. Let $z = \Phi_{x}(v)$ and $w = \Phi_{y}(v')$. Similarly to Lemma \ref{lemma: linear center} we write
\begin{align*}
D_{v}(\Phi_{y}^{-1} \circ h^{cs}_{xy} \circ \Phi_{x}) &= D_{0}(\Phi_{y}^{-1} \circ \Phi_{w}) \\
&\circ D_{0}(\Phi_{w}^{-1} \circ h^{cs}_{zw} \circ \Phi_{z}) \circ D_{v}(\Phi_{z}^{-1} \circ \Phi_{x}) \\
&= H^{u}_{wy} \circ D_{0}(\Phi_{w}^{-1} \circ h^{cs}_{zw} \circ \Phi_{z}) \circ H^{u}_{xz} \\
&= H^{u}_{wy} \circ H^{cs}_{zw} \circ H^{u}_{xz}
\end{align*}
where in the third line we used the fact that $H^{cs}_{zw}$ is the derivative of $h^{cs}_{zw}$ at $z$ from Lemma \ref{lemma: cs diff} and that both $D_{0}\Phi_{z} = Id_{E^{u}_{z}}$ and $D_{0}\Phi_{w} = Id_{E^{u}_{w}}$. By the invariance of $\tau$ under $H^{cs}$ and $H^{u}$ we conclude that $D_{v}(\Phi_{y}^{-1} \circ h^{cs}_{xy} \circ \Phi_{x})^{*}\omega_{y} = \omega_{x}$ for every $v \in E^{u}_{x}$. 

Identifying $E^{u}_{x}$ and $E^{u}_{y}$ with the Euclidean space $\R^{k}$, the inner products $\omega_{x}$ and $\omega_{y}$ are smoothly equivalent to the Euclidean norm on $\R^{k}$. Thus conformal mappings with respect to these inner products are the same as conformal mappings with respect to the standard Euclidean metric. Since $\Phi_{y}^{-1} \circ h^{cs}_{xy} \circ \Phi_{x}$ is conformal as a map between two open subsets of $\R^{k}$ we conclude that it is analytic: for $k = 2$ this is a classical result in one-variable complex analysis and for $k \geq 3$ this follows from Gehring's theorem that all 1-quasiconformal mappings between subdomains of $\R^{k}$ are the restrictions of M\"obius transformations to these domains \cite{G62}. 
\end{proof}

%\begin{rema}\label{rem: stable ac}
%The uniform $s$-quasiconformality of $f$ does not play a prominent role in the results of this section. It suffices for this section to only assume that $f$ is uniformly $u$-quasiconformal and the center foliation has absolutely continuous disintegration along the stable foliation. 
%\end{rema}

\subsection{Regularity of the foliations}\label{subsec: reg} We now prove higher regularity of the $\W^{cu}$, $\W^{cs}$, and $\W^{c}$ foliations under additional bunching hypotheses on $f$. We begin with a folklore lemma which enables us to deduce regularity properties of a foliation from regularity properties of its holonomy maps between a specific family of transversals. When this family of transversals is smooth Lemma \ref{lemma:holo reg implies foliation reg} follows directly from the claims in \cite{PSW}; the proof of Lemma \ref{lemma:holo reg implies foliation reg} is essentially identical to the proof of \cite[Lemma 3.2]{S05} which handled the specific case of weak foliations of Anosov flows which were transverse to the strong unstable foliation. 

\begin{lemma}\label{lemma:holo reg implies foliation reg}
Let an integer $r \geq 1$ and $\alpha > 0$ be given. Suppose that $\mathcal{W}$ and $\mathcal{F}$ are two transverse foliations of $M$ such that both $\W$ and $\mathcal{F}$ have uniformly $C^{r+\alpha}$ leaves. We further suppose that the local holonomy maps along $\mathcal{W}$ between any two $\mathcal{F}$-leaves are locally uniformly $C^{r+\alpha}$. Then $\mathcal{W}$ is a $C^{r+\alpha}$ foliation of $M$. 
\end{lemma}

\begin{proof}
Let $n = \dim M$ and $k = \dim \W$. As in \cite[Lemma 3.2]{S05}, we fix a point $x \in M$ together with a neighborhood $V$ of $x$ and choose a $C^{\infty}$ coordinate chart $g: V \rightarrow \RR^{k} \times \RR^{n-k}$ such that $g(V \cap \W(x)) \subset \RR^{k} \times \{0\}$ and $g(V \cap \mathcal{F}(x)) \subset \{0\} \times \RR^{n-k}$. We then define for $p = (y,z) \in g(V)$, 
\[
\Psi(p) = (y, g(\W)(p) \cap g(\mathcal{F})(0)) = (y, h_{p,0}(z)),
\]
where $g(\W)$, $g(\mathcal{F})$ denote the images of our foliations under $g$ and $h_{p,0}(y)$ is the unique intersection point of $g(\W)(p)$ with $g(\mathcal{F})(0)$ inside of $g(V)$. This map straightens the $\W$-foliation into a foliation of $\RR^{k} \times \RR^{n-k}$ by $k$-disks $D^{k} \times \{z\}$. Since the leaves of $\W$ are uniformly $C^{r+\alpha}$ the map $\Psi$ is $C^{r+\alpha}$ when restricted to the leaves of $\W$, and since the holonomy maps of the $\W$ foliation between $\mathcal{F}$-transversals are uniformly $C^{r+\alpha}$ the chart $\Psi$ is also $C^{r+\alpha}$ along the leaves of $\mathcal{F}$. By Journ\'e's lemma \cite{J} this implies that $\Psi$ is $C^{r+\alpha}$.

\end{proof}

\begin{lemma}\label{reg of cent folia}
Suppose $f$ is $r-$bunched for some $r\geq 1$, then there is an $\alpha > 0$ such that $\W^c$, $\W^{cs}$, and $\W^{cu}$ are $C^{r+\alpha}$ foliations of $M$. If $f$ is $\infty$-bunched then these foliations are all $C^{\infty}$. 
\end{lemma}
\begin{proof} Since $f$ is $C^\infty$ and $r-$bunched, there is an  $\alpha>0$ such that the leaves of $\W_{loc}^{*}$, $* \in \{cs,cu,c\}$ are uniformly $C^{r+\alpha}$. By Corollary \ref{cor: analytic} the $cs$-holonomy maps between local unstable leaves are  analytic diffeomorphisms. Hence using $\mathcal{W}^{u}$ as our transverse foliation $\mathcal{F}$ for Lemma \ref{lemma:holo reg implies foliation reg} we conclude that $\mathcal{W}^{cs}$ is a $C^{r+\alpha}$ foliation of $M$. 

By applying all of the results of this section to the $cs$-holonomy maps of $f^{-1}$ instead (i.e., the $cu$-holonomy maps of $f$) we conclude that the $cu$-holonomy maps are analytic between local stable leaves. Hence we also obtain that $\W^{cu}$ is a $C^{r+\alpha}$ foliation of $M$.
	
For each $x \in M$ and $m = \dim E^{u}$, $k = \dim E^{c}$, $n = \dim M$ we can thus find a neighborhood $V$ of  $x$ and a $C^{r+\alpha}$ foliation chart 
\[
\Psi: V \rightarrow D^{m+k} \times D^{n-m-k} = D^{m} \times D^{n-m} \subset \RR^{n},
\]
such that $\W^{cu}$ is mapped to the foliation by $(m+k)$-cubes $D^{m+k} \times \{z\}$, $z \in D^{n-m-k}$ and $\W^{cs}$ is mapped to the foliation by $(n-m)$-cubes $\{y\} \times D^{n-m}$, $y \in D^{m}$ (here $D^{j}$ again denotes the open unit cube in $\RR^{j}$). The intersection of these two foliations is the image of $\W^{c}$ which is a foliation by $k$-disks $\{y'\} \times D^{k} \times \{z'\}$, $y' \in D^{m}$, $z' \in D^{n-m-k}$.  Thus $\Psi$ is also a $C^{r+\alpha}$ foliation chart for $E^{c}$ and therefore $\W^{c}$ is also a $C^{r+\alpha}$ foliation of $M$. 
\end{proof}

Our final lemma applies under both assumptions (A) and (B) in the case that the center is 1-dimensional. It is a straightforward consequence of the $C^{1}$ regularity of the center foliation together with the fact that in dimension 1, length and volume are the same. 

\begin{lemma}\label{prop: 1dcenter}
If $\dim E^{c} = 1$ then $f$ is $\infty$-bunched and therefore $\W^{c}$, $\W^{cs}$,  and $\W^{cu}$ are $C^{\infty}$ foliations of $M$.  Furthermore there is a $C^{\infty}$ norm $|\cdot|$ on $E^{c}$ with respect to which $Df|_{E^{c}}$ acts by isometry.
\end{lemma}

\begin{proof}
When $\dim E^{c} = 1$, $f$ is always $1$-bunched. Hence by Lemma \ref{reg of cent folia} the center foliation $\W^{c}$ is $C^{1+\alpha}$ for some $\alpha >0$. Let $\nu_{x}^{c}$ denote the Riemannian volume on $U_{x} \subset \W^{c}(x)$. Since $\W^{c}$ is a $C^{1+\alpha}$ foliation the conditional measures $\{m_{x}^{c}\}_{x \in M}$ of the volume $m$ on the sets $U_{x}$ are absolutely continuous with continuous densities with respect to $\nu_{x}^{c}$. Thus there are positive continuous functions $\zeta_{x}: U_{x} \rightarrow \RR$ such that $dm_{x}^{c} = \zeta_{x} d\nu_{x}^{c}$ which also depend continuously on $x \in M$. 

Since $f_{*}m_{x}^{c} = m_{f(x)}^{c}$  and $f_{*}\nu_{x}^{c}(y) = \| Df_{y}|_{E^{c}_{y}}\|^{-1}\nu_{f(x)}^{c}(y)$ we thus derive the relationship
\[
\frac{\zeta_{x}(y)}{\zeta_{f(x)}(f(y))} = \| Df_{y}|_{E^{c}_{y}}\|,
\]
which is valid for $y \in U_{x}$. 

%Since the functions $\zeta_{x}$ are uniformly bounded above and below independent of $x$, setting $C = \sup_{x \in M} |\log \zeta_{x}|$ and using the above relationship we have for every $n \geq 1$ and every $x \in M$, 
%\begin{align*}
%\left|\log \| Df_{x}^{n}|_{E^{c}_{x}}\| \right| &= \left|\sum_{k=0}^{n-1}\log \| Df_{f^{k}x}|_{E^{c}_{f^{k}x}}\| \right| \\
% &= \left|\log \zeta_{f^{n}(x)}(f^{n}(x)) - \log \zeta_{x}(x) \right| \leq 2C
%\end{align*}
%Thus if we set 
%\begin{equation}\label{eqn:cohom}
%\xi(x) = \sup_{n \geq 1} \prod_{k=0}^{n-1}\| Df_{f^{k}x}|_{E^{c}_{f^{k}x}}\| = \sup_{n\geq 1} \| Df_{x}^{n}|_{E^{c}_{x}}\|
%\end{equation}
%Then $e^{-2C} \leq \xi(x) \leq e^{2C}$ for every $x \in M$. 

We set $\sigma(x):= \zeta_{x}(x)$. Then $\sigma: M \rightarrow (0,\infty)$ is a continuous function satisfying the equation,
\[
\frac{\sigma(x)}{\sigma(f(x))} = \| Df_{x}|_{E^{c}_{x}}\|,
\]
for every $x \in M$. It is then clear that $Df|_{E^{c}}$ acts by isometries with respect to the norm $|\cdot | = \sigma \cdot \|\cdot\|$ on $E^{c}$. 

%It remains only to show that $\xi$ is a $C^{\infty}$ function. We will show that $\xi$ is $C^{\infty}$ when restricted to each of the three foliations $\W^{s}$, $\W^{u}$, and $\W^{c}$ separately, from which we deduce that $\xi$ itself is $C^{\infty}$ by Journe's lemma \cite{J}.  

Hence there is a continuous norm $|\cdot | $ on $E^{c}$ with respect to which $Df|_{E^{c}}$ acts by isometries. This implies that $f$ is $r$-bunched for every $r \geq 1$, i.e., $f$ is $\infty$-bunched. Thus by Lemma \ref{reg of cent folia} $\W^{cu}$, $\W^{cs}$, and $\W^{c}$ are $C^{\infty}$ foliations of $M$. This implies that the conditional measures $\{\hat{m}_{x}^{c}\}_{x \in M}$ of $m$ from Proposition \ref{prop: fiber measure} on the $\W^{c}$ are both $C^{\infty}$ in the basepoint $x \in M$ and are $C^{\infty}$ equivalent to the smooth Riemannian arclength $\nu_{x}^{c}$ on $\W^{c}(x)$; for this assertion recall that we assume that $m$ is smoothly equivalent to the Riemannian volume on $M$.

In the case of assumption (A) this implies without further argument that the family of conditional measures $\{m_{x}^{c}\}_{x \in M}$ used in this proof are also $C^{\infty}$ in $x \in M$ and are $C^{\infty}$ equivalent to $\nu_{x}^{c}$, since there is a canonical smooth normalization of the family $\{\hat{m}_{x}^{c}\}_{x \in M}$ such that $m_{x}^{c}(\W^{c}(x)) = 1$ for each $x \in M$. In the case of assumption (B) we only need to note that the arcs $U_{x} \subset \W^{c}(x)$ are determined by their endpoints in a canonical smooth fashion according to the discussion at the beginning of Section \ref{section: higher} and these endpoints are given by $f^{-1}(x)$ and $f(x)$, which clearly smoothly depend on $x$. Hence there is a smooth normalization of the family $\{\hat{m}_{x}^{c}\}_{x \in M}$ of conditional measures such that $m_{x}^{c}(U_{x}) = 1$ for every $x \in M$ and we obtain the same conclusion as we did in the case of assumption (A). As a consequence the family of $C^{\infty}$ functions $\zeta_{x}: U_{x} \rightarrow (0,\infty)$ is also $C^{\infty}$ in the basepoint $x$, so we conclude that $\sigma(x) = \zeta_{x}(x)$ is $C^{\infty}$ and consequently the norm $|\cdot |$ on $E^{c}$ is $C^{\infty}$. 
\end{proof}

\section{Proofs of Theorems 2-4}\label{subsec:proofs}

\begin{proof}[Proof of Theorem \ref{theorem: compact center}]

Since $f$ has compact center foliation with trivial holonomy, by \cite{BB} $f$ is dynamically coherent. We conclude from the results of Sections \ref{section: higher}, \ref{section: higherend} that the foliations $\W^{cs}$, $\W^{cu}$ and $\W^{c}$ are $C^{r+\alpha}$ for some $\alpha > 0$ since $f$ is $r$-bunched, volume-preserving and uniformly quasiconformal. Then we get the proof of claim (1) of Theorem \ref{theorem: compact center}.

Now we define $N$ to be the quotient of $M$ by the equivalence relation $x \equiv y$ if $y \in \W^{c}(x)$. Since the center foliation is compact with trivial holonomy we conclude that $N$ is a topological manifold. Furthermore since $\W^{c}$ is a $C^{r}$ foliation of $M$ we actually conclude that $N$ is a $C^{r}$ manifold and $f$ descends to a $C^{r+\alpha}$ Anosov diffeomorphism $g: N \rightarrow N$. 
\color{black}
The invariance of the conformal structure $\tau$ from Lemma \ref{lemma: center conformal} under center holonomy implies that $\tau$ descends to a conformal structure $\bar{\tau}$ on the unstable bundle of $g$ acting on $N$. This shows that $g$ is uniformly $u$-quasiconformal. An analogous argument using the invariant conformal structure on $E^{s}$ shows that $g$ is also uniformly $s$-quasiconformal. Hence $g$ is a $C^{r+\alpha}$  uniformly quasiconformal Anosov diffeomorphism of $N$. This completes the proof for (2) of Theorem \ref{theorem: compact center}.\color{black}

We now assume that $f$ is $\infty$-bunched. Applying the results of the previous two paragraphs with $r = \infty$ we conclude that $\W^{cs}$, $\W^{cu}$ and $\W^{c}$ are $C^{\infty}$ foliations of $M$, $N$ is a $C^{\infty}$ manifold, and $g: N \rightarrow N$ is a $C^{\infty}$ volume-preserving uniformly quasiconformal Anosov diffeomorphism. By the classification theorem of Fang \cite{F04} $g$ is smoothly conjugate to a hyperbolic toral automorphism and the stable and unstable foliations $\W^{s,g}$ and $\W^{u,g}$ of $g$ are $C^{\infty}$. 
\end{proof}

\begin{proof}[Proof of Corollary \ref{theorem: isometric center}]We first prove (1). Since $f$  has one dimensional compact center foliation and each central leaf has trivial holonomy group, \color{black} by Lemma \ref{prop: 1dcenter} $f$ is $\infty$-bunched and there is a smooth norm $| \cdot |$ on $E^{c}$ such that $Df|_{E^{c}}$ is an isometry with respect to this norm. Hence the conclusions of part (3) of Theorem \ref{theorem: compact center} apply to $f$ so that the foliations $\W^{c}$, $\W^{cu}$ and $\W^{cs}$ \color{black} of $M$ are $C^{\infty}$, the quotient $\pi: M \rightarrow N$ of $M$ by the center foliation is a torus and there is a hyperbolic toral automorphism $g: N \rightarrow N$ such that $\pi \circ f = g \circ \pi$. Since there is a smooth norm on $E^{c}$ with respect to which $Df|_{E^{c}}$ acts by isometries we conclude that $f$ is an isometric extension of $g$. 

For (2), since $\dim(E^u(f))=\dim(E^s(f))=2$, by Theorem D of \cite{BoThesis} there is a finite cover of systems $(\tilde{f},\tilde{M})$ of $(f,M)$ such that each central leaf of $\tilde{f}$ has trivial holonomy group. We then apply Theorem \ref{theorem: compact center} to $(\tilde{f},\tilde{M})$ to obtain the result.\end{proof}

\begin{proof}[Proof of Theorem \ref{theorem: flow suspension}]
We will first prove the theorem for any $C^{1}$-small enough volume-preserving uniformly quasiconformal perturbation $f$ of $\psi_{1}$ under the assumption that $\psi_{t}$ has no periodic orbits of period $ \leq 2$. We will then show how to deduce the finite cover version from this. 

We claim that there is a $C^{1}$-open neighborhood $\mathcal{U}$ of $\psi_{1}$ in the space of smooth volume-preserving diffeomorphisms of $M$ such that if $f \in \mathcal{U}$ then $f$ satisfies assumption (B) of Section \ref{section: higher}. Since $\psi_{1}$ is partially hyperbolic, the center foliation $\W^{c,\psi_{1}}$ for $\psi_{1}$ is normally hyperbolic and every center leaf is fixed by $\psi_{1}$, by the work of Hirsch, Pugh, Shub \cite{HPS} we deduce that any $f$ which is $C^{1}$ close to $\psi_{1}$ is partially hyperbolic, dynamically coherent, and has the property that $f(\W^{c}(x)) = \W^{c}(x)$ for every $x \in M$. Furthermore the center bundle $E^{c}$ for $f$ is orientable with orientation preserved by $f$ because it is $C^{0}$ close to the orientable center bundle for $\psi_{1}$. Since $\psi_{t}$ has no periodic orbits of period $\leq 1$, $\psi_{1}$ has no fixed points and thus if the neighborhood $\mathcal{U}$ is chosen small enough then $f$ will have no fixed points as well.  Finally, consider for each $x \in M$ the flow line $U_{x}^{\psi}$ of $x$ in the center foliation $\W^{c,\psi_{1}}$ of $\psi_{1}$ given by $U_{x}^{\psi} = \bigcup_{t \in [-1,1]} \psi_{t}(x)$. Since $\psi_{t}$ has no periodic orbits of period $\leq 2$, $U_{x}^{\psi}$ is a subarc of $\W^{c,\psi_{1}}(x)$ which is not a circle. The subarc of $U_{x}$ of the center leaf $\W^{c}(x)$ for $f$ through $x$ constructed at the beginning of Section \ref{section: higher} is uniformly close to $U_{x}^{\psi}$ and thus if $f$ is $C^{1}$ close enough to $\psi_{1}$ then $U_{x}$ will be a subarc of $\W^{c}(x)$ instead of a circle. Thus $f$ does not wrap and so $f$ satisfies assumption (B) of Section \ref{section: higher}.

%observe that up to passing to a finite cover of $M$ assumption (2) of Section \ref{section: higher} holds for $C^{1}$-small perturbations of all the Anosov flows considered in the hypotheses of Theorem \ref{theorem: flow suspension}. The conclusions are invariant under passing to finite covers so it suffices to assume that $f$ itself satsifies assumption (2). 

By Lemma \ref{prop: 1dcenter} there is thus a $C^{\infty}$ norm $| \cdot |$ on $E^{c}$ with respect to which $Df|_{E^{c}}$ is an isometry and we also conclude that the invariant foliations $\W^{cs}$, $\W^{cu}$, and $\W^{c}$ for $f$ are $C^{\infty}$. Let $Z$ be the unique smooth, positively oriented vector field $Z: M \rightarrow E^{c}$ which satisfies $|Z(x)|_{x} = 1$ for every $x \in M$. Let $d_{x}$ be the Riemannian metric on the center leaf $\W^{c}(x)$ that is induced by the norm $|\cdot |$. Note that $d_{x} = d_{y}$ for each $y \in \W^{c}(x)$. Let $\varphi_{t}$ be the smooth flow generated by $Z$. Then $\varphi_{t}(x)$ is the endpoint of the unique geodesic in the metric $d_{x}$ of length $t$ that is tangent to $Z(x)$ at $x$. The flowlines of $\varphi_{t}$ are the leaves of $\mathcal{W}^{c}$. Since $Z(f(x)) = Z(x)$ we have $f \circ \varphi_{t} = \varphi_{t} \circ f$. 

Consider the local stable holonomy map $h^{s}_{xy}: \mathcal{W}^{c}_{loc}(x) \rightarrow \W^{c}_{loc}(y)$ for $y \in \W^{s}_{loc}(x)$. The $Df$-invariance of the norm $| \cdot |$ implies that this map is an isometry, in other words, $|(Dh^{s}_{xy})_{x}(v)|_{y} = |v|_{x}$ for each $v \in E^{c}_{x}$. Likewise the local unstable holonomy map $h^{u}_{xy}: \mathcal{W}^{c}_{loc}(x) \rightarrow \W^{c}_{loc}(y)$ for $y \in \W^{u}_{loc}(x)$ is an isometry in $|\cdot |$. Since $h^{s}$ and $h^{u}$ are isometries between center leaves in the $d_{x}$ metric, the flow $\varphi_{t}$ also preserves $\mathcal{W}^{s}$ and $\mathcal{W}^{u}$ leaves, i.e., $\varphi_{t}(\mathcal{W}^{s}(x)) = \mathcal{W}^{s}(\varphi_{t}(x))$ and $\varphi_{t}(\mathcal{W}^{u}(x)) = \mathcal{W}^{u}(\varphi_{t}(x))$.

We let $\xi: M \rightarrow (0,\infty)$ be the smooth function defined as follows: $\xi(x)$ is the unique minimal positive time $t \in (0,\infty)$ such that $\varphi_{t} = f(x)$. Since $f$ has no fixed points and the norm $|\cdot|$ is continuous, by compactness of $M$ there exists some $r, R > 0$ such that $0 < r \leq \xi(x) \leq R$ for all $x \in M$.  It is then easy to show, by combining this bound with the fact that $\varphi_{t}$ preserves the stable and unstable foliations of $f$, that $\varphi_{t}$ is an Anosov flow. 

We claim that $\varphi_{t}$ is topologically transitive. Since $\varphi_{\xi(x)} = f$ is volume-preserving, the nonwandering set of $\varphi_{t}$ is all of $M$. By the spectral decomposition theorem for flows \cite{S67} we can decompose $M$ into connected components invariant under $\varphi_{t}$ on which $\varphi_{t}$ is topologically transitive; since $M$ is connected we conclude that $\varphi_{t}$ is actually topologically transitive on $M$. 

Choose a point $x \in M$ such that $\{\varphi_{t}(x)\}_{t \in \R} = \mathcal{W}^{c}(x)$ is dense in $M$. We claim that $\xi(\varphi_{t}(x)) = \xi(x)$ for all $t \in \R$. Indeed we have the sequence of equalities,  
\begin{align*}
\varphi_{\xi(x)}(\varphi_{t}(x)) &= \varphi_{t}(\varphi_{\xi(x)}(x)) \\
&= \varphi_{t}(f(x)) \\
&= f(\varphi_{t}(x)) \\
&=  \varphi_{\xi(\varphi_{t}(x))}(\varphi_{t}(x))
\end{align*}
Since $\varphi_{t}$ is injective on $\mathcal{W}^{c}(x)$ (as this leaf is dense and therefore cannot be closed) we conclude that $\xi(\varphi_{t}(x)) = \xi(x)$. Thus $\xi$ is constant on the orbit of $x$;  since $\xi$ is continuous and the orbit of $x$ is dense we conclude that $\xi$ is constant.

Thus there is a constant $c > 0$ such that $\varphi_{c} = f$. By replacing $\varphi_{t}$ with $\varphi_{c^{-1}t}$ we obtain a smooth Anosov flow $\varphi_{t}$with $\varphi_{1} = f$. The fact that $\varphi_{t}$ is a uniformly quasiconformal Anosov flow follows from the partial hyperbolicity and uniform quasiconformality estimates for its time-1 map $f$. It only remains to show that $\varphi_{t}$ preserves a measure smoothly equivalent to volume. 

For each $t \in \R$, $\varphi_{t}: M \rightarrow M$ is a smooth diffeomorphism and thus the measures $(\varphi_{t})_{*}m$ and $m$ are smoothly equivalent with $C^{\infty}$ Radon-Nikodym derivative $\frac{d(\varphi_{t})_{*}m}{dm}:=J_{t}$. Since $\varphi_{1} = f$ we have $J_{1} \equiv 1$. For every $x \in M$ we clearly have $J_{t+s}(x) = J_{t}(\varphi_{s}(x)) \cdot J_{s}(x)$ from the property that $\varphi_{t+s} = \varphi_{t} \circ \varphi_{s}$. 

We can thus apply the following criterion for a topologically transitive Anosov flow $\varphi_{t}$ to preserve a measure smoothly equivalent to volume: $\varphi_{t}$ preserves a measure smoothly equivalent to volume if and only if for every periodic point $p$ of $\varphi_{t}$ of period $\ell(p)$ we have $J_{\ell(p)}(p) = 1$ \cite{LS72}.

Suppose that this does not hold. Then without loss of generality we can assume that there is a periodic point $p$ for which $J_{\ell(p)}(p)> 1$. For this point we then have 
$\lim_{n\to \infty}
J_{n\ell(p)}(p) = \lim_{n\to \infty}(J_{\ell(p)}(p))^{n}= \infty
$
\color{black}. On the other hand, let $\lfloor n\ell(p)\rfloor$ denote the greatest integer smaller than $n \ell(p)$ and let $K := \sup_{0 \leq t \leq 1} \sup_{x \in M} J_{t}(x)$. Then since $J_{1} \equiv 1$ we have 
\[
J_{n\ell(p)}(p)= J_{n\ell(p)-\lfloor n\ell(p)\rfloor}(p) \leq K < \infty,
\]
 for each integer $n> 0$. Thus we obtain a contradiction so that $J_{\ell(p)}(p) = 1$ for every periodic point $p$ and thus $\varphi_{t}$ preserves a measure smoothly equivalent to volume on $M$.

Finally suppose only that there is a finite cover $p: \hat{M} \rightarrow M$ of $M$ such that the lift $\hat{\psi}_{t}$ of $\psi_{t}$ to $\hat{M}$ has no periodic orbits of period $\leq 2$. Let $\Gamma$ be the automorphism group of this cover, i.e., $M = \hat{M}/\Gamma$. Let $\hat{\mathcal{U}}$ be the $C^{1}$-open neighborhood of $\hat{\psi}_{1}$ given by Theorem \ref{theorem: flow suspension} applied to $\hat{\psi}_{t}$. We let $\mathcal{U}$ be the $C^{1}$-open neighborhood of $\psi_{1}$ consisting of all smooth volume-preserving diffeomorphisms $f$ whose lift $\hat{f}: \hat{M} \rightarrow \hat{M}$ (where $\hat{f}$ is the lift acting trivially on $\Gamma$) lies in $\hat{\mathcal{U}}$. Note that we use here the fact that $f$ is $C^0$ close to $\psi_{1}$ and therefore homotopic to the identity on $M$.  

For $f \in \mathcal{U}$ we apply Proposition \ref{prop: 1dcenter} (along with all of the previous work in the paper) to $\hat{f}$ and thus obtain an $\hat{f}$-invariant $C^{\infty}$ norm $|\cdot|'$ on the lift of the center bundle $\hat{E}^{c}$ to $\hat{M}$. For $v \in \hat{E}^{c}_{x}$ we define a new norm $|\cdot |$ by 
\[
|v|_{x} = \frac{1}{|\Gamma|}\sum_{\gamma  \in \Gamma} |D\gamma(v)|'_{x}.
\] 
Since $\gamma \circ \hat{f} = \hat{f} \circ \gamma$ for all $\gamma \in \Gamma$ the norm $|\cdot|$ is also $\hat{f}$-invariant but is now $\Gamma$-invariant as well. We construct a $C^{\infty}$ volume-preserving, uniformly quasiconformal Anosov flow $\hat{\varphi}_{t}$ with $\hat{\varphi}_{1} = \hat{f}$ from the norm $|\cdot|$ as above and then note that since it was constructed from a $\Gamma$-invariant norm this flow is also $\Gamma$-invariant and thus descends to a flow $\varphi_{t}$ on $M$ with $\varphi_{1} = f$ and all of the same properties as $\hat{\varphi}_{t}$. This completes the proof. 
\end{proof}

%\begin{lemma}
%If $f$ is accessible and satisfying condition of Proposition \ref{proposition: flow cne}, moreover we assume $\tilde{f}$ some lift of $f$ admits global $su-$holonomy (for definition see chapter 5 of  \cite{avw}) then for any $r\geq 1$, $f$ is $r-$bunching.
%\end{lemma}
%\begin{rema}If f is $C^1-$close to the time-1 map of geodesic flow of a hyperbolic manifold $S$, then the lift of $f$ on the unit tangent bundle of the universal cover of $S$ admits global $su-$holonomy?
%\end{rema}	
%\begin{pf}
%By lemma \ref{wcs ac}, $\W^c$ are absolutely continuous with respect to the volume. Since $f$ is accessible, we can use the argument in [AVW] to get that $Df$ preserve a continuous vector field $X$ tangent to $E^c$, which implies $f$ is $r-$bunching for any $r\geq 1$.
%\end{pf}

%\begin{lemma}
%If $f$ is accessible and the center leaves are $1-$dimensional, compact with a uniformly bound on there diameters. Then for any $r\geq 1$, $f$ is $r-$bunching.
%\end{lemma}
%\begin{pf}
%The proof is basically the same as last lemma, since $f$ itself admits global $su-$holonomy?
%\end{pf}
%\color{black}

\section{Proof of Theorem \ref{theorem: hyperbolic perturbation}}\label{section:theorem4}
Let $X$ be a closed Riemannian manifold of constant negative curvature with $\dim X \geq 3$ and let $T^{1}X$ be the unit tangent bundle of $X$. We let $\pi: T^{1}X \rightarrow X$ denote the standard projection of a unit tangent vector to its basepoint in $X$. We let $\psi_{t}$ denote the geodesic flow on $T^{1}X$ and consider a smooth, volume-preserving perturbation $f$ of the time-1 map $\psi_{1}$. We will establish in this section that the equalities $\lambda_{+}^{u} = \lambda_{-}^{u}$ and $\lambda_{+}^{s} = \lambda_{-}^{s}$ imply that $Df|_{E^{u}}$ and $Df|_{E^{s}}$ respectively are uniformly quasiconformal for small enough volume-preserving perturbations of $\psi_{1}$. We will prove this implication for the unstable bundle $E^{u}$; the proof for $E^{s}$ will be analogous. By Theorem \ref{theorem: flow suspension} and the smooth orbit equivalence classification result of Fang \cite{F14} this suffices to complete the proof of Theorem \ref{theorem: hyperbolic perturbation} from the Introduction. 

We first need to recall some properties of the frame flow associated to closed Riemannian manifolds of constant negative curvature. Let $X^{(2)}$ be the \emph{2-frame bundle} over $X$ which has fiber over each $p \in X$ given by
\[
X^{(2)}_{p} = \{(v,w) \in T_{p}^{1}X: \text{$v$ is orthogonal to $w$}\}.
\] 
We let $\psi_{t}^{(2)}$ be the \emph{2-frame flow} on $X^{(2)}$ obtained by applying the geodesic flow $\psi_{t}$ to the first vector $v \in T_{p}^{1}X$ and then taking the image of $w$ under parallel transport along the geodesic $\gamma(s) = \pi(\psi_{s}(v))$, $s \in [0,t]$, on $X$. 

We let $E^{u,\psi}$ be the unstable bundle of the geodesic flow $\psi_{t}$ on $T^{1}X$ and we let $SE^{u,\psi}$ be the unit sphere inside of $E^{u,\psi}$, where we equip $E^{u,\psi}$ with the Riemannian norm $\| \cdot \|$ coming from its realization as the tangent spaces of unstable horospheres in the universal cover of $X$. We have a smooth identification $SE^{u,\psi} \rightarrow X^{(2)}$ coming from this realization by identifying a unit vector $v \in T^{1}_{p}X$ together with a unit vector $w \in SE^{u,\psi}_{v}$ to the orthonormal 2-frame $(v,w) \in X^{(2)}_{p}$ obtained from identifying $w$ with its image in the tangent space of the unstable horosphere through $p$ which is orthogonal to $v$. Since $X$ has constant negative curvature the geodesic flow is conformal on unstable horospheres and therefore under this identification the 2-frame flow $\psi_{t}^{(2)}$ corresponds to the renormalized derivative action $w \rightarrow \frac{D\psi_{t}(w)}{\|D\psi_{t}(w)\|}$ on $SE^{u,\psi}$. For a more detailed version of this discussion as well as the discussion in the paragraphs below we refer to \cite{BK84}, \cite{BP03}. 

We consider the stable and unstable holonomies $H^{s,\psi}$ and $H^{u,\psi}$ of $\psi_{t}$ on $E^{u,\psi}$ and their renormalized versions $SH^{s,\psi}(\cdot) = \frac{H^{s,\psi}(\cdot) }{\|H^{s,\psi}(\cdot) \|}$, $SH^{u,\psi}(\cdot)  = \frac{H^{u,\psi}(\cdot) }{\|H^{u,\psi}(\cdot) \|}$ which give isometric identifications $SE^{u,\psi}_{v} \rightarrow SE^{u,\psi}_{v'}, SE^{u,\psi}_{w} \rightarrow SE^{u,\psi}_{w'}$ for $v' \in \W^{s,\psi}(v)$ and $w' \in \W^{u,\psi}(w)$ respectively, where $\W^{s,\psi}$ and $\W^{u,\psi}$ denote the stable and unstable foliations of $\psi$ respectively.

An \emph{$su$-loop} based at $v \in T^{1}X$ is an $su$-path for $\psi_{t}$ which starts and ends at $v$. Based on the discussion of the previous paragraphs, given an $su$-loop $\gamma$ for $\psi_{t}$ based at a point $v \in T^{1}X$ we can associate an isometry $T_{\psi}(\gamma): SE^{u,\psi}_{v} \rightarrow SE^{u,\psi}_{v}$ obtained by composing the renormalized stable and unstable holonomy maps $SH^{s,\psi}_{v_{i}v_{i+1}}: SE^{u,\psi}_{v_{i}} \rightarrow SE^{u,\psi}_{v_{i+1}}$ and $SH^{u,\psi}_{v_{j}v_{j+1}}: SE^{u,\psi}_{v_{j}} \rightarrow SE^{u,\psi}_{v_{j+1}}$ along this loop, where $\gamma_{v_{i}v_{i+1}} \subset \W^{s,\psi}(v_{i})$ and $\gamma_{v_{j}v_{j+1}} \subset \W^{u,\psi}(v_{i})$. Thus, identifying $SE^{u,\psi}_{v}$ with the unit sphere $S^{n-1}$ in $\R^{n}$ for $n := \dim X-1$, $T_{\psi}(\gamma)$ gives us an element of the special orthogonal group $SO(n)$. 

The key observation due to Brin and Karcher \cite{BK84} is that for a closed constant negative curvature manifold $X$ and any $v \in T^{1}X$ there are finitely many $su$-loops $\gamma_{1},\dots,\gamma_{k}$ such that $T_{\psi}(\gamma_{1}), \dots, T_{\psi}(\gamma_{k})$ generate $SO(n)$ as a Lie group when we identify $E^{u}_{v}$. Moreover the number $k$ of loops used and the total lengths of these loops may both be taken to be bounded independently of the point $v$. As a consequence we have the following proposition, 

\begin{prop}\label{prop:quantdensity}
For any $\delta > 0$ there is a constant $L > 0$ and an integer $\ell > 0$ such that given any $v \in T^{1}X$ there is a finite collection $\gamma_{1},\dots,\gamma_{\ell}$ of $su$-loops based at $v$ of total length at most $L$ for which the collection of points $\{T_{\psi}(\gamma_{i})(w)\}_{i = 1}^{\ell}$ is $\delta$-dense in $SE^{u,\psi}_{v}$ for any $w \in SE^{u,\psi}_{v}$. 
\end{prop}

\begin{proof}
Fix a $\frac{\delta}{2}$-dense collection $\{w_{j}\}_{j=1}^{k}$ of points in $SE^{u,\psi}_{v}$. Since there are finitely many $su$-loops based at $v$ whose associated isometries generate $SO(n)$ as a Lie group and $SO(n)$ acts transitively on $SE^{u}_{v}$, there is a finite collection $\gamma_{1},\dots,\gamma_{\ell}$ of $su$-loops based at $v$ for which each of the sets $\{T_{\psi}(\gamma_{i})(w_{j})\}_{i = 1}^{\ell}$ for $1 \leq j \leq k$ is $\frac{\delta}{2}$-dense in $SE^{u,\psi}_{v}$. 

Now let $w$ be any point in $SE^{u,\psi}_{v}$. Then there is some $w_{j}$ such that $\|w-w_{j}\| < \frac{\delta}{2}$. Since each $T_{\psi}(\gamma_{i})$ is an isometry we then also have $\|T_{\psi}(\gamma_{i})(w_{j})-T_{\psi}(\gamma_{i})(w_{j})\| < \frac{\delta}{2}$ for each $1 \leq i \leq \ell$. This implies that $\{T_{\psi}(\gamma_{i})(w)\}_{i = 1}^{\ell}$ is a $\delta$-dense subset of $SE^{u,\psi}_{v}$. 
\end{proof}

Let $f$ be a $C^{1}$-small perturbation of the time-1 map $\psi_{1}$. If this perturbation is small enough then the linear cocycle $Df|_{E^{u}}$ is \emph{fiber bunched} and consequently the conclusions of Proposition \ref{proposition: holonomies} apply to $Df|_{E^{u}}$, see \cite[Proposition 4.2]{KS13}. Thus the linear cocycle $Df|_{E^{u}}$ admits linear stable and unstable holonomies $H^{s}$ and $H^{u}$. For $v \in T^{1}X$ we define $\mathbb{P}E^{u}_{v}$ to be the projective space of $E^{u}_{v}$ and we define $\mathbb{P}H^{s}$ and $\mathbb{P}H^{u}$ to be the induced maps of $H^{s}$ and $H^{u}$ on the projective spaces $\mathbb{P}E^{u}_{v} \rightarrow \mathbb{P}E^{u}_{v'}$ for $v' \in \W^{s}(v)$ and $v' \in \W^{u}(v)$ respectively, where now $\W^{s}$ and $\W^{u}$ denote the stable and unstable foliations of $f$. We let $\mathbb{P}Df_{v}: \mathbb{P}E^{u}_{v} \rightarrow \mathbb{P}E^{u}_{f(v)}$ be the induced map from $Df_{v}$.

We obtain below a version of Proposition \ref{prop:quantdensity} which also applies to the perturbation $f$ provided that this perturbation is small enough. We endow $\mathbb{P}E^{u}_{v}$ with the Riemannian metric induced from the Riemannian metric on $E^{u}$ which is in turn induced from the metric on $T^{1}X$. Given an $su$-loop $\gamma$ for $f$ based at $v \in T^{1}X$ we associate the map $T(\gamma):  \mathbb{P}E^{u}_{v} \rightarrow \mathbb{P}E^{u}_{v}$ obtained by composing the projectivized stable and unstable holonomies $\mathbb{P}H^{s}$ and $\mathbb{P}H^{u}$ along the segments of this loop which lie in the stable and unstable leaves of $f$ respectively. Unlike the case of $T_{\psi}$ above, $T(\gamma)$ is not necessarily an isometry of $\mathbb{P}E^{u}_{v}$. 

We will need the following proposition that follows from results of Katok and Kononenko, 

\begin{prop}[\cite{KK96}]\label{legs}
Let $\psi_{t}: M \rightarrow M$ be a contact Anosov flow on a closed Riemannian manifold. Then there is a $C^{2}$-open neighborhood $\mathcal{V}$ of $\psi_{1}$ in the space of $C^{2}$ diffeomorphisms of $M$ and an integer $J > 0$ such that for every $\varepsilon > 0$ and every $f \in \mathcal{V}$ there exists an $\eta  > 0$ such that for every $p, q \in M$ with $d(p, q) < \eta$, there exists a $J$-legged $su$-path from $p$ to $q$ of length less than $\varepsilon$.
\end{prop}

Recall that for each pair of nearby points $x,y \in T^{1}X$ we let $I_{xy}: E^{u}_{x} \rightarrow E^{u}_{y}$ be a linear identification which is H\"older close to the identity. This induces an identification $\mathbb{P}I_{xy}: \mathbb{P}E^{u}_{x} \rightarrow \mathbb{P}E^{u}_{y}$ that is H\"older close to the identity in $x$ and $y$. 

\begin{lemma}\label{lemma:denseperturb}
Given any $\delta > 0$ there is a $C^{2}$-open neighborhood $\mathcal{U}$ of $\psi_{1}$ such that if $f \in \mathcal{U}$ then for any $v \in T^{1}X$ there is a finite collection $\gamma_{1},\dots,\gamma_{\ell}$ of $su$-loops for $f$ based at $v$ such that the collection of points $\{T(\gamma_{i})(w)\}_{i = 1}^{\ell}$ is $\delta$-dense in $\mathbb{P}E^{u}_{v}$ for any $w \in \mathbb{P}E^{u}_{v}$.
\end{lemma} 

\begin{proof}
Let $\delta > 0$ be given. We first apply Proposition \ref{prop:quantdensity} to $\psi_{t}$ to obtain a constant $L > 0$ and integer $\ell > 0$ such that for any $v \in T^{1}X$ there is a collection of $su$-loops $\sigma_{1},\dots,\sigma_{\ell}$ based at $v$ of total length at most $L$ such that $\{T_{\psi}(\sigma_{i})(w)\}_{i = 1}^{\ell}$ is $\frac{\delta}{3}$-dense in $SE^{u,\psi}_{v}$ for any $w \in SE^{u,\psi}_{v}$. 

We apply Proposition \ref{legs} for a small $\varepsilon > 0$ to be determined. Given the $\eta > 0$ obtained from Proposition \ref{legs} for this $\varepsilon$ we claim that we can find a $C^{2}$-open neighborhood $\mathcal{U}'$ of $\psi_{1}$ such that for each $v \in T^{1}X$ there are points $v_{1},\dots,v_{\ell}$ satisfying $d(v,v_{i}) < \eta$ and for each $1 \leq i \leq \ell$ there is an $su$-path $\beta_{i}$ for $f$ from $v$ to $v_{i}$ such that the collection $\{\mathbb{P}I_{v_{i}v} \circ T(\beta_{i})(w)\}_{i = 1}^{\ell}$ is $\frac{2\delta}{3}$-dense in $\mathbb{P}E^{u}_{v}$ for any $w \in \mathbb{P}E^{u}_{v}$. This follows from the facts that the stable and unstable foliations $\mathcal{W}^{s}$ and $\mathcal{W}^{u}$ depend continuously on $f$ in the $C^{2}$ topology and the stable and unstable holonomies $H^{s}$ and $H^{u}$ of $Df|_{E^{u}}$ also depend continuously on $f$ in the $C^{2}$ topology\cite{ASV}. Hence we obtain this statement by considering $su$-paths $\beta_{1},\dots,\beta_{\ell}$ for $f$ which are close enough to the $su$-loops $\sigma_{1},\dots,\sigma_{\ell}$ for $\psi_{t}$; we can make these paths as close as desired to the loops for $\psi_{t}$ by making the neighborhood $\mathcal{U}$ small enough. 

For each $1 \leq i \leq \ell$ we let $\gamma_{i}$ be the $su$-loop based at $v$ for $f$ obtained by concatenating $\beta_{i}$ with the $J$-legged $su$-path of length less than $\varepsilon$ connecting $v_{i}$ to $v$ given by Proposition \ref{legs}. Since the number of legs $J$ is fixed and both $H^{u}_{xy}$ and $H^{s}_{xy}$ converge uniformly to the identity as $y$ converges to $x$ for $x,y \in T^{1}X$ we conclude that if $\varepsilon$ is small enough (independent of the choice of $v \in T^{1}X$) then the collection of points $\{T(\gamma_{i})(w)\}_{i = 1}^{\ell}$ is $\delta$-dense in $\mathbb{P}E^{u}_{v}$ for any $w \in \mathbb{P}E^{u}_{v}$.
\end{proof}

The use of Lemma \ref{lemma:denseperturb} is the reason that we lose $C^{1}$-openness of the neighborhood $\mathcal{U}$ in Theorem \ref{theorem: hyperbolic perturbation}. 

It is easy to see that there is a $\delta_{0} > 0$ with the property that if $V_{1}$ and $V_{2}$ are any two proper linear subspaces of $\R^{n}$ then the union $\mathbb{P}V_{1} \cup \mathbb{P}V_{2}$ of their projectivizations in $\R \mathbb{P}^{n-1}$ is \emph{not} $\delta_{0}$-dense. Thus it follows that there is a $\delta > 0$ and a $C^{1}$-open neighborhood $\mathcal{U}'$ of $\psi_{1}$ such that for any $f \in \mathcal{U}'$, any $v \in T^{1}X$, and any pair of proper linear subspaces $V_{1}$ and $V_{2}$ in $E^{u}_{v}$, the union  $\mathbb{P}V_{1} \cup \mathbb{P}V_{2}$ is not $\delta$-dense in $\mathbb{P}E^{u}_{v}$. We apply Lemma \ref{lemma:denseperturb} with this $\delta$ and let $f \in \mathcal{U} \subset \mathcal{U}'$ be a smooth volume-preserving diffeomorphism in the resulting open neighborhood with the property that $\lambda_{+}^{u} = \lambda_{-}^{u}$. 

We will now show that $Df|_{E^u}$ is uniformly quasiconformal to complete the proof of Theorem \ref{theorem: hyperbolic perturbation}. Since $\psi_{1}$ is a stably accessible partially hyperbolic diffeomorphism (this well-known fact can be derived as a consequence of Proposition \ref{legs}) we may assume that $f$ is also an accessible partially hyperbolic diffeomorphism. Since the neighborhood $\mathcal{U}$ is chosen small enough that $Df|_{E^{u}}$ satisfies the fiber bunching condition that guarantees the existence of the stable and unstable holonomies $H^{s}$ and $H^{u}$ we conclude by the work of Avila, Santamaria and Viana \cite{ASV} that the equality  $\lambda_{+}^{u} = \lambda_{-}^{u}$ implies that there is a $\mathbb{P}Df$-invariant probability  measure $\mu$ on $\mathbb{P}E^{u}$ projecting down to the invariant volume $m$ for $f$ on $T^{1}X$ and which has a disintegration $\{\mu_{v}\}_{v \in T^{1}X}$ into probability measures $\mu_{v}$ on the projective fibers $\mathbb{P}E^{u}_{v}$ which depend continuously on the basepoint $v$. Furthermore this disintegration is invariant under the projectived stable and unstable holonomy, that is to say, if $v' \in \W^{s}(v)$ then $(\mathbb{P}H^{s}_{vv'})_{*}\mu_{v} = \mu_{v'}$ and a similar equation holds for $\mathbb{P}H^{u}$. 

Suppose that $Df|_{E^{u}}$ is not uniformly quasiconformal. Then there is a point $v \in T^{1}X$, unit vectors $w_{1},w_{2} \in E^{u}_{v}$, and a sequence $n_{k} \rightarrow \infty$ such that $\frac{\|Df^{n_{k}}(w_{1})\|}{\|Df^{n_{k}}(w_{2})\|} \rightarrow \infty$ as $n_{k} \rightarrow \infty$. By passing to a further subsequence and using the compactness of $T^{1}X$ we can assume that there is some $z \in T^{1}X$ such that $f^{n_{k}}(v) \rightarrow z$ as $n_{k} \rightarrow \infty$. Since $f$ is accessible we can find an $su$-path $\sigma$ connecting $z$ to $v$. For $n_{k}$ large enough we let $\gamma_{n_{k}}$ be a $J$-legged $su$-path connecting $f^{n_{k}}(v)$ to $z$ of length at most 1, where $J$ is given by Proposition \ref{legs}. We then let $T(\gamma_{n_{k}}): \mathbb{P}E^{u}_{f^{n_{k}}(v)} \rightarrow \mathbb{P}E^{u}_{z}$ be the map obtained by composing the $s$- and $u$-holonomies along $\gamma_{n_{k}}$ from $f^{n_{k}}(v)$ to $z$.

Let 
\[
A_{n_{k}} = T(\sigma) \circ T(\gamma_{n_{k}}) \circ \mathbb{P}Df_{v}^{n_{k}}: \mathbb{P}E^{u}_{v} \rightarrow \mathbb{P}E^{u}_{v}.
\] 
The holonomy invariance of the disintegration of the $\mathbb{P}Df$-invariant measure $\mu$ implies that $(A_{n_{k}})_{*}\mu_{v} = \mu_{v}$ for each $n_{k}$. Choose a linear identification of $\mathbb{P}E^{u}_{v}$ with the real projective space $\mathbb{R}\mathbb{P}^{n-1}$. Then $A_{n_{k}}$ gives an element of the projective linear group $PSL(n,\R)$ for each $n_{k}$. Since the transformations $T(\gamma_{n_{k}})$ have uniformly bounded norm together with their inverses, and since there exist unit vectors $w_{1},w_{2} \in E^{u}_{v}$ such that $\frac{\|Df^{n_{k}}(w_{1})\|}{\|Df^{n_{k}}(w_{2})\|} \rightarrow \infty$, we conclude that the sequence of transformations $\{A_{n_{k}}\}$ is not contained in any compact subset of $PSL(n,\R)$. Hence, after passing to a further subsequence if necessary,there is a quasi-projective transformation $Q$ of  $\mathbb{R}\mathbb{P}^{n-1}$ such that $A_{n_{k}}$ converges to $Q$ on the complement of a proper linear subspace $V$ of $\mathbb{R}\mathbb{P}^{n-1}$ (see \cite{GS89}). Furthermore the image of $Q$ is a proper linear subspace $L$ of $\mathbb{R}\mathbb{P}^{n-1}$.

Thus there is a proper linear subspace $V$ of $\mathbb{P}E^{u}_{v}$ such that on the complement of $V$, $A_{n_{k}}$ converges pointwise to a continuous map which has image contained inside of a proper subspace $L$ of $\mathbb{P}E^{u}_{v}$. Since $(A_{n_{k}})_{*}\mu_{v} = \mu_{v}$ for every $n_{k}$, this shows that $\mu_{v}$ is supported on the union $V \cup L$ of two proper subspaces of $\mathbb{P}E^{u}_{v}$. Consider any point $w \in \text{supp}(\mu_{v})$. By Lemma \ref{lemma:denseperturb} there is a collection of $su$-loops $\gamma_{1},\dots,\gamma_{\ell}$ based at $v$ such that the collection of points $\{T(\gamma_{i})(w)\}_{i = 1}^{\ell}$ is $\delta$-dense in $\mathbb{P}E^{u}_{v}$. But by the holonomy invariance of the disintegration of $\mu$, if $\gamma$ is an $su$-loop based at $v$ then $T(\gamma)(w) \in \text{supp}(\mu_{v}) \subset V \cup L$. This proves that the union $V \cup L$ of two proper subspaces of $\mathbb{P}E^{u}_{v}$ is $\delta$-dense in $\mathbb{P}E^{u}_{v}$, which contradicts our choice of $\delta$. Thus $Df|_{E^{u}}$ is uniformly quasiconformal.

\bibliographystyle{plain}
\bibliography{CPH}

\end{document}